\documentclass[a4paper,reqno,10pt]{amsart}
\usepackage{amsfonts} 
\usepackage{amssymb}
\usepackage{amsthm}
\usepackage{amsmath} 
\usepackage{mathrsfs}
\usepackage{dsfont}
\usepackage{stmaryrd}
\usepackage[english]{babel} 
\usepackage[left=3.5cm,right=3.5cm,top=3.5cm,bottom=3cm,headsep=0.7cm]{geometry}

\usepackage{pgf,tikz}
\usetikzlibrary{arrows}
\usetikzlibrary{intersections}
\usepackage{ifthen}

\usepackage[scriptsize,bf]{caption}
\usepackage{floatrow} 

\usepackage{enumerate}
\usepackage{enumitem}

\usepackage{hyperref}

\theoremstyle{plain}
\begingroup
\theoremstyle{plain}
\newtheorem{theorem}{Theorem}[section]

\newtheorem{proposition}[theorem]{Proposition}
\newtheorem{lemma}[theorem]{Lemma}
\theoremstyle{definition}
\newtheorem{definition}[theorem]{Definition}
\theoremstyle{remark}
\newtheorem{remark}[theorem]{Remark}

\endgroup

\theoremstyle{definition}
\theoremstyle{remark}

\numberwithin{equation}{section}

\newcommand{\RR}{\mathbb{R}}
\newcommand{\NN}{\mathbb{N}}
\newcommand{\ZZ}{\mathbb{Z}}

\renewcommand{\SS}{\mathbb{S}}
\newcommand{\A}{\mathcal{A}}
\newcommand{\I}{\mathcal{I}}
\newcommand{\D}{\mathrm{D}}
\renewcommand{\L}{\mathcal{L}}
\newcommand{\V}{\mathcal{V}}
\newcommand{\M}{\mathcal{M}}
\renewcommand{\H}{\mathcal{H}}

\mathchardef\emptyset="001F
\renewcommand{\d}[1]{\, \mathrm{d} #1}
\newcommand{\de}{\partial}
\newcommand{\e}{\varepsilon}
\renewcommand{\tilde}{\widetilde}
\newcommand{\x}{{\times}}
\newcommand{\ol}{\overline}
\newcommand{\ul}{\underline}
\newcommand{\sm}{\setminus}
\newcommand{\dist}{{\rm dist}}

\renewcommand{\hat}{\widehat}
\newcommand{\weak}{\rightharpoonup}

\newcommand{\mres}{\mathbin{\vrule height 1.6ex depth 0pt width 0.13ex\vrule height 0.13ex depth 0pt width 1.3ex}}
\newcommand{\integral}[3]{\int \limits_{#1} \! #2 #3}
\newcommand{\PC}{\mathcal{PC}}
\newcommand{\theth}{\theta^{\mathrm{hor}}}
\newcommand{\thetv}{\theta^{\mathrm{ver}}}
\newcommand{\Hh}{H^{\mathrm{hor}}}
\newcommand{\Hv}{H^{\mathrm{ver}}}
\newcommand{\Lloc}{L^1_{\mathrm{loc}}}
\newcommand{\T}{T}
\renewcommand{\ln}{\lambda_n}
\newcommand{\dn}{\delta_n}
\newcommand{\en}{\varepsilon_n}
\DeclareMathOperator{\curl}{curl}
\newcommand{\Dom}{\mathrm{Dom}}
\newcommand{\vc}[2]{\big( \begin{smallmatrix}
    #1 \\ #2 
\end{smallmatrix} \big) }

\makeatletter
\newcommand*{\bigcdot}{}
\DeclareRobustCommand*{\bigcdot}{%
  \mathbin{\mathpalette\bigcdot@{}}%
}
\newcommand*{\bigcdot@scalefactor}{.5}
\newcommand*{\bigcdot@widthfactor}{1.15}
\newcommand*{\bigcdot@}[2]{%
  \sbox0{$#1\vcenter{}$}
  \sbox2{$#1\cdot\m@th$}%
  \hbox to \bigcdot@widthfactor\wd2{%
    \hfil
    \raise\ht0\hbox{%
      \scalebox{\bigcdot@scalefactor}{%
        \lower\ht0\hbox{$#1\bullet\m@th$}%
      }%
    }%
    \hfil
  }%
}
\makeatother

\newcommand{\subcc}{\subset \subset}
\newcommand{\supcc}{\supset \supset}

\author{Marco Cicalese}
\address[Marco Cicalese]{Technische Universit\"at M\"unchen, Germany}
\email{cicalese@ma.tum.de}

\author{Marwin Forster}
\address[Marwin Forster]{Technische Universit\"at M\"unchen, Germany}
\email{marwin.forster@ma.tum.de}

\author{Gianluca Orlando}
\address[Gianluca Orlando]{Technische Universit\"at M\"unchen, Germany}
\email{orlando@ma.tum.de}

\title[Variational analysis of a two-dimensional frustrated spin system]{Variational analysis of a two-dimensional frustrated spin system: emergence and rigidity of chirality transitions}

\begin{document}

\begin{abstract}
    We study the discrete-to-continuum variational limit of the $J_{1}$-$J_{3}$ spin model on the square lattice in the vicinity of the helimagnet/ferromagnet transition point as the lattice spacing vanishes. Carrying out the $\Gamma$-convergence analysis of proper scalings of the energy, we prove the emergence and characterize the geometric rigidity of the chirality phase transitions.
    \end{abstract}

\maketitle


\vspace{1em}


\setcounter{tocdepth}{1}
\tableofcontents

\section{Introduction}
 
Complex geometric structures may arise in two-dimensional magnetic compounds as a result of interactions of exchange or of anisotropic nature. Their emergence has attracted the attention of the statistical mechanics community and has been the object of many studies in the late years (see~\cite{Diep} and the references therein for some recent overviews on this topic). The modelling and the study of appropriate lattice energies that can be characterized in terms of a minimal number of parameters and whose ground states display such interesting structures is strictly related to the so-called frustration mechanisms. The latter refers to the presence of conflicting interatomic forces in the model that result, for instance, from the competition of short-range ferromagnetic (F) and antiferromagnetic (AF) interactions as in the model we are going to consider here. This is known as the $J_{1}$-$J_{3}$ F-AF classical spin model on the square lattice (see \cite{rastelli1979non}). To each configuration of two-dimensional unitary spins on the square lattice, namely $u \colon  \ZZ^2 \to \SS^1$, we associate the energy
\begin{equation*}
    E(u) := - J_1 \sum_{|\sigma - \sigma'| = 1}   u^\sigma \cdot u^{\sigma'}   + J_3\sum_{|\sigma - \sigma''| = 2}   u^\sigma \cdot u^{\sigma''}   ,
\end{equation*}
where $J_1$ and $J_3$ are positive constants (the interaction parameters of the model) and for every lattice point $\sigma \in \ZZ^2$ we denote by~$u^\sigma$ the value of the spin variable~$u$ at~$\sigma$. The first term in the energy is ferromagnetic as it favors aligned nearest-neighboring spins, whereas the second one is antiferromagnetic as it favors antipodal third-neighboring spins. In the case where $J_3 = 0$ the energy describes the so-called $XY$ model, whose variational analysis has been carried out in~\cite{AliCic, BadCicDLPon, CicOrlRuf} also in connection with the theory of dislocations~\cite{Pon, AliCicPon, AliDLGarPon}. 
When $J_3 > 0$, the ferromagnetic and the antiferromagnetic terms in $E$ compete. For an appropriate choice of the interaction parameters, the competition gives rise to frustration and to ground states that can present the complex structures mentioned above, here in the form of helices of possibly different chiralities (for recent experimental evidences see \cite{schoenherr2018topological, uchida2006real}).

To study the behavior of the energy $E$ as the number of particles diverges, we follow the scheme below. We fix a bounded open set $\Omega \subset \RR^2$ and  we scale the lattice spacing by a small parameter $\ln > 0$. Given $u \colon \ln \ZZ^2 \cap \Omega \to \SS^1$ and denoting the components of $\sigma$ by $(i,j) \in \ZZ \x \ZZ$ and by $u^{i,j}$ the value of $u$ at $(\ln i, \ln j)$, the study of the energy per particle in $\Omega$ can be reduced to considering the energy
\begin{equation} \label{eq:def of En}
   E_n(u;\Omega):= - \alpha \sum_{(i,j)} \ln^2 \big( u^{i,j} \cdot u^{i+1,j} + u^{i,j} \cdot u^{i,j+1} \big) + \sum_{(i,j)} \ln^2 \big( u^{i,j} \cdot u^{i+2,j} + u^{i,j} \cdot u^{i,j+2} \big) \, ,
\end{equation}
where $\alpha=J_1/J_3$ and the sums are taken over all $(i,j) \in \ZZ^2$ such that the evaluations of~$u$ above are defined. In this paper we are interested in the case $\alpha = \alpha_n = 4(1-\dn) \to 4$, that is when the system is close to the so-called ferromagnet/helimagnet transition point~\cite{DmiKri}.

For spin chains, the variational analysis as $\ln \to 0$ of the one-dimensional version~$E_n^{\mathrm{1d}}$ of~$E_n$ in this regime has been carried out in~\cite{CicSol} (see also~\cite{SciVal} for the analysis in other regimes of~$\alpha$). 
Even though our analysis relies on the one-dimensional result in \cite{CicSol}, it cannot be simply reduced to it and it presents additional difficulties peculiar of the higher-dimensional setting. In order to better describe them and summarize the content of this paper, we first recall some relevant results in the one-dimensional case.

The starting point of the analysis in \cite{CicSol} is the characterization the ground states of the physical system. 
These are referred to as helical spin chains: each spin is indeed  rotated with respect to its left neighbor by the same angle given by either $\arccos(1-\dn)$ or~$-\arccos(1-\dn)$. 
The two possible choices for the angle correspond to either counterclockwise or clockwise spin rotations, defining in this way a positive or a negative chirality of the spin chain, respectively. 
The description of the discrete-to-continuum limit as $\ln,\dn \to 0$  of~$E_n^{\mathrm{1d}}$ in terms of the spin variable $u$ results to provide poor information on the system, as the minimal energy can be attained also by fine mixtures of the spin field. 
Indeed, by~\cite[Proposition~4.1]{CicSol}, sequences of spin fields with equibounded energy are compact only with respect to the weak* topology of $L^\infty$ and the $\Gamma$-limit turns out to be constant on every field with values in the unit ball (cf.~\cite{AliCicGlo} for homogenization results for a general class of spin systems). 
In particular, the limit does not detect any kind of chirality transitions.

With the purpose of describing chirality transitions, one needs to carry out a finer analysis in the spirit of the development by $\Gamma$-convergence introduced in~\cite{BraTru}. 
One defines a new functional $H_n^{\mathrm{1d}}$ by properly scaling the energy $E_n^{\mathrm{1d}}$ referred to its minimum. 
In~\cite{CicSol} it has been shown that the relevant functional that captures chirality transitions is $H_n^{\mathrm{1d}}:= \frac{1}{\sqrt{2} \ln \dn^{3/2}} (E_n^{\mathrm{1d}} - \min E_n^{\mathrm{1d}})$. 
The $\Gamma$-convergence analysis of $H_n^{\mathrm{1d}}$ is carried out with respect to the $L^1$-convergence of the order parameter $w^i := \sqrt{\frac{2}{\dn}} \sin( \frac{1}{2} \theta^i )$, where~$\theta^i$ is the oriented angle between the two adjacent spins at positions~$\ln i$ and~$\ln (i+1)$, namely~$u^i$ and~$u^{i+1}$. 
One can interpret $w$ as a chirality order parameter. Indeed, if $u$ is a spin field with a constant chirality, then $w= 1$ ($w = -1$, resp.) if $u$ rotates counterclockwise (clockwise, resp.)  by the optimal angle $\arccos(1-\dn)$. 
The energy $H_n^{\mathrm{1d}}$ can be manipulated in a way that it can be recast as a discrete version of the Modica-Mortola functional in terms of the variable $w$ with transition length $\en = \frac{\ln}{\sqrt{2 \dn}}$. 
With the due care, a result contained in~\cite{Bra-Yip} leads to the well-known compactness and $\Gamma$-limit of the standard continuum Modica-Mortola functional~\cite{Mod, ModMor} with double well potential $W(s) = (1-s^2)^2$. 
In particular, assuming that $\en \to 0$, the limit functional is finite on $BV$ functions $w$ taking values in~$\{1,-1\}$ and counts (with a suitable multiplicative constant) the number of jump points of $w$, that is, the number of chirality transitions of the system, cf.~\cite[Theorem~4.2-(i)]{CicSol}. 
While chirality transitions in the one-dimensional setting can be satisfactorily described through the previous analysis, it will be clear in what follows that their description in the two-dimensional case requires additional ideas. 

We start the analysis of the two-dimensional problem by observing that the energy $E_n$ can be written as the sum of the one-dimensional energy $E_n^\mathrm{1d}$ on rows and columns of $\ln \ZZ^2$. 
In particular, each row and each column of the ground state configuration of the two-dimensional spin system is a helical spin chain, see Figure~\ref{fig:ground states}.  
Following the scheme already described above for the one-dimensional case, we refer the energy to its minimum and we scale it to obtain the energy~$H_n$. Manipulating~\eqref{eq:def of En} for $\alpha=\alpha_n = 4(1-\dn)$ and neglecting interactions at the boundary of $\Omega$ we get 
\begin{equation*}
        H_n(u;\Omega)  :=\frac{1}{\sqrt{2} \ln \dn^{3/2} }  \frac{1}{2}  \ \ln^2    \sum_{(i,j)} \Big| u^{i+2,j} - \frac{\alpha_n}{2} u^{i+1,j}  + u^{i,j} \Big|^2 + \Big| u^{i,j+2} - \frac{\alpha_n}{2} u^{i,j+1}  + u^{i,j} \Big|^2  .
\end{equation*}
In this paper we study the $\Gamma$-limit of $H_n$ as $\ln, \dn \to 0$ with respect to the $L^1$-convergence of the horizontal and vertical chirality order parameters $w$ and $z$ defined, respectively, by 
\begin{equation*}
    w^{i,j} := \sqrt{\frac{2}{\dn}} \sin\Big( \frac{1}{2} (\theth)^{i,j} \Big) \, , \quad  z^{i,j} := \sqrt{\frac{2}{\dn}} \sin\Big( \frac{1}{2} (\thetv)^{i,j} \Big) \, ,
\end{equation*}
where $(\theth)^{i,j}$ is the oriented angle between the two adjacent spins~$u^{i,j}$ and~$u^{i+1,j}$ and $(\thetv)^{i,j}$ is the oriented angle between the two adjacent spins~$u^{i,j}$ and~$u^{i,j+1}$. In terms of the pair of order parameters $(w,z)$, each of the ground states is characterized by one of the four pairs of (horizontal and vertical) chiralities $(1,1),(1,-1),(-1,1),(-1,-1)$, see Figure~\ref{fig:ground states}. 
In other words, all the columns of $\ln \ZZ^2$ are helical spin chains that rotate by the optimal angle $\arccos(1{-}\dn)$ and have the same chirality, the same being true for all the rows of $\ln \ZZ^2$.
\begin{figure}[H]
    \includegraphics{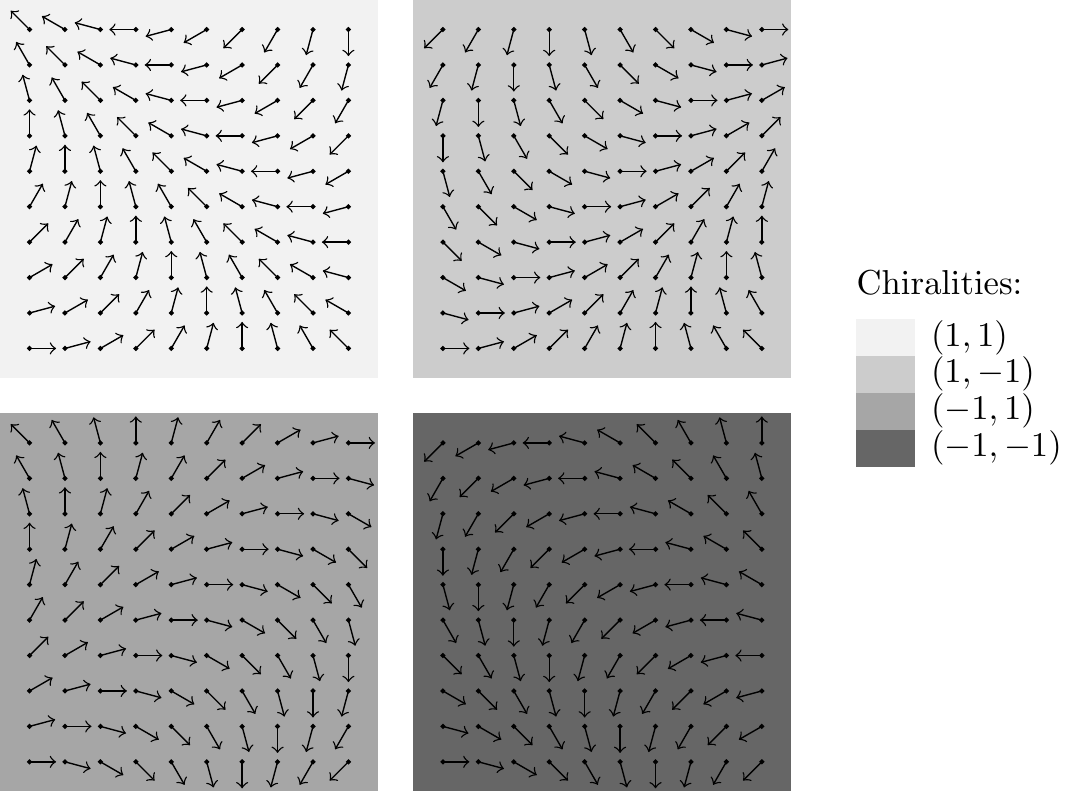}
    \caption{The ground states of the energy are classified according to the four possible values of the pair given by the horizontal and vertical chiralities.}
    \label{fig:ground states}
\end{figure}
For the reader's convenience we describe here the heuristics that leads to the limit behavior of~$H_n$. For $\ln$ small enough, we write $( u^{i+2,j} - 2 u^{i+1,j}  + u^{i,j}) / \ln^2 \simeq \de_{11} u$ to get
\begin{equation*}
    \begin{split}
         \ln^2    \sum_{(i,j)} \Big| u^{i+2,j} - \frac{\alpha_n}{2} u^{i+1,j}  + u^{i,j} \Big|^2 & = \ln^2    \sum_{(i,j)} \Big| \ln^2 \frac{u^{i+2,j} - 2 u^{i+1,j}  + u^{i,j}}{\ln^2} + 2 \dn u^{i+1,j}\Big|^2  \\
        & \simeq \int \! \ln^4 \big|\de_{11} u(x)\big|^2  + 4 \dn^2 + 4 \dn \ln^2  u(x)  \cdot \de_{11} u(x) \, \d x \\
        & \simeq \int \! \ln^4 \big|\de_{11} u(x)\big|^2  + 4 \dn^2 - 4 \dn \ln^2  \big|\de_{1} u(x)\big|^2 \, \d x \, .
    \end{split}
\end{equation*}
Taking formally $u = \big(\cos( \psi),\sin( \psi)\big)$ and using the fact that $|\de_1 u|^2 = |\de_1 \psi|^2$ and $|\de_{11}u|^2 = |\de_{11} \psi|^2 + |\de_{1} \psi|^4$, the above integral reads 
\begin{equation*}
   \int \! \ln^4 \big|\de_{11} \psi(x)\big|^2 + \big( 2 \dn - \ln^2 \big|\de_{1} \psi(x)\big|^2 \big)^2 \, \d x = \int \! \ln^4 \big|\de_{11} \psi(x)\big|^2 + 4 \dn^2\Big(1- \Big|\frac{\ln}{\sqrt{2 \dn}}  \de_{1} \psi(x)\Big|^2 \Big)^2 \, \d x  \,. 
\end{equation*}
Substituting $\varphi = \frac{\ln}{\sqrt{2 \dn}} \psi$ we obtain
\begin{equation*}
  \sqrt{2} \ln \dn^{3/2} 2 \int \!  \frac{\ln}{\sqrt{2 \dn}}  \big|\de_{11} \varphi(x)\big|^2 + \frac{\sqrt{2 \dn}}{\ln} \Big(1- \big| \de_{1} \varphi(x)\big|^2 \Big)^2 \, \d x  \,. 
\end{equation*}
Reasoning in an analogous way with $\de_{22} u$ and putting $\en := \frac{\ln}{\sqrt{2 \dn}}$, we have that  
\begin{equation} \label{eqintro:H in terms of phi}
    H_n(u;\Omega) \simeq \! \int \!  \en \big|\de_{11} \varphi(x)\big|^2 + \frac{1}{\en} \Big(1- \big| \de_{1} \varphi(x)\big|^2 \Big)^2 \! \d x {+} \int \!  \en \big|\de_{22} \varphi(x)\big|^2 + \frac{1}{\en} \Big(1- \big| \de_{2} \varphi(x)\big|^2 \Big)^2 \!  \d x\,.
\end{equation}
Noticing that $\de_1 \varphi = \frac{\ln}{\sqrt{2 \dn}} \de_1 \psi \simeq \frac{1}{\sqrt{2 \dn}} \theth \simeq \sqrt{\frac{2}{\dn}} \sin\Big( \frac{1}{2} \theth \Big)$ and $\de_2 \varphi \simeq \sqrt{\frac{2}{\dn}} \sin\Big( \frac{1}{2} \thetv \Big)$, by the very definition of $w$ and $z$ we have 
\begin{equation} \label{eqintro:w z almost gradient}
    (w,z) \simeq \nabla \varphi \, .
\end{equation}
Hence we conclude that 
\begin{equation} \label{eqintro:MM}
    H_n(u;\Omega) \simeq \int \!  \en \big|\de_{1} w(x)\big|^2 + \frac{1}{\en} \big(1-  | w(x) |^2 \big)^2   \d x +  \int \! \en \big|\de_{2} z(x)\big|^2 + \frac{1}{\en} \big(1- | z(x)|^2 \big)^2   \d x\,.
\end{equation}

This heuristics shows that the energy $H_n(u;\Omega)$ can be thought of as the sum of two functionals written in terms of the chirality variables $w$ and $z$. These functionals resemble the well-known Modica-Mortola functional, but each of them features only one partial derivative. Assuming $\en \to 0$, given a sequence $u_n$ with $H_n(u_n;\Omega) \leq C$ and with associated pairs of chiralities $(w_n,z_n)$ converging in $L^1$ to a pair~$(w,z)$, one expects that 
\begin{gather}
    w(x),z(x) \in \{1,-1\} \quad \text{for a.e.\ } x \in \Omega \, ,\label{eqintro:w and z 1}\\
    \D_1 w, \D_2 z \ \text{ are bounded measures in } \Omega \, , \label{eqintro:partial derivatives}
\end{gather}
where $\D_1 w, \D_2 z$ denote the distributional partial derivatives. (This can be seen, e.g., via a slicing argument in the vertical and horizontal directions.) However, these preliminary properties of $(w,z)$ do not characterize the admissible pairs of chiralities in the limit and need to be complemented in order to carry out an exhaustive $\Gamma$-convergence analysis. Indeed, they are derived by considering the two chiralities $w_n$ and $z_n$ as independent of each other, neglecting the fact that these are related to the same spin field $u_n$. In fact, such a constraint is already taken into account in the heuristic argument in~\eqref{eqintro:w z almost gradient} applied to $(w_n,z_n)$, which implies $\curl(w_n,z_n) \simeq 0$ and suggests for the limit $(w,z)$ that 
\begin{equation} \label{eqintro:curl 0}
    \curl(w,z) = 0 \  \text{ in the sense of distributions.}
\end{equation}
It is interesting to remark that a vector field $(w,z) \in L^1(\Omega;\RR^2)$ satisfying~\eqref{eqintro:partial derivatives} and~\eqref{eqintro:curl 0} does not belong, in general, to $BV(\Omega;\RR^2)$. 
This has been observed by Ornstein in~\cite{Orn}, where the author proves failure of the $L^1$ control of the mixed second derivatives of a function in terms of the $L^1$ norms of its pure second derivatives, see also~\cite[Theorem 3]{ConFarMag}. 
A natural question is whether the additional condition~\eqref{eqintro:w and z 1} guarantees that $(w,z) \in BV(\Omega;\RR^2)$. We do not answer this question here but, nonetheless, we prove that the limits $(w,z)$ of sequences~$(w_n,z_n)$ with equibounded energy do belong to $BV(\Omega;\RR^2)$. This is proven in Theorem~\ref{thm:main}-(i), where we show that a sequence~$(w_n,z_n)$ associated to a sequence of spin fields $u_n$ with $H_n(u_n;\Omega) \leq C$ admits a subsequence converging in $L^1$ to a limit pair $(w,z)$ that belongs to~$BV(\Omega;\RR^2)$. 

To prove the compactness result, the crucial observation is that the $L^2$ norm of the pure second derivatives of a compactly supported function does control the $L^2$ norm of its mixed derivative (in contrast to the $L^1$ case discussed above). As a consequence, from~\eqref{eqintro:H in terms of phi}--\eqref{eqintro:w z almost gradient} and taking care of the boundary conditions, we deduce that  (up to a multiplicative constant)
\begin{equation*}
    \begin{split}
        H_n(u_n;\Omega) &  \gtrsim  \int \!  \en \big|\nabla^2 \varphi_n(x)\big|^2 + \frac{1}{\en} \Big(1- \big| \de_{1} \varphi_n(x)\big|^2 \Big)^2 + \frac{1}{\en} \Big(1- \big| \de_{2} \varphi_n(x)\big|^2 \Big)^2 \!\! \d x  \\
        & \simeq   \int \!  \en \big|\nabla w_n(x)\big|^2 \!\! +  \frac{1}{\en} \big(1- | w_n(x)|^2 \big)^2 \! \d x  + \! \!\int \!  \en \big|\nabla z_n(x)\big|^2 \!\! + \frac{1}{\en} \big(1- |z_n(x)|^2 \big)^2 \! \d x  \, , 
    \end{split}
    \end{equation*}
which yields the desired compactness thanks to the well-known result for the Modica-Mortola functional in terms of the variables $w_n$ and $z_n$. Even though this argument might seem immediate from the heuristics we presented here, it does require a careful and technical analysis. One of the most delicate points is that of defining the appropriate variable that plays the role of the angular lifting $\varphi_n$, which may not exist due to possible topological defects of the spin field $u_n$.

The observations above lead us to the definition of the limit functional $H$. It is finite for a pair of chiralities $(w,z)$ belonging to $\Dom(H;\Omega)$, that is satisfying
\begin{gather*}
    w(x),z(x) \in \{1,-1\} \quad \text{for a.e.\ } x \in \Omega \, ,\\
     ( w, z ) \in BV(\Omega;\RR^2) \, ,  \\
     \curl(w,z) = 0 \  \text{ in the sense of distributions,}
\end{gather*} 
on which it takes the expression
\begin{equation*}
    H(w,z;\Omega) =  \frac{4}{3}  \Big( |\D_1 w|(\Omega) +  |\D_2 z|(\Omega) \Big) \, .
\end{equation*}
From~\eqref{eqintro:MM} one expects that 
\begin{equation} \label{eqintro:liminf}
    \liminf_{n \to \infty} H_n(u_n;\Omega) \geq H(w,z;\Omega) \, ,
\end{equation}
whenever the pairs of chiralities $(w_n,z_n)$ associated to $u_n$ converge in $L^1$ to $(w,z)$. This is the liminf inequality proven in Theorem~\ref{thm:main}-(ii) and it is obtained by a reduction to the one-dimensional setting.

\begin{figure}[H]
    \includegraphics{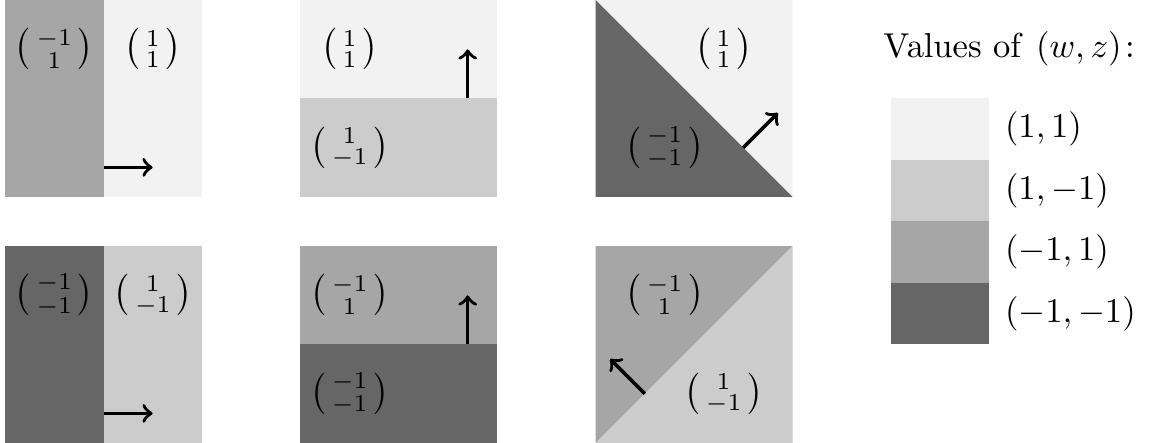}

    \caption{The six pictures above and the six analogous pictures obtained by flipping each of them with respect to the jump set  depict the only possible local configurations for the triple $((w,z)^+,(w,z)^-,\nu_{(w,z)})$ (up to the change to the triple $((w,z)^-, (w,z)^+, - \nu_{(w,z)})$). Each region where $(w,z)$ is constant is colored with a scale of grey going from the darkest to the lightest brightness according to the attained values ordered in the following way: $({-}1,{-}1)$, $({-}1,1)$, $(1,{-}1)$, $(1,1)$. These colored regions represent the ground states depicted in Figure~\ref{fig:ground states} in the limit.} 

    \label{fig:possible jumps}
\end{figure}

The construction of a recovery sequence that shows the optimality of the lower bound in~\eqref{eqintro:liminf} does not follow from the usual arguments exploited so far in the study of interfacial energies obtained as limits of Ising-type models~\cite{AliBraCic,BraGarPal,BraCic,BraKre}. In fact, it turns out to be a delicate task that can be carried out only after better understanding the rigidity induced to pairs $(w,z) \in \Dom(H;\Omega)$ by the $\curl$-free constraint. A first observation is that the properties defining $(w,z) \in \Dom(H;\Omega)$ fully characterize the  geometry of the blow-ups of $(w,z)$ at its jump points. Indeed, the condition $\curl(w,z) = 0$ implies the existence of a potential $\varphi$ such that $(w,z) = \nabla \varphi \in BV$. The normal to the jump set of a $BV$ gradient is forced to be aligned to the jump $(\nabla \varphi)^+ - (\nabla \varphi)^-$, where $(\nabla \varphi)^+$, $(\nabla \varphi)^-$ belong to $\{1,-1\}^2$. This entails a finite number of possible local geometries for $(w,z)$, see Figure~\ref{fig:possible jumps} and Subsection~\ref{subsec:rigid jump part} for a detailed discussion.

\begin{figure}[H]
    \vspace{-7em}
     \rotatebox{45}{
    \includegraphics{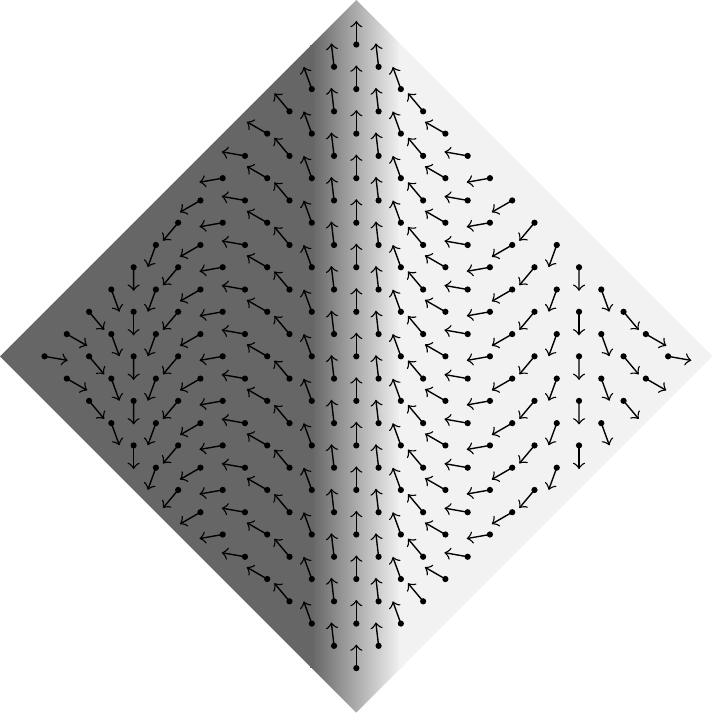}
     }
     \vspace{-6.5em}

    \caption{The picture shows a possible spin field configuration $u_n$ which has a chirality transition from $(-1,-1)$ to $(1,1)$.}

    \label{fig:transition}
\end{figure}

For these simple geometries of $(w,z)$ we can construct a recovery sequence, a particular case being shown in Figure~\ref{fig:transition}. 
However, there exist pairs in $\Dom(H;\Omega)$ whose jump sets present a less trivial geometry. 
For instance, in Figure~\ref{fig:fractal} we show how to construct an admissible pair $(w,z)$ with a jump set consisting of infinitely many segments. 
As it is customary when tackling these kinds of issues, one might look for an approximation of $(w,z) \in \Dom(H;\Omega)$, under which the functional $H$ is continuous, by means of a sequence of functions belonging to $\Dom(H;\Omega)$ having jump sets with a simpler structure, e.g., a polyhedral structure (consisting of finitely many segments).  
To the best of our knowledge, such a density result is not known and, unfortunately, techniques like the one used for Caccioppoli partitions in~\cite{BraConGar} or the polyhedral approximation for currents~\cite[Theorem~4.2.20]{Fed}  seem difficult to adapt to our framework due to the $\curl$-free constraint (see also~\cite{AliLazPal}). 
It is worth noticing that the more rigid case where $\nabla \varphi$ is a $BV$ function attaining only 3 values has been exhaustively described in~\cite{Mos}, where it has been proved that $\H^1$-a.e.\ point of the jump set has a neighborhood where the jump set is polyhedral. 
With an example similar to the one in Figure~\ref{fig:fractal}, in~\cite{Mos} the author shows that this does not hold true when $\nabla \varphi$ attains 4 values. 
Finally, every $BV$ gradient $\nabla \varphi$ attaining 2 values has a local laminar structure~\cite[Proposition~1]{BalJam}. Such a strong rigidity would allow to solve the problem of the limsup inequality by means of one-dimensional constructions. We also refer to~\cite{Con-Fon-Leo, ConSch, KitLucRul} for related problems motivated by elasticity.  


            

            


            
            



\begin{figure}[H]

    \includegraphics{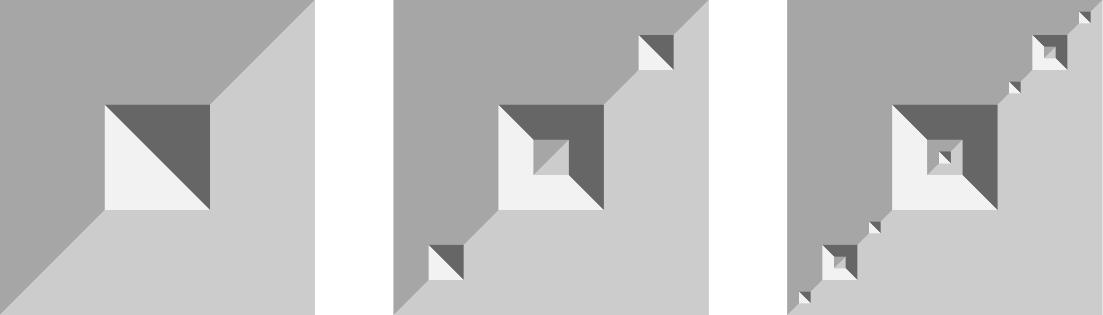}

    \caption{The first three steps of an iterative construction whose limit belongs to $\Dom(H;\Omega)$ and has a jump set consisting of infinitely many segments.}

    \label{fig:fractal}
\end{figure}

To overcome the difficulties mentioned above we resort to a more abstract technique introduced by Poliakovsky in~\cite{Pol07,Pol}, designed to provide upper bounds for general classes of functionals in singular perturbation problems and that is here for the first time applied in the context of discrete-to-continuum variational limits. In what follows we sketch the main steps of this technique and how it can be adapted to our framework. 
Given $(w,z) = \nabla \varphi$ belonging to $\Dom(H;\Omega)$, we fix a suitable convolution kernel $\eta$ and we define $\varphi^{\en}$ by mollifying $\varphi$ with~$\eta$ at the scale $\en$, the parameter of the Modica-Mortola functional in~\eqref{eqintro:H in terms of phi}. 
We exploit the discretization $\varphi_n$ of~$\varphi^{\en}$ to construct the angular lifting that defines the discrete spin field~$u_n$, whose associated chirality pairs $(w_n,z_n)$ provide a candidate for the recovery sequence. 
We relate the discrete energy $H_n(u_n;\Omega)$ to the Modica-Mortola-type functional in $\nabla \varphi^{\en}$ as in~\eqref{eqintro:H in terms of phi}. 
Thanks to~\cite{Pol}, this allows us to obtain an asymptotic upper bound which depends on the choice of $\eta$. 
The last main step in~\cite{Pol} is to further sharpen the upper bound by optimizing over the convolution kernel $\eta$. 
The sharpened upper bound can be explicitly expressed by integrating over the jump set of $(w,z)$ a surface density obtained by optimizing the transition energy over one-dimensional profiles.
In our case we can prove that this upper bound matches the lower bound in~\eqref{eqintro:liminf}, concluding the proof of the limsup inequality in Theorem~\ref{thm:main}-(iii). We stress that the technique described above requires that~$(w,z) \in BV(\Omega;\RR^2)$. For this reason it has been crucial to prove this property in the compactness result.

\section{The model and main result}

\subsection{Basic notation}

Given two vectors $a,b\in \RR^m$ we denote by $a \cdot b$ their scalar product. If $a,b\in \RR^2$, their cross product is given by $a \x b = a_1 b_2 - a_2 b_1$. As usual, the norm of $a$ is denoted by $|a| = \sqrt{a \cdot a}$. We denote by $\SS^1$ the unit circle in $\RR^2$. Given $a \in \RR^N$ and $b \in \RR^M$, their tensor product is the matrix $a \otimes b = (a_i b_j)^{i = 1,\dots,N}_{j = 1,\dots,M}   \in \RR^{N \times M}$.

In the whole paper $C$ denotes a constant that may change from line to line.   

\subsection{Discrete functions}

We  introduce here the notation used for functions defined on a square lattice in $\RR^2$. For the whole paper, $\ln$ denotes a sequence of positive lattice spacings that converges to zero. Given $i,j \in \ZZ$,  we define the half-open square $Q_{\ln}(i,j)$ with left-bottom corner in $(\ln i, \ln j)$ by $Q_{\ln}(i,j) := (\ln i, \ln j) + [0,\ln)^2$ and we denote its closure by $\ol Q_{\ln}(i,j)$. For a given set $S$, we introduce the class of functions with values in $S$ which are piecewise constant on the squares of the lattice $\ln \ZZ^2$: 
\begin{equation*}
    \PC_{\ln}(S) := \{ v \colon \RR^2 \to S  \ : \  v(x) = v(\ln i, \ln j) \text{ for } x \in Q_{\ln}(i,j)\} \, .
\end{equation*}
With a slight abuse of notation, we will always identify a function $v \in \PC_{\ln}(S)$ with the function defined on the points of the lattice $\ZZ^2$ given by $(i,j) \mapsto v^{i,j} := v(\ln i, \ln j)$ for $(i,j) \in \ZZ^2$. Conversely, given values $v^{i,j} \in S$ for $(i,j) \in \ZZ^2$, we define $v \in \PC_{\ln}(S)$ by $v(x) := v^{i,j}$ for $x \in Q_{\ln}(i,j)$.

Given a function $v \in \PC_{\ln}(\RR)$, we define the discrete partial derivatives $\de_1^{\mathrm{d}} v, \de_2^{\mathrm{d}} v \in \PC_{\ln}(\RR)$  by
\begin{equation*}
	\de_1^{\mathrm{d}}v^{i,j} := \frac{v^{i+1,j} - v^{i,j}}{\ln} \quad \text{and} \quad 	\de_2^{\mathrm{d}}v^{i,j}: = \frac{v^{i,j+1} - v^{i,j}}{\ln}  \, ,
\end{equation*}
and the discrete gradient by $\nabla^{\mathrm{d}} v := \vc{\de_1^{\mathrm{d}} v}{\de_2^{\mathrm{d}} v}$. Note that the operators $\de_1^{\mathrm{d}}$, $\de_2^{\mathrm{d}} \colon \PC_{\ln}(\RR) \to \PC_{\ln}(\RR)$ commute. Thus we may define the second order discrete derivatives $\de_{11}^{\mathrm{d}} v$, $\de_{12}^{\mathrm{d}} v = \de_{21}^{\mathrm{d}} v$, and $\de_{22}^{\mathrm{d}} v$ by iterative application of these operators in arbitrary order. Similarly, we define discrete partial derivatives of order higher than 2. 

  \subsection{Assumptions on the model}

Our model is an energy on discrete spin fields $u$ which are defined on square lattices inside a given domain $\Omega \subset \RR^2$. Our main result is valid whenever $\Omega$ is an open and bounded subset of $\RR^2$ that is simply connected and has a certain boundary regularity. Referring to Subsection~\ref{subsec:BV gradients} for the definition of a $BVG$ domain, we introduce the class of admissible domains by
\begin{equation*}
    \A_0 := \{ \Omega \subset \RR^2  \ : \  \Omega \text{ is an open, bounded, simply connected, $BVG$ domain} \} \, .
\end{equation*}
We recall that simply connected sets are by definition connected. 

To define the energies in our model, we introduce the set of indices  
\begin{equation*}
    \begin{split}
        \I^n(\Omega)  := \{(i,j) \in \ZZ \x \ZZ \ : \ & \ol Q_{\ln}(i,j), \ \ol Q_{\ln}(i+1,j), \ \ol Q_{\ln}(i,j+1) \subset \Omega \} 
    \end{split}
\end{equation*}
 for $\Omega \in \A_0$.  

Let $\dn$   be a sequence of positive real numbers that converges to zero and let  $\alpha_n = 4(1-\dn)$.  In the following we assume that $\frac{\ln}{\sqrt{\dn}} \to 0$ as $n \to \infty$.   We consider the functionals $\Hh_n, \Hv_n \colon L^\infty(\RR^2;\SS^1) \x \A_0 \to [0,+\infty]$ defined by 
\begin{equation*}
    \begin{aligned}
        \Hh_n(u;\Omega) & :=\frac{1}{\sqrt{2} \ln \dn^{3/2} }  \frac{1}{2}  \ \ln^2 \hspace{-1em}  \sum_{(i,j) \in \I^n(\Omega)} \Big| u^{i+2,j} - \frac{\alpha_n}{2} u^{i+1,j}  + u^{i,j} \Big|^2  , \\
        \Hv_n(u;\Omega) & :=  \frac{1}{\sqrt{2} \ln \dn^{3/2} } \frac{1}{2} \ \ln^2 \hspace{-1em}  \sum_{(i,j) \in \I^n(\Omega)} \Big| u^{i,j+2} - \frac{\alpha_n}{2} u^{i,j+1}  + u^{i,j} \Big|^2 \, ,
    \end{aligned}
\end{equation*}
for $u \in \PC_{\lambda_n}(\SS^1)$ and extended to $+\infty$ elsewhere. Then we define $H_n \colon L^\infty(\RR^2;\SS^1) \x \A_0 \to [0,+\infty]$ by
\begin{equation*}
    H_n := \Hh_n + \Hv_n \, .
\end{equation*}

 We next introduce the chirality order parameter $(w,z)$ associated to a spin field $u$. To every $u \in \PC_{\ln}(\SS^1)$ we associate the horizontal oriented angle between two adjacent spins~by 
\begin{equation} \label{eq:def of theth}
    (\theth)^{i,j} := \mathrm{sign}( u^{i,j} \x u^{i+1,j} ) \arccos( u^{i,j} \cdot u^{i+1,j} ) \, ,
\end{equation}
where we used the convention $\mathrm{sign}(0) = -1$. Analogously, we define 
\begin{equation} \label{eq:def of thetv}
    (\thetv)^{i,j} := \mathrm{sign}( u^{i,j} \x u^{i,j+1} ) \arccos( u^{i,j} \cdot u^{i,j+1} ) \, ,
\end{equation}
We additionally adopt the notation $\theth_n$ and $\thetv_n$ for the angles associated to $u_n$.  Note that $\theth, \thetv \in [-\pi,\pi)$. To every spin field $u \in \PC_{\ln}(\SS^1)$ we associate the order parameter $(w,z) \in \PC_{\ln}(\RR^2)$ defined by 
\begin{equation} \label{eq:order parameter}
    \begin{aligned}
        w^{i,j} & := \sqrt{\frac{2}{\delta_n}} \sin \Big( \frac{1}{2} (\theth)^{i,j} \Big) \, , \\
        z^{i,j} & := \sqrt{\frac{2}{\delta_n}} \sin \Big( \frac{1}{2} (\thetv)^{i,j} \Big) \,.
    \end{aligned}
\end{equation}
It is convenient to introduce the transformation $\T_n \colon \PC_{\ln}(\SS^1) \to \PC_{\ln}(\RR^2)$ given by 
\begin{equation*}
    \T_n(u) := (w,z) \, ,
\end{equation*}
where $(w,z)$ is given by~\eqref{eq:order parameter}. Note that if $u$ and $u'$ are such that $\T_n(u) = \T_n(u')$, then there exists a rotation $R \in  SO (2)$ such that $u' = R u$. 

With a slight abuse of notation we define the functional $H_n \colon \Lloc(\RR^2;\RR^2) \x \A_0 \to [0,+\infty]$
\begin{equation*}
    H_n(w,z; \Omega) := \begin{cases}
        H_n(u; \Omega) & \text{ if } \T_n(u) = (w,z) \text{ for some } u \in \PC_{\ln}(\SS^1) \, , \\
        + \infty & \text{ if } (w,z) \notin \T_n(\PC_{\ln}(\SS^1)) \, .
    \end{cases}
\end{equation*}
Notice that $H_n(u; \Omega)$ does not depend on the particular choice of $u$ since $H_n(u; \Omega)$ is rotation invariant. Analogously, we define $\Hh_n(w,z;\Omega)$ and $\Hv_n(w,z;\Omega)$.

To state the main result of the paper, we introduce the functional $H \colon \Lloc(\RR^2;\RR^2) \x \A_0 \to [0,+\infty]$ defined by
\begin{equation} \label{eq:def of H}
    H(w,z;\Omega)  := \begin{cases} \displaystyle
     \frac{4}{3}  \Big( |\D_1 w|(\Omega) +  |\D_2 z|(\Omega) \Big) & \text{ if } (w,z) \in \Dom(H;\Omega) \, , \\
    + \infty & \text{ if } (w,z) \in \Lloc(\RR^2;\RR^2) \sm \Dom(H;\Omega) \, ,
\end{cases}
\end{equation} 
where
\begin{equation}
    \begin{split}
        \Dom(H;\Omega) := \{ (w,z) \in \Lloc(\RR^2;\RR^2) \ : \ & (w,z) \in BV(\Omega;\RR^2) \, , \\
        & w \in \{1,-1\} \, \text{ and } z \in \{1,-1\} \text{ a.e.\ in } \Omega \, , \\
        & \curl(w,z) = 0 \text{ in } \mathcal{D}'(\Omega) \} \, . 
    \end{split}
\end{equation}
For the notions of $BV$ functions and of distributional $\curl$ we refer to Section~\ref{sec:BV preliminaries} below. 

\subsection{The main result}

The following theorem is the main result of the paper. 

\begin{theorem} \label{thm:main}
    Assume that $\Omega \in \A_0$. Then the following results hold true:
    \begin{itemize}
        \item[i)] (Compactness) Let $(w_n,z_n) \in \Lloc(\RR^2;\RR^2)$ be a sequence satisfying 
        \begin{equation*}
            H_n(w_n,z_n;\Omega) \leq C \, .
        \end{equation*}  
        Then there exists $(w,z) \in \Dom(H;\Omega)$ such that, up to a subsequence, $(w_n,z_n) \to (w,z)$  in $\Lloc(\Omega;\RR^2)$. 
        \item[ii)] (liminf inequality) Let $(w_n,z_n), (w,z) \in \Lloc(\RR^2;\RR^2)$. Assume that $(w_n,z_n) \to (w,z)$ in $\Lloc(\Omega;\RR^2)$. Then
        \begin{equation*}
            H(w,z;\Omega) \leq \liminf_{n \to \infty} H_n(w_n,z_n;\Omega) \, . 
        \end{equation*} 
        \item[iii)] (limsup inequality) Assume that  $(w,z) \in \Lloc(\RR^2;\RR^2)$. Then there exists a sequence $(w_n,z_n) \in \Lloc(\RR^2;\RR^2)$ such that $(w_n,z_n) \to (w,z)$ in $L^1(\Omega;\RR^2)$ and 
        \begin{equation*}
            \limsup_{n \to \infty} H_n(w_n,z_n;\Omega)  \leq  H(w,z;\Omega)  \, .
        \end{equation*} 
    \end{itemize}
\end{theorem}
\begin{remark}
    The proof of the compactness and of the  liminf  inequality actually work without requiring the simple connectedness of $\Omega$ and the regularity of its boundary.  
\end{remark}

\section{Preliminary results on $BV$ gradients} \label{sec:BV preliminaries}

Given an open set $\Omega \subset \RR^d$, we denote by $\M_b(\Omega;\RR^\ell)$ the space of all $\RR^\ell$-valued Radon measures on $\Omega$ with finite total variation. The total variation measure of $\mu \in \M_b(\Omega;\RR^\ell)$ is denoted by $| \mu |$. Moreover we denote by $\mathcal{D}'(\Omega)$ the space of distributions on $\Omega$ and by~$\langle T, \xi \rangle$ the duality between $T \in \mathcal{D}'(\Omega)$ and $\xi \in C^\infty_c(\Omega)$. 

\subsection{$BV$ functions} We start by recalling some basic facts about $BV$ functions, referring to the book~\cite{AmbFusPal} for a comprehensive treatment on the subject. 

Let $\Omega\subset\RR^d$ be an open set. A function $v \in L^1(\Omega;\RR^m)$ is a function of bounded variation, if its distributional derivative $\D v$ is a finite matrix-valued Radon measure, i.e., $\D v \in \M_b(\Omega;\RR^{m \x d})$. A function $v$ belongs to $BV_{\mathrm{loc}}(\Omega;\RR^m)$ if $v \in BV(\Omega';\RR^m)$ for every   open set   $\Omega' \subcc \Omega$. 

The distributional derivative $\D v \in \M_b(\Omega;\RR^{m \x d})$ of a function $v \in BV(\Omega;\RR^m)$ can be decomposed in the sum of three mutually singular matrix-valued measures 
\begin{equation} \label{eq:decomposition of BV}
    \D v = \D^a v + \D^c v + \D^j v = \nabla v \L^d + \D^c v + [v] \otimes \nu_v \H^{d-1} \mres J_v \, ,
\end{equation}
where $\L^d$ is the Lebesgue measure and $\H^{d-1}$ is the $(d-1)$-dimensional Hausdorff measure; $\nabla v \in L^1(\Omega;\RR^{m \x d})$ is the approximate gradient of $v$; $J_v$ denotes the $\H^{d-1}$-countably rectifiable jump set of $v$ oriented by the normal $\nu_v$, $[v] = (v^+ - v^-)$, and $v^+$ and $v^-$ denote the two traces of $v$ on $J_v$ in the sense that 
\begin{equation*}
    \lim_{r \to 0} \frac{1}{r^d} \hspace{-1em} \integral{B^{+}_r(x,\nu_v(x))}{  \hspace{-1em} |v(y) - v^+(x)|}{\d y} = 0 \, , \quad \lim_{r \to 0} \frac{1}{r^d}   \hspace{-1em} \integral{B^{-}_r(x,\nu_v(x))}{  \hspace{-1em} |v(y) - v^-(x)|}{\d y} = 0 \, ,
\end{equation*}
for $x \in J_v$, with $B^{\pm}_r(x,\nu) = \{ y \in B_r(x)\ : \ \pm (y-x) \cdot \nu > 0 \}$; $\D^c v$ is the so-called Cantor part of the derivative satisfying $\D^c v(B) = 0$ for every Borel set $B$ with $\H^{n-1}(B) < \infty$. We recall that the triple $(v^+,v^-,\nu_v)$ is determined uniquely up to the change to $(v^-,v^+,-\nu_v)$ and can be chosen as a Borel function on the Borel set $J_v$. 

\subsection{$BV$ gradients} \label{subsec:BV gradients} We recall here how the $\curl$-free condition on a $BV$ vector field enforces a rigid geometry on the jump part (actually on the full singular part) of its distributional derivative. Moreover, we recall the definition of $BVG$ functions introduced in~\cite{Pol07}.

Given a smooth vector field $v \colon \Omega \to \RR^d$ with components $(v^h)_{h=1,\dots, d}$, its $\curl$ is defined by $\curl(v) := (\de_h v^k - \de_k v^h)_{h,k=1,\dots,d}$. As usual, the operator $\curl$ is extended to vector-valued distributions $T \in \mathcal{D}'(\Omega;\RR^d)$ through the distribution $\curl(T) \in \mathcal{D}'(\Omega;\RR^{d \x d})$ defined in components by  
\begin{equation*}
    \langle (\curl(T))_{h,k}, \xi \rangle := - \langle T^k , \de_h\xi  \rangle + \langle T^h , \de_k  \xi  \rangle \, , \quad \text{for every } \xi \in C^\infty_c(\Omega) \, .
\end{equation*}

If $v \in BV(\Omega;\RR^d)$ and $\curl(v) = 0$, from the decomposition~\eqref{eq:decomposition of BV} we get 
\begin{equation*}
        0 = \D_h v^k - \D_k v^h  = \big( (\nabla v)_{kh} - (\nabla v)_{hk} \big) \L^d + \D_h^c v^k - \D_k^c v^h  +   \big( [v^k] \nu^h_v - [v^h] \nu^k_v \big) \H^{d-1} \mres J_v  \, , 
\end{equation*}
which yields, in particular, that $[v](x) \otimes \nu_v(x)$ is a symmetric rank-one matrix for $\H^{d-1}$-a.e.\ $x \in J_v$. This happens if and only if there exists  a function $\chi\in L^1_{\H^{d-1}}(J_v)$ such that $[v](x) = \chi(x) \nu_v(x)$ for $\H^{d-1}$-a.e.\ $x \in J_v$, i.e., 
\begin{equation} \label{eq:jump of curl free}
   \D^j v = [v] \otimes \nu_v \H^{d-1} \mres J_v = \chi \, \nu_v \otimes \nu_v \H^{d-1} \mres J_v \, .
\end{equation}

The structure of the jump set obtained in~\eqref{eq:jump of curl free} applies, e.g., to $BV$ gradients. Given a function $\varphi \in W^{1,1}(\Omega)$ such that $\nabla \varphi \in BV(\Omega;\RR^d)$, then there exists a function $\chi\in L^1_{\H^{d-1}}(J_{\nabla \varphi})$ such that 
\begin{equation*}
    \D^j \nabla \varphi = [\nabla \varphi] \otimes \nu_{\nabla \varphi} \H^{d-1} \mres J_{\nabla \varphi} = \chi \, \nu_{\nabla \varphi} \otimes  \nu_{\nabla \varphi} \H^{d-1} \mres J_{\nabla \varphi}  \, .
\end{equation*}
Functions $\varphi \in W^{1,1}(\Omega)$ such that $\nabla \varphi \in BV(\Omega;\RR^d)$ are usually called functions of bounded Hessian, see, e.g., \cite{Fon-Leo-Par}. In this paper we are interested in a different class: that of $BVG$ functions, introduced in~\cite{Pol07} and given by
\begin{equation*}
    BVG(\Omega) := \{ \varphi \in W^{1, \infty}(\Omega)  \ : \ \nabla \varphi \in BV(\Omega;\RR^d) \} \, .
\end{equation*}
Note that if $\Omega$ is a bounded open set, then clearly $BVG(\Omega)$ is contained in the class of functions of bounded Hessian.

In~\cite{Pol07}, the author proves a convenient extension result for functions in $BVG(\Omega)$ under suitable conditions on the regularity of the set $\Omega$. A bounded, open set $\Omega \subset \RR^d$ is called a $BVG$ domain if $\Omega$ can be described locally at its boundary as the epigraph of a $BVG$ function $\RR^{d-1} \to \RR$ with respect to a suitable choice of the axes, i.e., if every $x \in \de \Omega$ has a neighborhood $U_x \subset \RR^d$ such that there exists a function $\psi_x \in BVG(\RR^{d-1})$ and a rigid motion $R_x \colon \RR^d \to \RR^d$ satisfying
\begin{equation*}
R_x ( \Omega \cap U_x ) = \{ y = (y_1,y') \in \RR \x \RR^{d-1} \ : \ y_1 > \psi_x(y') \} \cap R_x (U_x) \, .
\end{equation*}
Note that every $BVG$ domain $\Omega \subset \RR^2$ is in particular a Lipschitz domain, but the converse is not true (an example is given by the epigraph of the primitive of a function which is continuous but not $BV$, e.g., the primitive of the Weierstrass function).  Note that smooth domains and polygons are $BVG$ domains.  

Every $BVG$ domain is an extension domain for $BVG$ functions in the following sense. 

\begin{proposition}[Proposition~4.1 in~\cite{Pol07}] \label{prop:extension of BVG}
Let $\Omega$ be a $BVG$ domain. Then for every $\varphi \in BVG(\Omega)$ there exists $\ol \varphi \in BVG(\RR^d)$ such that $\ol \varphi = \varphi$ in $\Omega$ and $|\D \nabla \ol \varphi|(\de \Omega) = 0$.
\end{proposition}

\subsection{$BV$ gradients with 4 values in dimension 2} \label{subsec:rigid jump part} Let $\Omega \subset \RR^2$ be an open set. Let us fix $(w,z) \in BV(\Omega; \RR^2)$ satisfying $\curl(w,z) = 0$ in $\mathcal{D}'(\Omega)$ and additionally
\begin{equation*}
    w(x) \in \{1,-1\} \, , \quad z(x) \in \{1,-1\} \, , \quad \text{for a.e.\ $x \in \Omega$} \, ,
\end{equation*}
i.e., $(w,z) \in \Dom(H;\Omega)$. This constraint on the values attained by $(w,z)$ enforces additional geometric structure on $\D (w,z) = \D^j (w,z)$. Indeed, by Subsection~\ref{subsec:BV gradients} we have that the jump $[(w,z)]$ and the normal $\nu_{(w,z)}$ are parallel vectors. Evaluating the possible values of the jump $[(w,z)]$ and taking into account that $|\nu_{(w,z)}|=1$, we conclude that for $\H^1$-a.e.\ $x \in J_{(w,z)}$ the triple $((w,z)^+(x), (w,z)^-(x),\nu_{(w,z)}(x))$ can be one of the following, up to the change to the triple $((w,z)^-(x), (w,z)^+(x), - \nu_{(w,z)}(x))$, 
\begin{equation} \label{eq:possible values of normal}
    \begin{split}
        & \Big(  \vc{1}{1}, \vc{-1}{1} , \vc{\pm 1}{0}  \Big) \, , \Big(\vc{1}{-1},  \vc{-1}{-1} , \vc{\pm 1}{0}  \Big)  \, , \Big( \vc{1}{1}, \vc{1}{-1} , \vc{0}{\pm 1}  \Big)  \, , \Big( \vc{-1}{1}, \vc{-1}{-1} ,  \vc{0}{\pm 1}  \Big)  \, , \\
        & \Big( \vc{1}{1}, \vc{-1}{-1} ,  \vc{\pm 1/\sqrt{2}}{\pm 1/\sqrt{2}}  \Big) \, , \Big(\vc{-1}{1} , \vc{1}{-1},   \vc{\mp 1/\sqrt{2}}{\pm 1/\sqrt{2}}  \Big) \, .
    \end{split}
\end{equation}
In particular, if $w$ jumps on a line that is not vertical (if $z$ jumps on a set that is not horizontal, resp.) then also $z$ must jump (also $w$ must jump, resp.) and the normal to the jump set is forced to be on one of the two directions $\vc{ 1/\sqrt{2}}{ 1/\sqrt{2}}$, $\vc{ - 1/\sqrt{2}}{ 1/\sqrt{2}}$.

This observation leads to the following result, which states that for this kind of functions a control on the sole partial derivatives $\D_1 w$, $\D_2 z$ is enough to bound the total variation of the full gradient.
\begin{proposition} \label{prop:bootstrap}
    Let $\Omega \subset \RR^2$ be an open set. Assume that $(w,z) \in BV_{\mathrm{loc}}(\Omega;\RR^2)$, $w(x) \in \{1,-1\}$ and $z(x) \in \{1,-1\}$ for a.e.\ $x \in \Omega$, and $\curl(w,z) = 0$ in $\mathcal{D}'(\Omega)$. Assume that 
    \begin{equation*}
        |\D_1 w|(\Omega) + |\D_2 z|(\Omega) < \infty \, .
    \end{equation*}
    Then $(w,z) \in BV(\Omega;\RR^2)$. 
\end{proposition}
\begin{proof}
    Thanks to the observation made above to deduce~\eqref{eq:possible values of normal}, we have that for $\H^1$-a.e.\ $x \in J_{(w,z)}$ the triple $((w,z)^+(x), (w,z)^-(x),\nu_{(w,z)}(x))$ attains one of the values listed in~\eqref{eq:possible values of normal}. We now write $J_{(w,z)}$ (up to $\H^1$ null sets) as the union of the three sets $\mathcal{J}_1$, $\mathcal{J}_2$, $\mathcal{J}_3$ of points $x \in J_{(w,z)}$ such that respectively
    \begin{align*}
        ((w,z)^+(x), (w,z)^-(x),\nu_{(w,z)}(x)) &\in \Big\{ \Big( \vc{-1}{1} , \vc{1}{1}, \vc{\pm 1}{0}  \Big) \, , \Big( \vc{-1}{-1} , \vc{1}{-1}, \vc{\pm 1}{0}  \Big) \Big\} \, , \\
        ((w,z)^+(x), (w,z)^-(x),\nu_{(w,z)}(x)) &\in \Big\{ \Big( \vc{1}{-1} , \vc{1}{1}, \vc{0}{\pm 1}  \Big)  \, , \Big( \vc{-1}{-1} , \vc{-1}{1}, \vc{0}{\pm 1}  \Big) \Big\} \, , \\
        ((w,z)^+(x), (w,z)^-(x),\nu_{(w,z)}(x)) &\in \Big\{  \Big( \vc{-1}{-1} , \vc{1}{1}, \vc{\pm 1/\sqrt{2}}{\pm 1/\sqrt{2}}  \Big) \, , \Big( \vc{-1}{1} , \vc{1}{-1}, \vc{\mp 1/\sqrt{2}}{\pm 1/\sqrt{2}}  \Big) \Big\} \, .
    \end{align*}
    We remark that
    \begin{equation*}
        \begin{aligned}
            \nu_{(w,z)}^2(x) &= 0 \, , \quad && \text{for $\H^1$-a.e.\ $x \in \mathcal{J}_1$} \, ,\\
            [w](x) &= 0 \, , \quad && \text{for $\H^1$-a.e.\ $x \in \mathcal{J}_2$} \, ,\\
            [w](x) &= \pm [z](x) \, ,\quad && \text{for $\H^1$-a.e.\ $x \in \mathcal{J}_3$} \, .
        \end{aligned}
        \end{equation*}
    Therefore
    \begin{equation*}
        \begin{split}
            |\D_2 w|(\Omega) & = \int \limits_{J_{(w,z)}} \! \! | [w] \, \nu_{(w,z)}^2 | \, \d \H^1 \\
            & = \int \limits_{\mathcal{J}_1} \!  |[w] \, \nu_{(w,z)}^2 | \, \d \H^1 + \int \limits_{\mathcal{J}_2} \!  | [w] \, \nu_{(w,z)}^2 | \d  \H^1+ \int \limits_{\mathcal{J}_3} \!  | [w] \, \nu_{(w,z)}^2 | \, \d \H^1  \\
            & = \int \limits_{\mathcal{J}_3} \! | [z] \, \nu_{(w,z)}^2 | \, \d \H^1 \leq |\D_2 z|(\Omega)
        \end{split}
    \end{equation*}
    Analogously we prove that $|\D_1 z|(\Omega) \leq |\D_1 w|(\Omega) < \infty$.
\end{proof}

\begin{remark}
    Let us fix $(w,z) \in L^1(\Omega;\RR^2)$ satisfying 
    \begin{gather*}
        w(x),z(x) \in \{1,-1\} \quad \text{for a.e.\ } x \in \Omega \, , \\
    \D_1 w, \D_2 z \in \M_b(\Omega) \, ,  \\
    \curl(w,z) = 0 \  \text{ in } \mathcal{D}'(\Omega) \, .
    \end{gather*}
    A natural question related to a problem already presented in the introduction is whether the condition $(w,z) \in BV(\Omega;\RR^2)$ follows immediately from the structure of $w$ and $z$ as indicator functions. More precisely, let 
    \begin{equation*}
        \begin{aligned}
            E_{+,+} & := \{ w = 1,\  z = 1 \} \, , && E_{+,-}  := \{ w = 1,\  z = - 1 \} \, , \\
            E_{-,+} & := \{ w = -1,\  z =  1 \} \, , && E_{-,-}  := \{ w = - 1,\  z = - 1 \} \, ,
        \end{aligned}
    \end{equation*}
and let $\mathds{1}_E$ denote the function given by $\mathds{1}_E(x) = 1$ if $x \in E$ and $\mathds{1}_E(x)=0$ if $x \notin E$. Then 
\begin{equation} \label{eq:char add up to 1}
    \mathds{1}_{E_{+,+}} +  \mathds{1}_{E_{+,-}} +  \mathds{1}_{E_{-,+}} + \mathds{1}_{E_{-,-}} = 1 \, ,
\end{equation}
and
\begin{align*}
    w = \mathds{1}_{E_{+,+}} +  \mathds{1}_{E_{+,-}} -  \mathds{1}_{E_{-,+}} - \mathds{1}_{E_{-,-}} \,, \\
    z = \mathds{1}_{E_{+,+}} -  \mathds{1}_{E_{+,-}} +  \mathds{1}_{E_{-,+}}  - \mathds{1}_{E_{-,-}} \, .
\end{align*}
From~\eqref{eq:char add up to 1} it follows that 
\begin{align}
    \D_1 \mathds{1}_{E_{+,+}} + \D_1 \mathds{1}_{E_{+,-}} +  \D_1 \mathds{1}_{E_{-,+}} + \D_1 \mathds{1}_{E_{-,-}} = 0  \label{eq:2703191136}\, , \\
    \D_2 \mathds{1}_{E_{+,+}} + \D_2 \mathds{1}_{E_{+,-}} +  \D_2\mathds{1}_{E_{-,+}} + \D_2 \mathds{1}_{E_{-,-}} = 0 \, .
\end{align}
Since $\D_1 w, \D_2 z \in \M_b(\Omega)$ we have that 
\begin{align}
     \D_1 \mathds{1}_{E_{+,+}} + \D_1 \mathds{1}_{E_{+,-}} -  \D_1 \mathds{1}_{E_{-,+}} - \D_1 \mathds{1}_{E_{-,-}}  \in \M_b(\Omega) \,, \\
    \D_2 \mathds{1}_{E_{+,+}}- \D_2 \mathds{1}_{E_{+,-}}  + \D_2 \mathds{1}_{E_{-,+}}  - \D_2 \mathds{1}_{E_{-,-}}  \in \M_b(\Omega)\, ,
\end{align}
while $\curl(w,z) = 0$ yields 
\begin{equation} \label{eq:2703191137}
    \D_2 \mathds{1}_{E_{+,+}} + \D_2 \mathds{1}_{E_{+,-}} -  \D_2 \mathds{1}_{E_{-,+}} - \D_2 \mathds{1}_{E_{-,-}}  - \D_1 \mathds{1}_{E_{+,+}}+ \D_1 \mathds{1}_{E_{+,-}}  - \D_1 \mathds{1}_{E_{-,+}}  + \D_1 \mathds{1}_{E_{-,-}}  = 0 \, .
\end{equation}
Proving that $(w,z) \in BV(\Omega;\RR^2)$ reduces to prove that $\D_2 \mathds{1}_{E_{+,+}} + \D_2 \mathds{1}_{E_{+,-}} -  \D_2 \mathds{1}_{E_{-,+}} - \D_2 \mathds{1}_{E_{-,-}}  \in \M_b(\Omega)$ (or, equivalently, that $ \D_1 \mathds{1}_{E_{+,+}}- \D_1 \mathds{1}_{E_{+,-}}  +\D_1 \mathds{1}_{E_{-,+}}  - \D_1 \mathds{1}_{E_{-,-}}  \in \M_b(\Omega)$). This condition follows from~\eqref{eq:2703191136}--\eqref{eq:2703191137} if $\D_2 \mathds{1}_{E_{+,+}} + \D_2 \mathds{1}_{E_{+,-}} -  \D_2 \mathds{1}_{E_{-,+}} - \D_2 \mathds{1}_{E_{-,-}}$ can be written as a linear combination of the expressions appearing in~\eqref{eq:2703191136}--\eqref{eq:2703191137}. However, this is not possible, since the matrix
 \begin{equation*}
     \begin{pmatrix}
         1 & 1 & 1 & 1 & 0 & 0 & 0 & 0 \\
         0 & 0 & 0 & 0 & 1 & 1 & 1 & 1 \\ 
         1 & 1 & - 1 & - 1 & 0 & 0 & 0 & 0 \\
         0 & 0 & 0 & 0 & 1 & - 1 & 1 & - 1 \\
         - 1 & 1 & - 1 & 1 & 1 & 1 & - 1 & - 1 \\
         0 & 0 & 0 & 0 & 1 & 1 & - 1 & - 1
     \end{pmatrix}
 \end{equation*} 
 has maximal rank.
\end{remark}

We conclude this section by showing a convenient result that allows us to work with potentials of the vector field $(w,z)$. This result will be useful in the proof of the limsup inequality. The definition and the properties of weakly Lipschitz sets that appear in the statement of the next lemma are recalled in the Appendix. 

\begin{lemma} \label{lemma:existence of potential}
    Let $\Omega \subset \RR^2$ be an open, bounded, weakly Lipschitz, and simply connected set. Assume that $(w,z) \in L^\infty(\Omega;\RR^2) \cap BV(\Omega;\RR^2)$, and $\curl(w,z) = 0$ in $\mathcal{D}'(\Omega)$. Then there exists a~$\varphi \in BVG(\Omega)$ such that $\nabla \varphi = (w,z)$. 
\end{lemma}  
\begin{proof}
    The proof is based on a smoothing argument which requires some care. 

    We start by finding a family of open, bounded, weakly Lipschitz, and simply connected sets $(\Omega_s)_s$ compactly contained in $\Omega$ that exhausts $\Omega$. To do so, we rely on a result proven in~\cite[Theorem 2.3]{Lic} (see also~\cite[Theorem 7.4 and Corollary 7.5]{Lui-Vai}), which guarantees that there exists a $t > 0$ and a bi-Lipschitz map $\Gamma \colon \de \Omega \x (-t,t) \to \Gamma(\de \Omega \x (-t,t)) \subset \RR^2$ such that the image of~$\Gamma$ is an open neighborhood of $\de \Omega$, $\Gamma(x,0) = x$ for  $x \in \de \Omega$, $\Gamma(\de \Omega \x (-t, 0)) \subset \Omega$, and $\Gamma(\de \Omega \x (0, t)) \subset \RR^2 \sm \ol{\Omega}$.
   For $s \in (0,t)$ we define the set $\Omega_s = \Omega \sm \Gamma( \de \Omega \x [-s, 0) ) \subset \Omega$. According to~\cite[Theorem 2.3]{Lic}, we can assume that each $\Omega_s$ is an open, bounded, and weakly Lipschitz set. We claim that each $\Omega_s$ is simply connected. To prove this we find a homeomorphism $f_s \colon \Omega \to \Omega_s$ by shrinking a neighborhood of the boundary of $\Omega$ along the paths $\tau \mapsto \Gamma(x_0, \tau)$. 
    We fix some $t_* \in (0,t)$ and for $s \in (0,t_*)$ we define
    \begin{equation*}
    	G_s \colon \de \Omega \x [-t_*,0] \to \de \Omega \x [-t_*, -s] \, , \quad G_s(x_0, \tau) = \Big(x_0, -s - \tau \frac{s-t_*}{t_*}\Big) \, .
	\end{equation*}
	Note that $G_s$ is a homeomorphism. We then define
    \begin{equation*}
    	f_s \colon \Omega \to \Omega_s \, , \quad f_s(x) :=
    	\begin{cases}
    		\Gamma(G_s(\Gamma^{-1}(x))) & \text{if } x \in \Gamma(\de \Omega \x [-t_*, 0)) \, ,\\
    		x & \text{if } x \in \Omega \setminus \Gamma(\de \Omega \x (-t_*, 0)) \, .
    	\end{cases}
    \end{equation*}

    Notice that $\Gamma(G_s(\Gamma^{-1}(x))) = x$ for $x \in \Gamma(\de \Omega \x \{ -t_* \})$ and that $f_s|_{\Gamma(\de \Omega \x [-t_*, 0))}$ and $f_s|_{\Omega \setminus \Gamma(\de \Omega \x (-t_*, 0))}$ are continuous functions. This yields that $f_s$ is continuous, since the sets  $\Gamma(\de \Omega \x [-t_*, 0))$ and $\Omega \sm \Gamma(\de \Omega \x (-t_*, 0))$ are relatively closed in $\Omega$.\footnote{    Indeed, $\de \Omega \x [-t_*, 0]$ is compact and thus so is $\Gamma(\de \Omega \x [-t_*, 0])$. Consequently, $\Gamma(\de \Omega \x [-t_*, 0)) = \Gamma(\de \Omega \x [-t_*, 0]) \cap \Omega$ is relatively closed in $\Omega$. Moreover, since $\Gamma^{-1}$ is continuous, $\Gamma(\de \Omega \x (-t_*, 0))$ is relatively open in $\Gamma(\de \Omega \x (-t,t))$, which is open in $\RR^2$. In particular, $\Gamma(\de \Omega \x (-t_*, 0))$ is open. It thereby follows that $\Omega \sm \Gamma(\de \Omega \x (-t_*, 0))$ is relatively closed in $\Omega$.}
    Analogously we prove the continuity of its inverse
    \begin{equation*}
    	f_s^{-1} \colon \Omega_s \to \Omega \, , \quad f_s^{-1}(x) =
    	\begin{cases}
    		\Gamma(G_s^{-1}(\Gamma^{-1}(x))) & \text{if } x \in \Gamma(\de \Omega \x [-t_*, -s)) \, ,\\
    		x & \text{if } x \in \Omega_s \sm \Gamma(\de \Omega \x (-t_*, -s)) \, .
    	\end{cases}
    \end{equation*}
    This concludes the proof of the simple connectedness of $\Omega_s$. Note that the sets~$\Omega_s$ increase and exhaust $\Omega$ as $s$ decreases to zero. Moreover $\Omega_s \subcc \Omega$, since for every $s$ we have that $\Gamma( \de \Omega \x (-s,s))$ is an open neighborhood of $\de \Omega$ that does not intersect $\Omega_s$.

    \begin{figure}[H]

        \scalebox{0.8}{
        \includegraphics{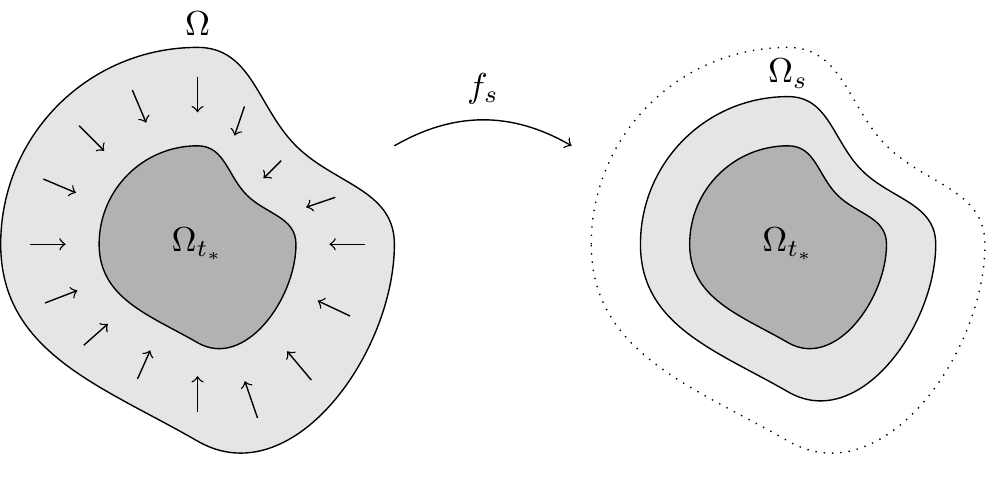}
        }
    
        \caption{The map $f_s$ transforms $\Omega$ into $\Omega_s$ by shrinking the collar $\Gamma(\de \Omega \x [-t_*, 0)) = \Omega \sm \Omega_{t_*}$ (in light grey) into the collar $\Gamma(\de \Omega \x [-t_*, -s)) = \Omega_s \sm \Omega_{t_*}$, while keeping fixed the set $\Omega \setminus \Gamma(\de \Omega \x (-t_*, 0)) = \ol \Omega_{t_*}$ (in dark grey).}
        
    \end{figure}

    We are now in a position to mollify $(w,z)$. Let us define $d_s := \dist (\Omega_s, \de \Omega) > 0$. We fix a mollifier $\rho \in C^\infty_c(B_1(0))$ with $\int \! \rho \d x = 1$, $\rho \geq 0$, and define for $s > 0$ the function $\rho_s(z) := \frac{1}{d_s^2} \rho(\frac{z}{d_s})$. Then we put  $(w_s,z_s):= \rho_s * (w,z) \in C^\infty(\Omega_s; \RR^2)$. Note that  $\| (w_s,z_s) \|_{L^\infty(\Omega_s)}$ is bounded and that $(w_s,z_s) \to (w,z)$ in $\Lloc(\Omega;\RR^2)$. 
    Moreover $\curl(w_s,z_s) = \rho_s * \curl(w,z) = 0$ in $\Omega_s$. Since the set $\Omega_s$ is simply connected, we find a function $\varphi_s \in C^\infty(\Omega_s; \RR)$ such that $\nabla \varphi_s = (w_s,z_s)$ in $\Omega_s$.
    
    A limit of $\varphi_s$ as $s \to 0$ will be a good candidate for a potential of $(w,z)$. To find such a limit, we proceed as follows. We fix an $s_0 \in (0, t_*)$. For $s < s_0$ we define 
    \begin{equation*}
    	\varphi_{s,s_0} := \varphi_s - \frac{1}{| \Omega_{s_0} |} \integral{\Omega_{s_0}} {\varphi_s}{\d x} \, .
    \end{equation*} 
    The above function is well-defined since $\varphi_s \in C(\ol{\Omega}_{s_0})$. Note that $\nabla \varphi_{s,s_0} =(w_s,z_s)$ in~$\Omega_{s_0}$. Since $\Omega_{s_0}$ is a connected, weakly Lipschitz set, the Poincar\'e-Wirtinger inequality (cf.\ Remark~\ref{rmk:rellich-kondrachov in weakly lipschitz}) yields boundedness of $(\varphi_{s,s_0})_s$ in $L^\infty(\Omega_{s_0})$.
    Thus, a subsequence converges weakly* in $L^\infty(\Omega_{s_0})$ to some $\varphi^{s_0} \in L^\infty(\Omega_{s_0})$. Using the convergence $(w_s,z_s) \to (w,z)$ in $L^1(\Omega_{s_0};\RR^2)$ we obtain that $\nabla \varphi^{s_0} = (w,z)$ in $\Omega_{s_0}$ and~$\varphi^{s_0} \in W^{1,\infty}(\Omega_{s_0})$. Repeating this construction for a sequence $s_n \to 0$ (in place of $s_0$) we see that $(w,z)$ has potentials $\varphi^{s_n}$ on connected sets $\Omega_{s_n}$ which increase and exhaust $\Omega$ as $n \to \infty$. Note that for every $n$ there exists a constant $c_n$ such that $\varphi^{s_n} - c_n = \varphi^{s_0}$ a.e.\ in $\Omega_{s_0}$ and moreover $\varphi^{s_n} -c_n = \varphi^{s_m} - c_m$ a.e.\ in $\Omega_{s_n} \cap \Omega_{s_m}$ for all $n$ and $m$. We define the function $\varphi \in L^\infty_{\mathrm{loc}}(\Omega)$ by
    \begin{equation*}
    	\varphi (x) := \varphi^{s_n} (x) - c_n \quad \text{for } x \in \Omega_{s_n} \, ,
    \end{equation*}
    for any $n$.   The function  $\varphi$ belongs to $W^{1,\infty}_{\text{loc}}(\Omega)$ and is a potential of $(w,z)$. By Proposition~\ref{prop:bounded gradient and Lipschitz boundary} in the Appendix it follows that $\varphi \in L^\infty(\Omega)$ and this concludes the proof.
 \end{proof}

\section{Proof of compactness and of the  liminf  inequality} 

In this section we prove the compactness result and the liminf   inequality, i.e., Theorem~\ref{thm:main}-(i), (ii). As explained in the introduction, the key to our $\Gamma$-convergence analysis is the relation of $H_n$ to a Modica-Mortola-type functional as in~\eqref{eqintro:MM}. The starting point is to prove a precise version of this approximate equality. To state this result, we need to introduce a correcting factor $\rho \colon [-\pi,\pi]^2 \to \RR$ defined by 
    \begin{equation} \label{eq:def of rho}
        \rho(\theta_1, \theta_2) := \begin{cases} \displaystyle
            \frac{-  \big(1 - \cos(\theta_1 + \theta_2 ) \big) + \sin^2(\theta_1) + \sin^2(\theta_2)}{ 2 \Big( \sin\big( \tfrac{1}{2}\theta_2 \big) - \sin\big( \tfrac{1}{2} \theta_1 \big) \Big)^2 } & \text{if } \theta_1 \neq \theta_2 \, , \\
            1 & \text{if } \theta_1 = \theta_2 \, .
        \end{cases}
    \end{equation}

    We set $W(s) := (1-s^2)^2$ and
    \begin{equation*}
        \en := \frac{\ln}{\sqrt{2 \dn}}   
    \end{equation*}
    and we work under the assumption $\en \to 0$.

\begin{lemma} \label{lemma:H is a discrete MM}
    Let $(w,z) \in \T_n(\PC_{\ln}(\SS^1))$ be as in~\eqref{eq:order parameter} and let $\rho$ be as in~\eqref{eq:def of rho}. Then
    \begin{align*}
            \Hh_n(w,z;\Omega) & =  \frac{1}{2\en} \ln^2 \hspace{-1em} \sum_{(i,j) \in \I^n(\Omega)} \hspace{-1em}   W(w^{i,j})  + W(w^{i+1,j}) + \en \ln^2 \hspace{-1em}  \sum_{(i,j) \in \I^n(\Omega)} \hspace{-1em}  (\rho^\mathrm{hor})^{i,j} \big| \de^{\mathrm{d}}_{1} w^{i,j} \big|^2  , \\
            \Hv_n(w,z;\Omega) & =  \frac{1}{2\en} \ln^2 \hspace{-1em} \sum_{(i,j) \in \I^n(\Omega)} \hspace{-1em}   W(z^{i,j})  + W(z^{i,j+1}) + \en \ln^2 \hspace{-1em}  \sum_{(i,j) \in \I^n(\Omega)} \hspace{-1em}  (\rho^\mathrm{ver})^{i,j} \big| \de^{\mathrm{d}}_{2} z^{i,j} \big|^2,
    \end{align*}
    where $(\rho^\mathrm{hor})^{i,j} = \rho((\theth)^{i,j},(\theth)^{i+1,j})$ and  $(\rho^\mathrm{ver})^{i,j} = \rho((\thetv)^{i,j},(\thetv)^{i,j+1})$.
\end{lemma} 
\begin{proof}
    Let $u \in \PC_{\ln}(\SS^1)$ be such that $(w,z) = T_n(u)$.
    We start by observing that 
    \begin{equation} \label{eq:1912181339}
        \begin{split}
                       & \frac{1}{2}  \ \ln^2 \hspace{-1em}  \sum_{(i,j) \in \I^n(\Omega)} \Big| u^{i+2,j} - \frac{\alpha_n}{2} u^{i+1,j}  + u^{i,j} \Big|^2 \\
                       & \quad  = \ln^2 \hspace{-1em}  \sum_{(i,j) \in \I^n(\Omega)} \hspace{-1em} u^{i+2,j} \cdot u^{i,j}   - \alpha_n  \frac{1}{2} \ \ln^2 \hspace{-1em} \sum_{(i,j) \in \I^n(\Omega)} \hspace{-1em} u^{i+2,j} \cdot u^{i+1,j} + u^{i+1,j} \cdot u^{i,j}  \\
                       & \quad \quad +  \Big(1 + \frac{\alpha_n^2}{8} \Big) \ \ln^2 \#\I^n(\Omega) \\
                       & \quad = - \ln^2 \hspace{-1em}  \sum_{(i,j) \in \I^n(\Omega)} \hspace{-1em} (1 - u^{i+2,j} \cdot u^{i,j} ) \\
                       & \quad \quad + \alpha_n  \frac{1}{2} \ \ln^2 \hspace{-1em} \sum_{(i,j) \in \I^n(\Omega)} \hspace{-1em} (1 - u^{i+2,j} \cdot u^{i+1,j}) + (1 - u^{i+1,j} \cdot u^{i,j}) \\
                       & \quad \quad + \Big(2 - \alpha_n + \frac{\alpha_n^2}{8} \Big) \ \ln^2 \#\I^n(\Omega)  \, .
    \end{split}
\end{equation}
On the one hand, by~\eqref{eq:order parameter} and since $2 \dn (w^{i,j})^2 \big(1 - \tfrac{\dn}{2}(w^{i,j})^2 \big) = \sin^2((\theth)^{i,j})$, we get that
\begin{equation} \label{eq:1912181340}
    \begin{split}
        & \alpha_n  \frac{1}{2} \ \ln^2 \hspace{-1em} \sum_{(i,j) \in \I^n(\Omega)} \hspace{-1em}  (1 - u^{i+1,j} \cdot u^{i,j}) \\
        & \quad =  \alpha_n  \frac{1}{2} \ \ln^2 \hspace{-1em} \sum_{(i,j) \in \I^n(\Omega)} \hspace{-1em}  \big(1 - \cos((\theth)^{i,j}) \big) =  4(1-\dn)  \ \ln^2 \hspace{-1em} \sum_{(i,j) \in \I^n(\Omega)} \hspace{-1em}  \sin^2\big( \tfrac{1}{2} (\theth)^{i,j} \big)  \\
         & \quad =   \ln^2 \hspace{-1em} \sum_{(i,j) \in \I^n(\Omega)} \hspace{-1em}  (2 \dn - 2\dn^2)  (w^{i,j})^2 - 2 \dn (w^{i,j})^2 \big(1 - \tfrac{\dn}{2}(w^{i,j})^2 \big) + \sin^2((\theth)^{i,j}) \\
         & \quad = \dn^2 \ln^2 \hspace{-1em} \sum_{(i,j) \in \I^n(\Omega)} \hspace{-1em}   \big( 1- (w^{i,j})^2 \big)^2  +  \ln^2 \hspace{-1em} \sum_{(i,j) \in \I^n(\Omega)} \hspace{-1em} \sin^2((\theth)^{i,j}) - \dn^2 \ln^2 \#\I^n(\Omega) \, .\\
    \end{split}
\end{equation}
Analogously, we have that 
\begin{equation} \label{eq:1912181341}
    \begin{split}
        & \alpha_n  \frac{1}{2} \ \ln^2 \hspace{-1em} \sum_{(i,j) \in \I^n(\Omega)} \hspace{-1em}  (1 - u^{i+2,j} \cdot u^{i+1,j}) \\
        & \quad = \dn^2 \ln^2 \hspace{-1em} \sum_{(i,j) \in \I^n(\Omega)} \hspace{-1em}   \big( 1- (w^{i+1,j})^2 \big)^2  +  \ln^2 \hspace{-1em} \sum_{(i,j) \in \I^n(\Omega)} \hspace{-1em} \sin^2((\theth)^{i+1,j}) - \dn^2 \ln^2 \#\I^n(\Omega)  \, .
    \end{split}
\end{equation}
On the other hand, using the definition of $\rho$ in~\eqref{eq:def of rho} we have that
\begin{equation} \label{eq:1912181342}
    \begin{split}
    &  \ln^2 \hspace{-1em}  \sum_{(i,j) \in \I^n(\Omega)} \hspace{-1em} - (1 - u^{i+2,j} \cdot u^{i,j} ) + \sin^2((\theth)^{i+1,j}) + \sin^2((\theth)^{i,j}) \\
    & \quad =  \ln^2 \hspace{-1em}  \sum_{(i,j) \in \I^n(\Omega)} \hspace{-1em} -  \big(1 - \cos \big((\theth)^{i,j} + (\theth)^{i+1,j} \big) \big) + \sin^2((\theth)^{i+1,j}) + \sin^2((\theth)^{i,j}) \\
    & \quad = \ln^2 \hspace{-1em}  \sum_{(i,j) \in \I^n(\Omega)} \hspace{-1em} \ln^2 \dn (\rho^\mathrm{hor})^{i,j} \big| \de^{\mathrm{d}}_{1} w^{i,j} \big|^2.
    \end{split}
\end{equation}

Gathering \eqref{eq:1912181339}--\eqref{eq:1912181342} and using the fact that $2 - \alpha_n + \frac{\alpha_n^2}{8} -2\dn^2 = 0$, we conclude that 
\begin{equation*}
    \begin{split}
        \Hh_n(w,z;\Omega) & = \frac{1}{\sqrt{2} \ln \dn^{3/2}}  \Big[ \dn^2 \ln^2 \hspace{-1em} \sum_{(i,j) \in \I^n(\Omega)} \hspace{-1em}   \big( 1- (w^{i,j})^2 \big)^2  + \big( 1- (w^{i+1,j})^2 \big)^2 \\ 
        & \hspace{10em} + \ln^4 \dn \hspace{-1em}  \sum_{(i,j) \in \I^n(\Omega)} \hspace{-1em}  (\rho^\mathrm{hor})^{i,j} \big| \de^{\mathrm{d}}_{1} w^{i,j} \big|^2\Big] \\
        & =  \frac{1}{2\en} \ln^2 \hspace{-1em} \sum_{(i,j) \in \I^n(\Omega)} \hspace{-1em}   W(w^{i,j})  + W(w^{i+1,j}) + \en \ln^2 \hspace{-1em}  \sum_{(i,j) \in \I^n(\Omega)} \hspace{-1em}  (\rho^\mathrm{hor})^{i,j} \big| \de^{\mathrm{d}}_{1} w^{i,j} \big|^2.
    \end{split}
\end{equation*}
This concludes the proof of the equality for $\Hh_n$. The equality for $\Hv_n$ is proven analogously. 
\end{proof}

In the next lemma we explain in which sense the correcting factor $\rho$ is close to 1.

\begin{lemma} \label{lemma:rho is close to 1} The function $\rho$ defined in~\eqref{eq:def of rho} is bounded from below and
    \begin{equation*}
        \lim_{(\theta_1,\theta_2) \to (0,0)} \rho(\theta_1, \theta_2) = 1 \, .
    \end{equation*}
\end{lemma}
\begin{proof}
    Fix $(\theta_1, \theta_2) \in [-\pi,\pi]^2$ with $\theta_1 \neq \theta_2$ and set $\theta_+ = \frac{\theta_2 + \theta_1}{2}$ and $\theta_- = \frac{\theta_2 - \theta_1}{2}$. We rewrite $\rho(\theta_1,\theta_2)$ in terms of $\theta_+$ and $\theta_-$. Using trigonometric identities, we have
    \begin{equation*}
    	\begin{split}
    		&  - \big(1 - \cos(\theta_1 + \theta_2 ) \big) + \sin^2(\theta_1) + \sin^2(\theta_2) \\
    		& \quad =  \cos(2\theta_+) - \cos^2(\theta_+ - \theta_-) + \sin^2(\theta_+ + \theta_-) \\
    		& \quad = \cos(2\theta_+) - (\cos (\theta_+) \cos (\theta_-) + \sin (\theta_+) \sin (\theta_-))^2 + (\sin (\theta_+) \cos (\theta_-) + \cos (\theta_+) \sin (\theta_-))^2 \\
            & \quad = \cos(2\theta_+) + ( \cos^2 (\theta_+) - \sin^2 (\theta_+) ) ( \sin^2 (\theta_-) - \cos^2 ( \theta_-) )  = (1- \cos(2\theta_-)) \cos(2\theta_+) \\
            & \quad = 2 \sin^2 (\theta_-) \cos(2\theta_+)  = 8 \sin^2 (\theta_-/2) \cos^2 (\theta_-/2) \cos(2\theta_+) \, .
    	\end{split}
    \end{equation*}
    Moreover, using the identity $\sin x - \sin y = 2 \sin \big( \frac{x-y}{2} \big) \cos \big( \frac{x+y}{2} \big)$ we get
    \begin{equation*}
    	2 \Big( \sin\big( \tfrac{1}{2}\theta_2 \big) - \sin\big( \tfrac{1}{2} \theta_1 \big) \Big)^2 = 8 \sin^2 (\theta_-/2) \cos^2 (\theta_+/2) \, .
    \end{equation*}
    Our choice of $\theta_1$ and $\theta_2$ implies that $\sin(\theta_-/2) \neq 0$ and thus, from the equalities above, 
    \begin{equation*}
    	\rho(\theta_1,\theta_2) = \frac{\cos^2(\theta_-/2) \cos(2\theta_+)}{\cos^2(\theta_+/2)} = \frac{\cos^2((\theta_2-\theta_1)/4) \cos(\theta_1+\theta_2)}{\cos^2((\theta_1+\theta_2)/4)} \, .
    \end{equation*}
    From this we easily deduce the claimed limit of $\rho$ at $(0,0)$. To show that $\rho$ is bounded from below, it is enough to show that the function
    \begin{equation*}
    	[-\pi,\pi]^2 \sm \{(-\pi,-\pi),(\pi,\pi)\} \ni (\theta_1, \theta_2) \mapsto \frac{\cos^2((\theta_2-\theta_1)/4) \cos(\theta_1+\theta_2)}{\cos^2((\theta_1+\theta_2)/4)}
    \end{equation*}
    is bounded from below. This holds true since it is continuous on $[-\pi,\pi]^2 \sm \{(-\pi,-\pi),(\pi,\pi)\}$ (its domain) and positive in neighborhoods of the points $(-\pi, -\pi)$ and $(\pi,\pi)$.
\end{proof}

Let us now fix $\Omega \in \A_0$ and $(w_n,z_n) \in L^1_{\mathrm{loc}}(\RR^2;\RR^2)$ such that $H_n(w_n,z_n;\Omega) \leq C$. Let $u_n \in \PC_{\ln}(\SS^1)$ be such that $(w_n,z_n) = T_n(u_n)$, let $\theth_n$ and $\thetv_n$ be as in~\eqref{eq:def of theth}--\eqref{eq:def of thetv}, and let $\rho^\mathrm{hor}_n$ and $\rho^\mathrm{ver}_n$ be defined as in Lemma~\ref{lemma:H is a discrete MM}. 

A first consequence of Lemma~\ref{lemma:rho is close to 1} and of the boundedness of the energy is contained in Lemma \ref{lemma:removing rho}. It implies that, despite $\rho^\mathrm{hor}_n$ and $\rho^\mathrm{ver}_n$ might attain negative values, a sequence $(w_n,z_n)$ with equibounded $H_{n}$ energy has an equibounded discrete Modica-Mortola-type energy: 
\begin{align}
   \frac{1}{ 2\en} \ln^2 \hspace{-1em} \sum_{(i,j) \in \I^n(\Omega)} \hspace{-1em}   W(w_n^{i,j})  + W(w_n^{i+1,j}) + \en \ln^2 \hspace{-1em}  \sum_{(i,j) \in \I^n(\Omega)} \hspace{-1em}   \big| \de^{\mathrm{d}}_{1} w_n^{i,j} \big|^2 & \leq C \, , \label{eq:discrete MM bound 1} \\
    \frac{1}{ 2\en} \ln^2 \hspace{-1em} \sum_{(i,j) \in \I^n(\Omega)} \hspace{-1em}   W(z_n^{i,j})  + W(z_n^{i,j+1}) + \en \ln^2 \hspace{-1em}  \sum_{(i,j) \in \I^n(\Omega)} \hspace{-1em}  \big| \de^{\mathrm{d}}_{2} z_n^{i,j} \big|^2 & \leq C \label{eq:discrete MM bound 2} \, .
\end{align}  

\begin{lemma} \label{lemma:removing rho}
    Assume that $(w_n,z_n) \in L^1_{\mathrm{loc}}(\RR^2;\RR^2)$ is such that $H_n(w_n,z_n;\Omega) \leq C$. Then there exists a constant $C > 0$ such that 
    \begin{align}
     \en \ln^2 \hspace{-1em}  \sum_{(i,j) \in \I^n(\Omega)} \hspace{-1em}  (\rho^\mathrm{hor}_n)^{i,j} \big| \de^{\mathrm{d}}_{1} w_n^{i,j} \big|^2  & \geq   \frac{1}{2}\en \ln^2 \hspace{-1em}  \sum_{(i,j) \in \I^n(\Omega)} \hspace{-1em}  \big| \de^{\mathrm{d}}_{1} w_n^{i,j} \big|^2 - C  \, , \label{eq:bounding from below with gradient 1} \\
     \en \ln^2 \hspace{-1em}  \sum_{(i,j) \in \I^n(\Omega)} \hspace{-1em}  (\rho^\mathrm{ver}_n)^{i,j} \big| \de^{\mathrm{d}}_{2} z_n^{i,j} \big|^2  & \geq   \frac{1}{2}\en \ln^2 \hspace{-1em}  \sum_{(i,j) \in \I^n(\Omega)} \hspace{-1em}  \big| \de^{\mathrm{d}}_{2} z_n^{i,j} \big|^2 - C  \label{eq:bounding from below with gradient 2} \, .  
    \end{align}
\end{lemma}
\begin{proof}
    We prove here only~\eqref{eq:bounding from below with gradient 1}, since \eqref{eq:bounding from below with gradient 2} can be proved similarly.

    \ul{Step 1}: Let us count the number of points $(i,j) \in \I^n(\Omega)$ such that $(\rho^\mathrm{hor}_n)^{i,j} < \frac{1}{2}$. By Lemma~\ref{lemma:rho is close to 1}, there exists a $\sigma \in (0,\pi)$ such that 
    \begin{equation*}
        \text{if } \ \big|(\theth_n)^{i,j}\big|, \big|(\theth_n)^{i+1,j}\big| < \sigma \  \text{ then }  \ (\rho^\mathrm{hor}_n)^{i,j} \geq \tfrac{1}{2} \, .
    \end{equation*}
    Therefore, if $(\rho^\mathrm{hor}_n)^{i,j} < \tfrac{1}{2}$ then either $\big|(\theth_n)^{i,j}\big| \geq \sigma$ or  $ \big|(\theth_n)^{i+1,j}\big| \geq \sigma$. It follows that it is enough to estimate the number of points $(i,j) \in \I^n(\Omega)$ such that $\big|(\theth_n)^{i,j}\big| \geq \sigma$ and those such that  $ \big|(\theth_n)^{i+1,j}\big| \geq \sigma$. In order to do so, we 
   exploit the inequality 
    \begin{equation*}
        \bigg( \Big| u^{i,j}  - \frac{\alpha_n}{2} u^{i+1,j} \Big|  + 1 \bigg)^{\! 2} \leq   \bigg( 2 + \frac{\alpha_n}{2}  \bigg)^{\! 2}
    \end{equation*}
    and the bound $H_n(u_n;\Omega) = H_n(w_n,z_n;\Omega) \leq C$ to deduce that, for $n$ large enough, 
    \begin{equation*}
        \begin{split}
            C \frac{\dn^{3/2}}{\ln} & \geq  \hspace{-1em}   \sum_{(i,j) \in \I^n(\Omega)} \Big| u^{i+2,j} - \frac{\alpha_n}{2} u^{i+1,j}  + u^{i,j} \Big|^2 \\
            & \geq  \hspace{-1em}  \sum_{(i,j) \in \I^n(\Omega)} \bigg( \Big| u^{i,j}  - \frac{\alpha_n}{2} u^{i+1,j} \Big| - 1 \bigg)^{\! 2} \\
            & \geq \frac{1}{\big( 2 + \frac{\alpha_n}{2}  \big)^2}  \sum_{(i,j) \in \I^n(\Omega)} \bigg( \Big| u^{i,j}  - \frac{\alpha_n}{2} u^{i+1,j} \Big|^2 - 1 \bigg)^{\! 2} \\
            & = \frac{\alpha_n^2 }{\big( 2 + \frac{\alpha_n}{2}  \big)^2}  \sum_{(i,j) \in \I^n(\Omega)} \bigg( \frac{\alpha_n}{4} -  u^{i,j} \cdot u^{i+1,j}  \bigg)^{\! 2} \\
            & \geq \frac{9}{16}   \sum_{(i,j) \in \I^n(\Omega)}  \Big( 1-\dn -  \cos\big( (\theth_n)^{i,j} \big)  \Big)^{\! 2}.
        \end{split}
    \end{equation*}
    Analogously, we obtain the bound 
    \begin{equation*}
        \sum_{(i,j) \in \I^n(\Omega)}  \Big( 1-\dn - \cos\big( (\theth_n)^{i+1,j} \big)  \Big)^{\! 2} \leq  C \frac{\dn^{3/2}}{\ln} \, .
    \end{equation*}
    As a consequence, for $n$ large enough,
    \begin{equation} \label{eq:counting theta}
        \begin{split}   
            C \frac{\dn^{3/2}}{\ln} & \geq \sum_{\substack{(i,j) \in \I^n(\Omega) \\  |(\theth_n)^{i,j}| \geq \sigma}}  \Big( 1-\dn -  \cos\big( (\theth_n)^{i,j} \big)  \Big)^{\! 2}  \geq \sum_{\substack{(i,j) \in \I^n(\Omega) \\ |(\theth_n)^{i,j}| \geq \sigma}}  \Big( 1-\dn -  \cos \sigma  \Big)^{\! 2} \\
            & \geq \# \big\{ (i,j) \in \I^n(\Omega)  \ : \  \big|(\theth_n)^{i,j}\big| \geq \sigma \big\} \frac{1}{4} \big(1 - \cos \sigma\big)^2.
        \end{split}
    \end{equation}
    Analogously, we infer that 
    \begin{equation*}
        \# \big\{ (i,j) \in \I^n(\Omega)  \ : \  \big|(\theth_n)^{i+1,j}\big| \geq \sigma \big\} \leq C \frac{\dn^{3/2}}{\ln}  \, .
    \end{equation*}

    In conclusion, 
    \begin{equation} \label{eq:number of bad rhos}
        \# \big\{ (i,j) \in \I^n(\Omega)  \ : \  \big|(\rho^\mathrm{hor}_n)^{i,j}\big| < \tfrac{1}{2} \big\} \leq C \frac{\dn^{3/2}}{\ln} \, .
    \end{equation}

    \ul{Step 2}: We split the right-hand side in~\eqref{eq:bounding from below with gradient 1} as follows
    \begin{equation*}
        \begin{split}
            & \frac{1}{2}\en \ln^2 \hspace{-1em}  \sum_{(i,j) \in \I^n(\Omega)} \hspace{-1em}  \big| \de^{\mathrm{d}}_{1} w_n^{i,j} \big|^2 \\
            &  \leq  \en \ln^2 \hspace{-1em}  \sum_{\substack{(i,j) \in \I^n(\Omega)\\ (\rho^\mathrm{hor}_n)^{i,j} \geq \frac{1}{2} }} \hspace{-1em} (\rho^\mathrm{hor}_n)^{i,j} \big| \de^{\mathrm{d}}_{1} w_n^{i,j} \big|^2  +  \frac{1}{2} \en \ln^2 \hspace{-1em}  \sum_{\substack{(i,j) \in \I^n(\Omega)\\ (\rho^\mathrm{hor}_n)^{i,j} < \frac{1}{2} }} \hspace{-1em}  \big| \de^{\mathrm{d}}_{1} w_n^{i,j} \big|^2   \\
            & = \en \ln^2 \hspace{-1em}  \sum_{ (i,j) \in \I^n(\Omega)} \hspace{-1em} (\rho^\mathrm{hor}_n)^{i,j} \big| \de^{\mathrm{d}}_{1} w_n^{i,j} \big|^2  +  \frac{1}{2} \en \ln^2 \hspace{-1em}  \sum_{\substack{(i,j) \in \I^n(\Omega)\\ (\rho^\mathrm{hor}_n)^{i,j} < \frac{1}{2} }} \hspace{-1em}  (1 - 2 (\rho^\mathrm{hor}_n)^{i,j}) \big| \de^{\mathrm{d}}_{1} w_n^{i,j} \big|^2.
        \end{split}
    \end{equation*}
 From the definition of $w_n$, cf.~\eqref{eq:order parameter}, we get $|w_n^{i,j}| \leq \sqrt{\tfrac{2}{\dn}}$ and thus $| \de_1^\mathrm{d} w_n^{i,j}|^2 \leq \tfrac{8}{\dn \ln^2}$. Together with the fact that $(\rho_n^\mathrm{hor})^{i,j}$ is bounded from below (Lemma~\ref{lemma:rho is close to 1}) and~\eqref{eq:number of bad rhos}, this yields 
 \begin{equation*}
    \frac{1}{2} \en \ln^2 \hspace{-1em}  \sum_{\substack{(i,j) \in \I^n(\Omega)\\ (\rho^\mathrm{hor}_n)^{i,j} < \frac{1}{2} }} \hspace{-1em}  (1 - 2 (\rho^\mathrm{hor}_n)^{i,j}) \big| \de^{\mathrm{d}}_{1} w_n^{i,j} \big|^2 \leq C \en \ln^2 \frac{\dn^{3/2}}{\ln} \frac{1}{\dn \ln^2} = C 
 \end{equation*}
and concludes the proof of~\eqref{eq:bounding from below with gradient 1}.
\end{proof}

Our next goal, which is at the core of the compactness result, is to prove that it is possible to obtain a bound for the pair $(w_n,z_n)$ in terms of a full discrete Modica-Mortola functional, namely, a Modica-Mortola functional where the regularization term features the whole discrete gradients of $w_n$ and $z_n$ instead of the sole partial discrete derivatives~$\de_1^\mathrm{d} w_n$ and~$\de_2^\mathrm{d} z_n$. As suggested by the heuristics presented in the introduction, the condition $(w_n,z_n) \in T_n(\PC_{\ln}(\SS^1))$ implies that the pair $(w_n,z_n)$ can be approximately written as the discrete gradient of a suitable potential~$\varphi_n$ which is obtained as a scaled angular lifting of $u_n$. The partial derivatives $\de_2^\mathrm{d} w_n$ and~$\de_1^\mathrm{d} z_n$ (whose control is as yet missing) are then related to the mixed second derivative $\de_{12}^\mathrm{d} \varphi_n$, whose $L^2$ norm can be controlled by that of the pure derivatives  $\de_{11}^\mathrm{d} \varphi_n$ and~$\de_{22}^\mathrm{d} \varphi_n$.

\begin{proposition} \label{prop:full MM bound}
    Let $(w_n,z_n)$ be as in Lemma~\ref{lemma:removing rho}. Let $Q \subcc \Omega$ be an open square with sides parallel to the coordinate axes. Then there exists a constant $C > 0$ (depending on $Q$) such that 
    \begin{align}   
        \frac{1}{ \en}    \integral{Q}{W(w_n )}{\d x}  + \en \integral{Q}{ \big| \nabla^{\mathrm{d}} w_n  \big|^2}{\d x} & \leq C \, , \label{eq:discrete full MM bound 1} \\
        \frac{1}{ \en}   \integral{Q}{W(z_n )}{\d x}  + \en \integral{Q}{ \big| \nabla^{\mathrm{d}} z_n  \big|^2}{\d x} & \leq C \label{eq:discrete full MM bound 2} \, .
    \end{align}
    Moreover 
    \begin{equation} \label{eq:curl to 0} 
        \curl(w_n,z_n) \weak 0  \quad \text{ in } \mathcal{D}'(Q) \, .     
    \end{equation}
\end{proposition}
\begin{proof}
    To prove the statement, we have to show that 
    \begin{equation} \label{eq:2103191338}
        \en \|\de_2^\mathrm{d} w_n \|_{L^2(Q)}^2 \leq C \, , \quad \en \|\de_1^\mathrm{d} z_n \|_{L^2(Q)}^2 \leq C \, .
    \end{equation}

We first control $\en \|\de_1^\mathrm{d} z_n\|_{L^2(Q)}^2$.  Let us fix two additional open   squares $Q'$ and $Q''$ with sides parallel to the coordinate axes such that $Q \subcc Q' \subcc Q'' \subcc \Omega$. We fix $(i_0,j_0) \in \ZZ^2$ (actually depending also on $n$) such that~$i_0$ and~$j_0$ are the smallest integer indices such that $(\ln i_0, \ln j_0) \in Q''$. Then for every $(i, j) \in \ZZ^2$  such that $(\ln i, \ln j) \in Q''$ we define
    \begin{equation} \label{eq:def of phi}
        \varphi_n^{i,j} := \frac{1}{\arccos(1-\dn)} \bigg( \sum_{h=i_0}^{i-1} \ln (\theth_n)^{h,j_0} + \sum_{k=j_0}^{j-1} \ln (\thetv_n)^{i, k} \bigg) - a^n_{Q'} \, ,
    \end{equation}
  where $a^n_{Q'} \in \RR$. The particular value of $a^n_{Q'}$ will be relevant only in Step~3 of Lemma~\ref{lemma:all estimates} below ($a^n_{Q'}$ is chosen so that  the piecewise affine interpolation of $\varphi_n$ has zero average in the square $Q'$; this choice will be clear later). The function $\varphi_n$ satisfies 
  \begin{equation} \label{eq:de 2 phi and thetv}
      \de_2^\mathrm{d} \varphi_n^{i,j} = \frac{1}{\arccos(1-\dn)}  (\thetv_n)^{i,j}  .
  \end{equation}
  Therefore, from~\eqref{eq:order parameter} it follows that 
  \begin{equation} \label{eq:relation between z and de 2 phi}
    z_n^{i,j} = \sqrt{\frac{2}{\dn}} \sin \Big( \frac{1}{2} \arccos(1-\dn) \de_2^\mathrm{d} \varphi_n^{i,j} \Big)  \, .
\end{equation}
By definition of $(\thetv_n)^{i,j}$ we have that $(\thetv_n)^{i,j} \in [-\pi,\pi)$. Thus  $\tfrac{1}{2} \arccos(1-\dn) \de_2^\mathrm{d} \varphi_n^{i,j}  \in \big[-\tfrac{\pi}{2}, \tfrac{\pi}{2}\big)$ and we can invert the formula above, getting
\begin{equation} \label{eq:de 2 phi with arcsin} 
    \de_2^\mathrm{d} \varphi_n^{i,j} =  \frac{2}{\arccos(1-\dn)} \arcsin \Big( \sqrt{\frac{\dn}{2}} z_n^{i,j} \Big) \, .
\end{equation}

    We claim that 
    \begin{equation} \label{eq:claim on control on mixed derivative}
        \en \|\de_{12}^\mathrm{d} \varphi_n \|_{L^2(Q)}^2 \leq C \, .
    \end{equation}
    To prove this, we resort to a discrete integration by parts formula that allows us to control~$\de_{12}^\mathrm{d} \varphi_n$ in terms of the pure second discrete derivatives. To make the argument rigorous, let us fix a cut-off function $\zeta \in C^\infty_c(Q')$ such that $0 \leq \zeta \leq 1$ and $ \zeta \equiv 1 $ on a neighborhood of~$\ol Q$. We introduce the piecewise constant function defined by $\zeta_n^{i,j} := \zeta(\ln i, \ln j)$.  For $n$ large enough we have
    \begin{equation} \label{eq:control on mixed derivative}
        \begin{split}
            & \|\de_{12}^\mathrm{d} \varphi_n\|^2_{L^2(Q)} \\
            & \leq \sum_{(i,j) \in \ZZ^2} \ln^2 \Big|\de_{12}^\mathrm{d}\big( \zeta_n \varphi_n \big)^{i,j} \Big|^2 = \sum_{(i,j) \in \ZZ^2} \ln^2 \de_{12}^\mathrm{d}\big( \zeta_n \varphi_n \big)^{i,j}  \cdot \de_{12}^\mathrm{d}\big( \zeta_n \varphi_n \big)^{i,j}  \\
            & = - \hspace{-0.7em}\sum_{(i,j) \in \ZZ^2} \ln^2 \de_{2}^\mathrm{d}\big( \zeta_n \varphi_n \big)^{i+1,j} \cdot \de_{112}^\mathrm{d}\big( \zeta_n \varphi_n \big)^{i,j} \\
            & = \sum_{(i,j) \in \ZZ^2} \ln^2 \de_{22}^\mathrm{d}\big( \zeta_n \varphi_n \big)^{i+1,j}  \cdot \de_{11}^\mathrm{d}\big( \zeta_n \varphi_n \big)^{i,j+1}  .
        \end{split} 
    \end{equation}
In Lemma~\ref{lemma:all estimates} below we deduce the following bounds 
\begin{gather} 
     \en \|\varphi_n\|_{L^2(Q')}^2 + \en \| \nabla^\mathrm{d} \varphi_n \|_{L^2(Q'')}^2  \to 0  \label{eq2:bounds lot}\, , \\
     \en \| \de_{22}^\mathrm{d} \varphi_n \|_{L^2(Q'')}^2   \leq C  \, , \label{eq2:bound de 22}  
\end{gather}
that will be useful to continue the estimate in~\eqref{eq:control on mixed derivative}, since these terms  appear as a result of a (discrete) product rule applied to $\zeta_n \varphi_n$. 
A major difficulty in the proof arises in the bound on $\de_{11}^\mathrm{d} \varphi_n$, whose $L^2$ norm we cannot control.
    This is due to the fact that, in contrast to~\eqref{eq:de 2 phi with arcsin},  $\de_1^\mathrm{d} \varphi_n^{i,j}$ might not be equal to $\tfrac{2}{\arccos(1-\dn)} \arcsin ( \sqrt{\tfrac{\dn}{2}} w_n^{i,j})$. For this reason it is convenient to introduce the auxiliary variable
\begin{equation}  \label{eq:def of tilde w}
    \tilde w_n^{i,j} :=  \frac{2}{\arccos(1-\dn)} \arcsin \Big( \sqrt{\frac{\dn}{2}} w_n^{i,j} \Big) = \frac{1}{\arccos(1-\dn)} (\theth_n)^{i,j}   .
\end{equation}
 Notice that the second equality above is a consequence of the definition of $w_n^{i,j}$ in~\eqref{eq:order parameter} and of the fact that $\frac{1}{2} (\theth_n)^{i,j} \in [-\frac{\pi}{2}, \frac{\pi}{2})$.  In Lemma~\ref{lemma:all estimates} below we deduce the convergences 
 \begin{equation} \label{eq2:de 1 phi and tilde w are close}
    \|\de_1^\mathrm{d} \varphi_n - \tilde w_n\|_{L^1(Q'')} \to 0 \, , \quad \en \|\de_1^\mathrm{d} \varphi_n - \tilde w_n\|^2_{L^2(Q'')} \to 0 \,         
\end{equation}
 and we prove the estimate
 \begin{equation} \label{eq2:bound de1 w}
     \en \|\de_1^\mathrm{d} \tilde w_n\|_{L^2(Q'')} \leq C \, .
 \end{equation}

    With these bounds at hand, we can go further with the estimate in~\eqref{eq:control on mixed derivative}. By a discrete product rule we get 
    \begin{equation*}
        \begin{split}
        \|\de_{12}^\mathrm{d} \varphi_n\|^2_{L^2(Q)} &  \leq \sum_{(i,j) \in \ZZ^2} \ln^2 \de_{22}^\mathrm{d}\big( \zeta_n \varphi_n \big)^{i+1,j}   \cdot \Big( \de_{11}^\mathrm{d} \zeta_n^{i,j+1} \varphi_n^{i+2,j+1} \\[-1em]
        & \hphantom{\sum_{(i,j) \in \ZZ^2} \ln^2} + 2 \de_1^\mathrm{d} \zeta_n^{i,j+1} \de_1^\mathrm{d} \varphi_n^{i+1,j+1} + \zeta_n^{i,j+1} \de_{11}^\mathrm{d} \varphi_n^{i,j+1} \Big) \\
        & \leq \sum_{(i,j) \in \ZZ^2} \ln^2 \de_{22}^\mathrm{d}\big( \zeta_n \varphi_n \big)^{i+1,j} \cdot \Big( \de_{11}^\mathrm{d} \zeta_n^{i,j+1} \varphi_n^{i+2,j+1} \\[-1em]
        & \hphantom{\sum_{(i,j) \in \ZZ^2} \ln^2} + 2 \de_1^\mathrm{d} \zeta_n^{i,j+1} \de_1^\mathrm{d} \varphi_n^{i+1,j+1} + \zeta_n^{i,j+1} \de_{1}^\mathrm{d} \tilde w_n^{i,j+1} \Big) \\
        & \hphantom{\leq} +  \sum_{(i,j) \in \ZZ^2} \ln^2 \de_{22}^\mathrm{d}\big( \zeta_n \varphi_n \big)^{i+1,j}   \cdot \Big( \zeta_n^{i,j+1} \de_{1}^\mathrm{d} \big( \de_{1}^\mathrm{d} \varphi_n - \tilde w_n \big)^{i,j+1} \Big)  \\
        & \leq C  \|\de_{22}^\mathrm{d}\big( \zeta_n \varphi_n \big) \|_{L^2(Q')} \Big( \|\varphi_n  \|_{L^2(Q')} +  \|\de_1^\mathrm{d}\varphi_n  \|_{L^2(Q'')} + \|\de_1^\mathrm{d} \tilde w_n\|_{L^2(Q'')} \Big) + R_n \\
        & \leq C  \Big( \|\varphi_n\|_{L^2(Q')} + \|\de_2^\mathrm{d} \varphi_n\|_{L^2(Q'')} + \|\de_{22}^\mathrm{d} \varphi_n\|_{L^2(Q'')}  \Big)  \cdot \\
        & \hspace{4em} \cdot  \Big( \|\varphi_n  \|_{L^2(Q')} +  \|\de_1^\mathrm{d}\varphi_n  \|_{L^2(Q'')} + \|\de_1^\mathrm{d} \tilde w_n\|_{L^2(Q'')} \Big) + R_n \, ,
    \end{split}
    \end{equation*}
    where we put
    \begin{equation} \label{eq:def of Rn}
    R_n := \sum_{(i,j) \in \ZZ^2} \ln^2 \de_{22}^\mathrm{d}\big( \zeta_n \varphi_n \big)^{i+1,j} \cdot \Big( \zeta_n^{i,j+1} \de_{1}^\mathrm{d} \big( \de_{1}^\mathrm{d} \varphi_n - \tilde w_n \big)^{i,j+1} \Big) \, 
    \end{equation}
    and we recall that $\mathrm{supp}(\zeta) \subcc Q'$. Therefore by~\eqref{eq2:bounds lot}, \eqref{eq2:bound de 22}, and~\eqref{eq2:bound de1 w}, we obtain that
    \begin{equation} \label{eq:almost finished bound}
    \begin{split}
         \en \|\de_{12}^\mathrm{d} \varphi_n\|_{L^2(Q)}^2 & \leq C \Big( \en \|\varphi_n\|^2_{L^2(Q')} + \en \| \de_1^\mathrm{d} \varphi_n\|_{L^2(Q'')}^2 + \en \| \de_2^\mathrm{d} \varphi_n\|_{L^2(Q'')}^2 \\
         & \hphantom{\leq C \Big(}  + \en \| \de_1^\mathrm{d} \tilde w_n\|_{L^2(Q'')}^2 + \en \| \de_{22}^\mathrm{d} \varphi_n\|_{L^2(Q'')}^2 \Big) + \en R_n \\
         &  \leq C  + \en R_n \, .
    \end{split}
    \end{equation}
    It only remains to control $\en R_n$.
    
    By a discrete integration by parts applied to~\eqref{eq:def of Rn} we get
    \begin{equation} \label{eq:integrating Rn by parts}
        \begin{split}    
           R_n = & - \hspace{-0.5em}\sum_{(i,j) \in \ZZ^2} \ln^2 \de_{2}^\mathrm{d}\big( \zeta_n \varphi_n \big)^{i+1,j+1}   \cdot \de_{2}^\mathrm{d} \big( \zeta_n  \de_{1}^\mathrm{d} \big( \de_{1}^\mathrm{d} \varphi_n - \tilde w_n \big)\big)^{i,j+1}  \, .
        \end{split}
    \end{equation}
    Using a discrete product rule we infer that 
    \begin{equation*}
        \begin{split}
         \de_{2}^\mathrm{d}\big( \zeta_n \varphi_n \big)^{i+1,j+1} = \de_{2}^\mathrm{d} \zeta_n^{i+1,j+1} \varphi_n^{i+1,j+1} + \zeta_n^{i+1,j+2} \de_{2}^\mathrm{d}\varphi_n^{i+1,j+1} 
        \end{split}
    \end{equation*}
    and
    \begin{equation*}
        \begin{split}          
            & \de_2^\mathrm{d} \big( \zeta_n \de_{1}^\mathrm{d} \big( \de_{1}^\mathrm{d} \varphi_n  - \tilde w_n  \big)  \big)^{i,j+1} \\ 
            &\quad = \de_2^\mathrm{d} \zeta_n^{i,j+1} \de_{1}^\mathrm{d} \big( \de_{1}^\mathrm{d} \varphi_n  - \tilde w_n  \big)^{i,j+1} +   \zeta_n^{i,j+2} \de_{1}^\mathrm{d} \Big( \de_{2}^\mathrm{d}  \big( \de_{1}^\mathrm{d} \varphi_n  - \tilde w_n  \big) \Big)^{i,j+1} \, .
        \end{split}
    \end{equation*}
    As a consequence, by suitably integrating by parts and applying a discrete product rule, the estimate of $\en R_n$ boils down to the estimate of the integrals of the following types:  
    \begin{equation*}
        \begin{aligned}
            & I_1  := \en \integral{Q'}{|\varphi_n| |\de_1^\mathrm{d} \varphi_n - \tilde w_n|}{\d x} \, , \quad 
            && I_2 :=  \en \integral{Q'}{|\nabla^\mathrm{d} \varphi_n| |\de_1^\mathrm{d} \varphi_n - \tilde w_n|}{\d x} \, , \\
            & I_3 := \en \integral{Q'}{| \varphi_n| |\de_{2}^\mathrm{d}( \de_1^\mathrm{d} \varphi_n - \tilde w_n ) |}{\d x} \, , \quad 
            && I_4 := \en \integral{Q'}{|\nabla^\mathrm{d} \varphi_n| |\de_{2}^\mathrm{d}( \de_1^\mathrm{d} \varphi_n - \tilde w_n ) |}{\d x} \, ,\\
            & I_5 := \en \integral{Q'}{|\de_2^\mathrm{d} \varphi_n| |\de_{12}^\mathrm{d}( \de_1^\mathrm{d} \varphi_n - \tilde w_n ) |}{\d x}  \, ,
        \end{aligned}
        \end{equation*}
        where we recall that $\mathrm{supp}(\zeta) \subcc Q'$. 
        
        The integrals $I_1$ and $I_2$ are easily treated with~\eqref{eq2:bounds lot} and~\eqref{eq2:de 1 phi and tilde w are close}. Indeed by H\"older's Inequality
        \begin{align*}
             I_1^2 & \leq \en  \|\varphi_n\|_{L^2(Q')}^2 \en  \|\de_1^\mathrm{d} \varphi_n - \tilde w_n\|_{L^2(Q'')}^2  \to 0 \, , \\
             I_2^2 & \leq \en \|\nabla^\mathrm{d} \varphi_n\|_{L^2(Q'')}^2 \en  \|\de_1^\mathrm{d} \varphi_n - \tilde w_n\|_{L^2(Q'')}^2 \to 0 \, .
        \end{align*}
    
        To bound $I_3$ and $I_4$, we use the control~\eqref{eq:mixed derivative for de1 - w} in Lemma~\ref{lemma:all estimates} below, i.e., 
        \begin{equation*} 
            \en \| \de_{2}^\mathrm{d}( \de_1^\mathrm{d} \varphi_n - \tilde w_n ) \|_{L^2(Q'')}^2 \leq C \, . 
        \end{equation*}
        Together with~\eqref{eq2:bounds lot}, by H\"older's Inequality this implies that 
        \begin{align*}
            I_3^2 & \leq \en  \|\varphi_n\|_{L^2(Q')}^2 \en \| \de_{2}^\mathrm{d}( \de_1^\mathrm{d} \varphi_n - \tilde w_n ) \|_{L^2(Q'')}^2  \to 0 \, , \\
            I_4^2 & \leq \en \|\nabla^\mathrm{d} \varphi_n\|_{L^2(Q'')}^2  \en \| \de_{2}^\mathrm{d}( \de_1^\mathrm{d} \varphi_n - \tilde w_n ) \|_{L^2(Q'')}^2  \to 0 \, .
       \end{align*}
    
       Finally, to estimate $I_5$ we resort to the bound~\eqref{eq:third derivative for de1 - w} in Lemma~\ref{lemma:all estimates} below, i.e.,
       \begin{equation*}
        \|\de_{12}^\mathrm{d}( \de_1^\mathrm{d} \varphi_n - \tilde w_n ) \|_{L^1(Q')} \leq C \frac{\dn}{\ln} \, .
       \end{equation*}
       Together with the $L^\infty$ bound $|\de_2^\mathrm{d} \varphi_n| \leq \frac{\pi}{\sqrt{2\dn}} $ (cf.~\eqref{eq:de 2 phi with arcsin}) and since $\arccos(1-\dn) \geq \sqrt{2 \dn}$ this yields
       \begin{equation*}
           I_5 \leq  C \en \frac{1}{\sqrt{\dn}} \frac{\dn}{\ln} \leq C \, .
       \end{equation*}
    
       The bounds on $I_1,I_2,I_3,I_4,I_5$ allow us to conclude that $\en R_n \leq C$ and, thanks to~\eqref{eq:almost finished bound}, that
       \begin{equation} \label{eq:final control on de12 phi}
           \en \|\de_{12}^\mathrm{d} \varphi_n\|_{L^2(Q)}^2 \leq C \, .
       \end{equation}
    
       We can now conclude the proof of the bound on $\en \|\de_1^\mathrm{d} z_n\|_{L^2(Q)}^2$. Thanks to~\eqref{eq:relation between z and de 2 phi} we have that 
       \begin{equation*}
        \begin{split}
        	| \de_1^\mathrm{d} z_n^{i,j} | & = \Big| \frac{z_n^{i+1,j} - z_n^{i,j}}{\ln} \Big| \\
        	& = \sqrt{\frac{2}{\dn}}\frac{1}{\ln} \Big| \sin \Big(\frac{1}{2} \arccos(1-\dn) \de_2^\mathrm{d} \varphi_n^{i+1,j} \Big) - \sin \Big( \frac{1}{2} \arccos(1-\dn) \de_2^\mathrm{d} \varphi_n^{i,j} \Big) \Big| \\
        	& \leq \frac{\arccos(1-\dn)}{\sqrt{2 \dn}} \Big| \frac{\de_2^\mathrm{d} \varphi_n^{i+1,j} - \de_2^\mathrm{d} \varphi_n^{i,j}}{\ln} \Big| \leq C | \de_{12}^\mathrm{d} \varphi_n^{i,j} | \, .
        \end{split}
        \end{equation*}
        The bound
       \begin{equation*}
        \en \|\de_1^\mathrm{d} z_n\|_{L^2(Q)}^2 \leq C 
       \end{equation*}
       then follows from~\eqref{eq:final control on de12 phi}.
     
       To control $\en \|\de_2^\mathrm{d} w_n\|_{L^2(Q)}^2$, we employ~\eqref{eq:def of tilde w} to deduce that
       \begin{equation*}
       		| \de_2^\mathrm{d} w_n^{i,j} | \leq \sqrt{\frac{2}{\dn}} \frac{\arccos(1-\dn)}{2} |\de_2^\mathrm{d} \tilde w_n^{i,j} | \, .
       \end{equation*}
       In view of~\eqref{eq:final control on de12 phi} and~\eqref{eq:mixed derivative for de1 - w} in Lemma~\ref{lemma:all estimates} below we conclude that\footnote{An alternative proof is to define  
$           \psi_n^{i,j} := \frac{1}{\arccos(1-\dn)} \bigg(\sum_{k=j_0}^{j-1} \ln (\thetv_n)^{i_0, k} +  \sum_{h=i_0}^{i-1} \ln (\theth_n)^{h,j}  \bigg) - b^n_{Q'} $
       with a suitable $b^n_{Q'} \in \RR$ and to follow the lines of the control on $\en \|\de_1^\mathrm{d} z_n\|_{L^2(Q)}^2$  with $\psi_n$ in place of $\varphi_n$. 
       }
       \begin{equation*}
       		\en \| \de_2^\mathrm{d} w_n \|_{L^2(Q)}^2 \leq  C \en \big( \| \de_{12}^\mathrm{d} \varphi_n \|_{L^2(Q)}^2 + \| \de_2^\mathrm{d} \tilde w_n - \de_{12}^\mathrm{d} \varphi_n \|_{L^2(Q'')}^2 \big) \leq C \, .
       \end{equation*}

       We prove the convergence $\curl(w_n,z_n) \weak 0$ in Lemma~\ref{lemma:curl to 0} below.
\end{proof}

In the following lemma we prove the bounds used in the proof of Proposition~\ref{prop:full MM bound}. In Step~3 we make use of a discrete Poincar\'e-Wirtinger inequality. Even though the inequality holds true under more general assumptions, for simplicity of notation in Step 3 we provide a proof in the setting we are interested in.  

\begin{lemma} \label{lemma:all estimates}
    Let $\varphi_n$   and  $Q \subcc Q' \subcc Q'' \subcc \Omega$ be as in the proof of Proposition~\ref{prop:full MM bound}. Then 
    \begin{gather}
         \en \|\varphi_n\|_{L^2(Q')}^2 + \en \| \nabla^\mathrm{d} \varphi_n \|_{L^2(Q'')}^2 \to 0  \, , \label{eq:lot} \\
         \en \| \de_{22}^\mathrm{d} \varphi_n \|_{L^2(Q'')}^2  \leq C  \label{eq:de 22 phi}\, .
    \end{gather}
    Moreover, let $\tilde w_n$ be as in~\eqref{eq:def of tilde w}. Then 
    \begin{gather} 
        \|\de_1^\mathrm{d} \varphi_n - \tilde w_n\|_{L^1(Q'')} \to 0 \, , \quad \en \|\de_1^\mathrm{d} \varphi_n - \tilde w_n\|^2_{L^2(Q'')} \to 0 \, , \label{eq:de 1 phi and tilde w are close}  \\    
        \en \|\de_1^\mathrm{d} \tilde w_n\|_{L^2(Q'')}^2 \leq C \, ,\label{eq:de 1 tilde w bounded}  \\
        \en \| \de_{2}^\mathrm{d}( \de_1^\mathrm{d} \varphi_n - \tilde w_n ) \|_{L^2(Q'')}^2 \leq C \, , \label{eq:mixed derivative for de1 - w} \\
        \|\de_{12}^\mathrm{d}( \de_1^\mathrm{d} \varphi_n - \tilde w_n ) \|_{L^1(Q')} \leq C \frac{\dn}{\ln} \, . \label{eq:third derivative for de1 - w}
    \end{gather}
\end{lemma}
\begin{proof}

\ul{Step 1}: Control on $\|\de_2^\mathrm{d} \varphi_n\|_{L^2(Q'')}$. 
By~\eqref{eq:de 2 phi with arcsin} we deduce that 
\begin{equation*}
    |\de_2^\mathrm{d} \varphi_n^{i,j}| \leq C |z_n^{i,j}| \, .
\end{equation*}
Thanks to~\eqref{eq:discrete MM bound 2} and using that $Q'' \subcc \Omega$ we infer that 
\begin{equation*}
        C \en  \geq  \ln^2 \hspace{-1.5em} \sum_{Q_{\ln}(i,j) \cap Q'' \neq \emptyset} \hspace{-1.5em}   W(z_n^{i,j}) =    \ln^2 \hspace{-1.5em} \sum_{Q_{\ln}(i,j) \cap Q'' \neq \emptyset} \hspace{-1.5em}   \big(1- (z_n^{i,j})^2\big)^2  \geq \ln^2 \hspace{-1.5em} \sum_{Q_{\ln}(i,j) \cap Q'' \neq \emptyset} \hspace{-1.5em}\big( -1 +  \tfrac{1}{2}(z_n^{i,j})^4 \big)
\end{equation*}
which yields the bound 
\begin{equation} \label{eq:L4 bound on z}
    \|\de_2^\mathrm{d} \varphi_n\|_{L^4(Q'')} \leq C \| z_n \|_{L^4(Q'')} \leq C
\end{equation}
and, in particular, the bound 
\begin{equation} \label{eq:control on de 2 phi}
    \|\de_2^\mathrm{d} \varphi_n\|_{L^2(Q'')} \leq C \, .
\end{equation}

\ul{Step 2}: Control on $\en \|\de_1^\mathrm{d} \varphi_n\|^2_{L^2(Q'')}$.  Following the lines of Step 1 with $w_n$ in place of~$z_n$ and with $\tilde w_n$ in place of~$\de_2^\mathrm{d} \varphi_n$, cf.~\eqref{eq:de 2 phi with arcsin}, we deduce that 
\begin{equation} \label{eq:bound on tilde w}
    \|\tilde w_n\|_{L^2(Q'')} \leq C \, .
\end{equation}

We prove now~\eqref{eq:de 1 phi and tilde w are close}. We start by observing that
\begin{equation*}
    \begin{split}    
       &   \de_1^\mathrm{d} \varphi_n^{i,j} - \tilde w_n^{i,j} \\
       & = \frac{1}{\arccos(1-\dn)} \bigg(   \big( (\theth_n)^{i,j_0} - (\theth_n)^{i,j}\big) + \sum_{k=j_0}^{j-1}  \big( (\thetv_n)^{i+1, k}  -   (\thetv_n)^{i, k} \big) \bigg) \, .
    \end{split}
\end{equation*}
Subtracting the previous equality evaluated at $j+1$ and the same one evaluated at $j$, we get 
\begin{equation} \label{eq:gap is due to vorticity}
    \begin{split}
        & \de_1^\mathrm{d} \varphi_n^{i,j+1} - \tilde w_n^{i,j+1} - \de_1^\mathrm{d} \varphi_n^{i,j} + \tilde w_n^{i,j}\\
        & =  \frac{1}{\arccos(1-\dn)} \bigg(    (\theth_n)^{i,j} + (\thetv_n)^{i+1, j}  - (\theth_n)^{i,j+1}  -   (\thetv_n)^{i, j}   \bigg) \, .
    \end{split}
\end{equation}
This entails that for $i$ fixed the difference $\de_1^\mathrm{d} \varphi_n^{i, \, \bigcdot \,} - \tilde w_n^{i, \, \bigcdot \,}$ changes from $j$ to $j+1$ if and only if the discrete vorticity 
\begin{equation*}      
    V_n^{i,j} :=   (\theth_n)^{i,j} + (\thetv_n)^{i+1, j}  - (\theth_n)^{i,j+1}  -   (\thetv_n)^{i, j}   
\end{equation*}
is non-zero.  Note that $\de_1^\mathrm{d} \varphi_n^{i,j_0} = \tilde w_n^{i,j_0}$, so that $\de_1^\mathrm{d} \varphi_n^{i, \, \bigcdot \,} \equiv \tilde w_n^{i, \, \bigcdot \,}$ if $V_n^{i, \, \bigcdot \,} \equiv 0$.

We count how many times $V_n^{i,j}$ is different from zero. First of all, we remark that~$V_n^{i,j} \in \{-2\pi,0,2\pi\}$ (indeed, it is a multiple of $2\pi$ and $\theth_n, \thetv_n \in [-\pi,\pi)$). Next we observe that $|V_n^{i,j}| = 2\pi$ implies that one of the four values $|(\theth_n)^{i,j}|$, $|(\thetv_n)^{i+1, j}|$, $|(\theth_n)^{i,j+1}|$, $|(\thetv_n)^{i, j}|$ must be greater than or equal to $\frac{\pi}{2}$. This observation leads us to a counting argument analogous to the one done in~\eqref{eq:counting theta}. Hence  we conclude that 
\begin{equation} \label{eq:number of bad for V}
    \# \{(i,j) \ : \ Q_{\ln}(i,j) \cap Q'' \neq \emptyset \text{ and } V_n^{i,j} \neq 0 \} \leq C \frac{\dn^{3/2}}{\ln} \, .
\end{equation}

Let us now fix $i$. We need to estimate the distance between $\de_1^\mathrm{d} \varphi_n^{i, \, \bigcdot \,}$ and $\tilde w_n^{i,\, \bigcdot \,}$ if it occurs that $V_n^{i,j} \neq 0$ for some $j$. By~\eqref{eq:gap is due to vorticity} we deduce that $|V_n^{i,j}|=2\pi$ implies
\begin{equation*}
    \begin{split}
         \big|\de_1^\mathrm{d} \varphi_n^{i,j+1} - \tilde w_n^{i,j+1}\big| & \leq  \big|\de_1^\mathrm{d} \varphi_n^{i,j} - \tilde w_n^{i,j} \big| + \frac{1}{\arccos(1-\dn)} | V_n^{i,j} | \\
        & \leq \big|\de_1^\mathrm{d} \varphi_n^{i,j} - \tilde w_n^{i,j} \big| + \frac{2\pi}{\sqrt{2\dn}} \, ,
    \end{split}
\end{equation*}
for $n$ large enough. As already stressed, $|V_n^{i,j}|=0$ implies instead that  $|\de_1^\mathrm{d} \varphi_n^{i,j+1} - \tilde w_n^{i,j+1}|  = |\de_1^\mathrm{d} \varphi_n^{i,j} - \tilde w_n^{i,j} |$.  With a chain of inequalities (that starts for the maximal $j$ such that $Q_{\ln}(i,j) \cap Q'' \neq \emptyset$ and stops when $j=j_0$) and using the fact that $\de_1^\mathrm{d} \varphi_n^{i,j_0} = \tilde w_n^{i,j_0}$, we conclude that 
\begin{equation*}
   \sup_j \big|\de_1^\mathrm{d} \varphi_n^{i,j} - \tilde w_n^{i,j}\big| \leq \frac{2 \pi}{\sqrt{2 \dn}}    \# \{ j \ : \ Q_{\ln}(i,j) \cap Q'' \neq \emptyset \text{ and } V_n^{i,j} \neq 0 \} \, .
\end{equation*}
Together with~\eqref{eq:number of bad for V}, this allows us to estimate the $L^1$ norm as follows 
\begin{equation*}
    \begin{split}
        \| \de_1^\mathrm{d} \varphi_n  - \tilde w_n  \|_{L^1(Q'')} & \leq \ln^2  \sum_{(i,j)}    \big| \de_1^\mathrm{d} \varphi_n^{i,j}  - \tilde w_n^{i,j}  \big|  \leq C \ln \hspace{-1em}\sum_{\substack{i \\ V_n^{i,j^*} \neq 0 \\ \text{for some } j^*}} \hspace{-1em} \sup_j \big| \de_1^\mathrm{d} \varphi_n^{i,j}  - \tilde w_n^{i,j}  \big| \\
        & \leq C \ln \hspace{-1em} \sum_{\substack{i \\ V_n^{i,j^*} \neq 0 \\ \text{for some } j^*}} \hspace{-1em} \frac{2 \pi}{\sqrt{2 \dn}}   \# \{ j \ : \  V_n^{i,j} \neq 0 \}  \leq C \frac{\ln}{\sqrt{\dn}} \#\{ (i,j) \ : \ V_n^{i,j} \neq 0\} \\
        & \leq  C  \dn \, ,
    \end{split}
\end{equation*}
the sums being always taken for $(i,j)$ such that $Q_{\ln}(i,j) \cap Q'' \neq \emptyset$. Analogously, by the superadditivity of $t \mapsto t^2$ in $[0,+\infty)$ and since $\en = \frac{\ln}{\sqrt{2 \dn}}$ we get that 
\begin{equation*}
    \| \de_1^\mathrm{d} \varphi_n  - \tilde w_n  \|^2_{L^2(Q'')}  \leq C \frac{\ln}{ \dn} \big( \#\{ (i,j) \ : \ V_n^{i,j} \neq 0\} \big)^2 \leq C \frac{\dn^{2}}{\ln} \leq C \frac{\dn^{3/2}}{\en} \, .
\end{equation*}
This concludes the proof of the claim~\eqref{eq:de 1 phi and tilde w are close}. We remark that \eqref{eq:bound on tilde w} and \eqref{eq:de 1 phi and tilde w are close} imply that 
\begin{equation} \label{eq:control on de 1 phi}
    \en \|\de_1^\mathrm{d} \varphi_n\|^2_{L^2(Q'')} \to 0 \, .
\end{equation}

\ul{Step 3}: Control on $\en \| \varphi_n\|^2_{L^2(Q')}$. We claim that 
\begin{equation} \label{eq:control on phi}
    \en \| \varphi_n\|^2_{L^2(Q')} \to 0 \, ,
\end{equation}
where we recall that $Q'$ is such that $\mathrm{supp}(\zeta) \subcc Q'$. To prove the claim we resort to the following discrete Poincar\'e-Wirtinger inequality: for $n$ large enough we have that
\begin{equation} \label{eq:discrete Poincare-Wirtinger}
    \|\varphi_n\|_{L^2(Q')} \leq C \| \nabla^\mathrm{d} \varphi_n \|_{L^2(Q'')} \, .
\end{equation}
This directly implies~\eqref{eq:control on phi} by~\eqref{eq:control on de 1 phi} and~\eqref{eq:control on de 2 phi}.

To prove the discrete Poincar\'e-Wirtinger inequality, we apply the well-known classical inequality to the piecewise affine interpolation $\hat \varphi_n$ of $\varphi_n$ defined as follows. Let $T^-_{\ln}(i,j)$ and $T^+_{\ln}(i,j)$ be the two triangles partitioning the square $Q_{\ln}(i,j)$ defined by 
\begin{align*}
    T^-_{\ln}(i,j) & := \{ \ln (i + s, j + t) \in Q_{\ln}(i,j)  \ : \ s \in [0,1] \, , t \in [0,1-s]\} \, , \\
    T^+_{\ln}(i,j) & := \{ \ln (i + s, j + t) \in Q_{\ln}(i,j) \ : \ s \in (0,1) \, , t \in (1-s,1) \} \, .
\end{align*}
The function $\hat \varphi_n$ is defined in $T^-_{\ln}(i,j)$ by interpolating the values of $\varphi_n$ on the three vertices of $T^-_{\ln}(i,j)$, i.e., 
\begin{equation} \label{eq:affine interpolation 1}
    \hat \varphi_n(x) := \varphi_n^{i,j} + (x_1 - \ln i) \frac{\varphi_n^{i+1,j} - \varphi_n^{i,j}}{\ln} + (x_2 - \ln j) \frac{\varphi_n^{i,j+1} - \varphi_n^{i,j}}{\ln} \, ,
\end{equation}
for every $x = (x_1,x_2) \in T^-_{\ln}(i,j)$. Analogously, 
\begin{equation}  \label{eq:affine interpolation 2}
    \begin{split}
        \hat \varphi_n(x)  := \varphi_n^{i,j+1} & + (x_1 - \ln i) \frac{\varphi_n^{i+1,j+1} - \varphi_n^{i,j+1}}{\ln} \\
        & + (x_2 - \ln (j+1)) \frac{\varphi_n^{i+1,j+1} - \varphi_n^{i+1,j}}{\ln} \, ,
    \end{split}
\end{equation}
for every $x = (x_1,x_2) \in T^+_{\ln}(i,j)$. 

With the definition of $\hat \varphi_n$ at hand, we can choose the value of $a^n_{Q'}$ in~\eqref{eq:def of phi}. We remark that adding a constant to $\varphi_n$ results into the addition of the same constant to the average $\frac{1}{|Q'|} \int_{Q'} \hat \varphi_n$, cf.~\eqref{eq:affine interpolation 1}--\eqref{eq:affine interpolation 2}. Therefore we can choose $a^n_{Q'}$ so that $ \frac{1}{|Q'|} \int_{Q'} \hat \varphi_n = 0$.  

The gradient of $\hat \varphi_n$ is constant in $T^\pm_{\ln}(i,j)$ and is given by
\begin{equation*}  
    \begin{aligned}
        \nabla \hat \varphi_n &=  \big(\de_1^\mathrm{d} \varphi_n^{i,j}, \de_2^\mathrm{d} \varphi_n^{i,j} \big) \quad \text{in } T^-_{\ln}(i,j) \, , \\
        \nabla \hat \varphi_n &= \big(\de_1^\mathrm{d} \varphi_n^{i,j+1}, \de_2^\mathrm{d} \varphi_n^{i+1,j} \big) \quad \text{in } T^+_{\ln}(i,j) \, .
    \end{aligned}
\end{equation*}
In particular, for $n$ is large enough,
\begin{equation} \label{eq:gradient of affine controlled by pc}        
    \| \nabla \hat \varphi_n \|_{L^2(Q')} \leq C \| \nabla^\mathrm{d} \varphi_n \|_{L^2(Q'')} \, .
\end{equation}

We can control the $L^2$ distance between the piecewise constant function $\varphi_n$ and the piecewise affine interpolation $\hat \varphi_n$ as follows. By~\eqref{eq:affine interpolation 1} we obtain that 
\begin{equation*}
    \hat \varphi_n(x) -  \varphi_n^{i,j} =   (x_1 - \ln i) \de_1^\mathrm{d} \varphi_n^{i,j} + (x_2 - \ln j) \de_2^\mathrm{d} \varphi_n^{i,j}
\end{equation*}
for $x \in T^-_{\ln}(i,j)$, yielding  
\begin{equation*}
    \| \hat \varphi_n - \varphi_n \|_{L^2(T^-_{\ln}(i,j))} \leq  C \ln \| \nabla^\mathrm{d} \varphi_n\|_{L^2(T^-_{\ln}(i,j))} \, .
\end{equation*}
Analogously, by~\eqref{eq:affine interpolation 2} for $x \in T^+_{\ln}(i,j)$ we have that 
\begin{equation*}
    \begin{split}
        \hat \varphi_n(x) - \varphi_n^{i,j}= \ln \de_2^\mathrm{d}\varphi_n^{i,j} & + (x_1 - \ln i) \de_1^\mathrm{d} \varphi_n^{i,j+1} \\
        & + (x_2 - \ln (j+1)) \de_2^\mathrm{d} \varphi_n^{i+1,j} \, ,
    \end{split}
\end{equation*}
and thus
\begin{equation*}
    \begin{split}
       & \|  \hat \varphi_n - \varphi_n \|_{L^2(T^+_{\ln}(i,j))}\\
       &\quad  \leq C \ln \big( \| \nabla^\mathrm{d} \varphi_n\|_{L^2(T^+_{\ln}(i,j))} + \|  \de_1^\mathrm{d} \varphi_n \|_{L^2(T^+_{\ln}(i,j+1))} + \| \de_2^\mathrm{d} \varphi_n \|_{L^2(T^+_{\ln}(i+1,j))} \big)  \, .
    \end{split}
\end{equation*}
In conclusion, we get 
\begin{equation} \label{eq:affine and constant}
    \|  \hat \varphi_n - \varphi_n \|_{L^2(Q')} \leq C  \ln  \| \nabla^\mathrm{d} \varphi_n \|_{L^2(Q'')}
\end{equation}
and thus $ \| \varphi_n \|_{L^2(Q')} \leq   \| \hat \varphi_n \|_{L^2(Q')} +   \| \nabla^\mathrm{d} \varphi_n \|_{L^2(Q'')}$ for $n$ large enough. Since $\hat \varphi_n$ has zero average in $Q'$, by the Poincar\'e-Wirtinger inequality and by~\eqref{eq:gradient of affine controlled by pc} we deduce that 
\begin{equation*}
    \begin{split}
        \| \varphi_n \|_{L^2(Q')} & \leq   \| \hat \varphi_n \|_{L^2(Q')} +  \| \nabla^\mathrm{d} \varphi_n \|_{L^2(Q'')}  \leq  C \| \nabla \hat \varphi_n \|_{L^2(Q')} +  \| \nabla^\mathrm{d} \varphi_n \|_{L^2(Q'')}  \\
        & \leq  C \| \nabla^\mathrm{d} \varphi_n \|_{L^2(Q'')} \, ,
    \end{split}
\end{equation*}
which proves~\eqref{eq:discrete Poincare-Wirtinger}.

By~\eqref{eq:control on phi} together with~\eqref{eq:control on de 2 phi} and~\eqref{eq:control on de 1 phi}, we conclude the proof of~\eqref{eq:lot}.

\ul{Step 4}: Control on $\en \|\de_{22}^\mathrm{d} \varphi_n\|^2_{L^2(Q'')}$.  Unfortunately this does not follow immediately from~\eqref{eq:de 2 phi with arcsin}  and the energy bound~\eqref{eq:discrete MM bound 2}, since the function $\arcsin$ is   only locally Lipschitz on the interval $(-1,1)$. However, we can estimate the number of points $(i,j)$ for which $\sqrt{\tfrac{\dn}{2}} z_n^{i,j}$ is close to~$1$ or~$-1$. More precisely, if $|z_n^{i,j}| \geq \tfrac{1}{2}\sqrt{\tfrac{2}{\dn}}$, then (cf.~\eqref{eq:order parameter}) $|(\thetv_n)^{i,j}| = 2 \arcsin \big( \sqrt{\tfrac{\dn}{2}} |z_n^{i,j}| \big) \geq \tfrac{\pi}{3}$. With a counting argument analogous to that in~\eqref{eq:counting theta}   we obtain that 
\begin{equation*}
    \# \big\{ (i,j) \in \I^n(\Omega)  \ : \  \big|(\thetv_n)^{i,j}\big| \geq \tfrac{\pi}{3} \big\} \leq C \frac{\dn^{3/2}}{\ln}  \, .
\end{equation*}
Analogously, if $|z_n^{i,j+1}| \geq \tfrac{1}{2}\sqrt{\tfrac{2}{\dn}}$, then $|(\thetv_n)^{i,j+1}| \geq \tfrac{\pi}{3}$. Furthermore we have that 
\begin{equation*}
    \# \big\{ (i,j) \in \I^n(\Omega)  \ : \  \big|(\thetv_n)^{i,j+1}\big| \geq \tfrac{\pi}{3} \big\} \leq C \frac{\dn^{3/2}}{\ln}  \, .
\end{equation*}
In both cases, from~\eqref{eq:de 2 phi with arcsin} we get the rough bound $|\de_2^\mathrm{d} \varphi_n^{i,j}| \leq \frac{2}{\arccos(1-\dn)} \frac{\pi}{2}$ and thus
\begin{equation*}
\big| \de_{22}^\mathrm{d} \varphi_n^{i,j}\big| \leq \frac{C}{\ln \sqrt{\dn}} \, ,
\end{equation*}
for $n$ large enough.

If, instead, $|z_n^{i,j}| < \tfrac{1}{2}\sqrt{\tfrac{2}{\dn}}$ and  $|z_n^{i,j+1}| < \tfrac{1}{2}\sqrt{\tfrac{2}{\dn}}$ then~\eqref{eq:de 2 phi with arcsin} yields for $n$ large enough
\begin{equation*}
    \big|\de_{22}^\mathrm{d} \varphi_n^{i,j} \big| \leq \frac{2}{\arccos(1-\dn)} \sqrt{\frac{\dn}{2}}  \frac{2}{\sqrt{3}}   \big|\de_2^\mathrm{d} z_n^{i,j} \big|  \leq C  \big|\de_2^\mathrm{d} z_n^{i,j} \big| \, ,
\end{equation*}
where we used the fact that $\arcsin$ is Lipschitz in $[-\tfrac{1}{2},\tfrac{1}{2}]$ with Lipschitz constant $\tfrac{2}{\sqrt{3}}$.

We now split the sum in
\begin{equation*}
    \|\de_{22}^\mathrm{d} \varphi_n\|^2_{L^2(Q'')} \leq  \ln^2 \hspace{-1.5em} \sum_{\substack{Q_{\ln}(i,j) \cap Q'' \neq \emptyset}} \hspace{-1.5em}   \big| \de_{22}^\mathrm{d} \varphi_n^{i,j} \big|^2 
\end{equation*}
 into the sum over indices such that 
\begin{equation*}
    |z_n^{i,j}| \geq \frac{1}{2}\sqrt{\frac{2}{\dn}} \ \text{ or }\  |z_n^{i,j+1}| \geq \frac{1}{2}\sqrt{\frac{2}{\dn}}
\end{equation*}
and the sum over indices such that 
\begin{equation*}
    |z_n^{i,j}| < \frac{1}{2}\sqrt{\frac{2}{\dn}} \ \text{ and }\  |z_n^{i,j+1}| < \frac{1}{2}\sqrt{\frac{2}{\dn}} \, .
\end{equation*}
Putting together the estimates obtained in this step we conclude that 
\begin{equation*}
    \en \|\de_{22}^\mathrm{d} \varphi_n\|^2_{L^2(Q'')}  \leq C \en \frac{\ln^2}{\ln^2 \dn} \frac{\dn^{3/2}}{\ln} + C \en \ln^2 \hspace{-1.5em} \sum_{\substack{Q_{\ln}(i,j) \cap Q'' \neq \emptyset}} \hspace{-1.5em}   \big| \de_{2}^\mathrm{d}z_n^{i,j} \big|^2 \leq C \bigg( 1 + \en \ln^2 \hspace{-1.5em} \sum_{\substack{(i,j) \in \I^n(\Omega)}} \hspace{-1.5em}   \big| \de_{2}^\mathrm{d}z_n^{i,j} \big|^2 \bigg) \, .
\end{equation*}
Thanks to the energy bound~\eqref{eq:discrete MM bound 2}, we conclude the proof of~\eqref{eq:de 22 phi}.

\ul{Step 5}: Control on $\en \|\de_1^\mathrm{d} \tilde w_n\|_{L^2(Q'')}$. Following the lines of Step 4 with $w_n$ in place of~$z_n$, $\tilde w_n$ in place of $\de_2^\mathrm{d} \varphi_n$, and $\de_1^\mathrm{d} \tilde w_n$ in place of $\de_{22}^\mathrm{d} \varphi_n$, we deduce the bound~\eqref{eq:de 1 tilde w bounded}.

\ul{Step 6}: Proof of \eqref{eq:mixed derivative for de1 - w}. The equality in~\eqref{eq:gap is due to vorticity} can be recast as follows 
        \begin{equation} \label{eq:second derivative ist V}
            \de_2^\mathrm{d} \big( \de_1^\mathrm{d} \varphi_n - \tilde w_n \big)^{i,j} = \frac{1}{\ln \arccos(1-\dn)} V_n^{i,j} \, .
        \end{equation}
        Recalling from Step~2 that $|V_n^{i,j}| \leq 2\pi$, together with the estimate~\eqref{eq:number of bad for V} on the number of points where $V_n^{i,j} \neq 0$ the previous equality yields 
        \begin{equation*}
            \en \| \de_{2}^\mathrm{d}( \de_1^\mathrm{d} \varphi_n - \tilde w_n ) \|_{L^2(Q'')}^2 \leq \en \frac{\dn^{3/2}}{\ln}  \ln^2 \frac{C}{\ln^2 \dn } \leq C \, ,
        \end{equation*}  
        namely~\eqref{eq:mixed derivative for de1 - w}.

\ul{Step 7}:  Proof of \eqref{eq:third derivative for de1 - w}. By~\eqref{eq:second derivative ist V} and~\eqref{eq:number of bad for V} we have that
       \begin{equation*}
        \begin{split}
            \|\de_{12}^\mathrm{d}( \de_1^\mathrm{d} \varphi_n - \tilde w_n ) \|_{L^1(Q')} & \leq \frac{2}{\ln} \sum_{Q_{\ln}(i,j) \cap Q'' \neq \emptyset} \hspace{-1em} \ln^2 | \de_{2}^\mathrm{d}( \de_1^\mathrm{d} \varphi_n  - \tilde w_n )^{i,j}| \\
            & \leq 2 \sum_{V_n^{i,j} \neq 0}  \ln \frac{2\pi}{\ln \arccos(1-\dn)} \leq  C \frac{\dn^{3/2}}{\ln} \frac{1}{ \sqrt{2 \dn}}  \\
            & \leq C \frac{\dn}{\ln} \, .
        \end{split}
       \end{equation*}
\end{proof}

\begin{lemma} \label{lemma:curl to 0}
    Let $(w_n,z_n)$ and $Q$ be as in Proposition~\ref{prop:full MM bound}. Then 
    \begin{equation} \label{eq2:curl to 0}
        \curl(w_n,z_n) \weak 0  \quad \text{ in } \mathcal{D}'(Q) \, .     
    \end{equation}
\end{lemma}
\begin{proof}
Let $\varphi_n$ be as in the proofs of Proposition~\ref{prop:full MM bound} and Lemma~\ref{lemma:all estimates}.  The idea of the proof is to make use of the fact that $(w_n,z_n)$ is close to a $\curl$-free vector field on $Q$. Specifically, we shall prove that $ (w_n,z_n) - \nabla \hat \varphi_n \to 0$ in $L^1(Q;\RR^2)$, where $\hat \varphi_n$ is the piecewise affine interpolation defined in~\eqref{eq:affine interpolation 1}--\eqref{eq:affine interpolation 2}. Let us start by proving that 
\begin{equation} \label{eq:gradient of pa close to discrete}
    \|\de_1 \hat \varphi_n - \de_1^\mathrm{d} \varphi_n\|_{L^1(Q)} \to 0 \, , \quad \|\de_2 \hat \varphi_n -\de_2^\mathrm{d} \varphi_n\|_{L^1(Q)} \to 0 \, .
\end{equation}
Following the notation introduced in Step~3 of Lemma~\ref{lemma:all estimates} we have that 
\begin{equation*}
 \nabla \hat \varphi_n(x) - \nabla^\mathrm{d} \varphi_n^{i,j} = 0
 \end{equation*}
 for $x \in T^-_{\ln}(i,j)$. In contrast, for $x \in T^+_{\ln}(i,j)$ we have 
 \begin{equation*}
    \nabla \hat \varphi_n(x) - \nabla^\mathrm{d} \varphi_n^{i,j} = \ln  (\de_{12}^\mathrm{d}  \varphi_n^{i,j},\de_{12}^\mathrm{d}  \varphi_n^{i,j}) \, .
 \end{equation*}
 It follows that 
 \begin{equation*}
     \|  \nabla \hat \varphi_n - \nabla^\mathrm{d} \varphi_n \|_{L^2(Q)}^2 \leq  \ln^2 \| \de_{12}^\mathrm{d}  \varphi_n \|_{L^2(Q)}^2 \leq C  \ln \sqrt{\dn}  \to 0\, ,
 \end{equation*}
 where in the last inequality we used~\eqref{eq:claim on control on mixed derivative} and the fact that $\en = \frac{\ln}{\sqrt{2 \dn}}$.  

 Next, we show that 
 \begin{equation} \label{eq:discrete gradient close to w and z}
     \|w_n - \de_1^\mathrm{d} \varphi_n\|_{L^1(Q)} \to 0 \, , \quad \|z_n - \de_2^\mathrm{d} \varphi_n\|_{L^1(Q)} \to 0 \, .
 \end{equation}
  From~\eqref{eq:relation between z and de 2 phi} and since $|\sin x - x \, | \leq C |x|^3$ it follows that 
 \begin{equation*}
     \begin{split}    
         \Big| z_n^{i,j} - \de_2^\mathrm{d} \varphi_n^{i,j} \Big| & \leq \sqrt{\frac{2}{\dn}} \Big|  \sin \Big( \frac{\arccos(1-\dn) }{2} \de_2^\mathrm{d} \varphi_n^{i,j} \Big) - \frac{\arccos(1-\dn)}{2}  \de_2^\mathrm{d} \varphi_n^{i,j}  \Big| \\
        & \hphantom{\leq \sqrt{\frac{2}{\dn}} } + \Big|  \frac{ \arccos(1-\dn) }{\sqrt{2\dn}}\de_2^\mathrm{d} \varphi_n^{i,j} - \de_2^\mathrm{d} \varphi_n^{i,j} \Big| \\
        & \leq  C \sqrt{\frac{2}{\dn}} \Big|   \frac{1}{2} \arccos(1-\dn) \de_2^\mathrm{d} \varphi_n^{i,j}  \Big|^3  + \Big| \frac{\arccos(1-\dn)}{\sqrt{2 \dn}}  -1 \Big| |\de_2^\mathrm{d} \varphi_n^{i,j}|  \, .
     \end{split}
 \end{equation*}
 By~\eqref{eq:L4 bound on z} we obtain that 
 \begin{equation*}
     \begin{split}
         \| z_n - \de_2^\mathrm{d} \varphi_n \|_{L^1(Q)} & \leq C \frac{\dn^{3/2}}{\sqrt{\dn}}   \| \de_2^\mathrm{d} \varphi_n  \|_{L^3(Q'')}^3  + \Big| \frac{\arccos(1-\dn)}{\sqrt{2 \dn}}  -1 \Big| \|\de_2^\mathrm{d} \varphi_n\|_{L^1(Q'')} \\
         & \leq C \Big( \dn + \Big| \frac{\arccos(1-\dn)}{\sqrt{2 \dn}}  -1 \Big| \Big) \to 0 \, .
     \end{split}
 \end{equation*}
Following the same lines by replacing $z_n$ with $w_n$ and $\de_2^\mathrm{d} \varphi_n$ with the auxiliary variable~$\tilde w_n$ defined in~\eqref{eq:def of tilde w}, we deduce that $\| w_n - \tilde w_n \|_{L^1(Q)} \to 0$. Then~\eqref{eq:de 1 phi and tilde w are close} yields $\| w_n - \de_1^\mathrm{d} \varphi_n \|_{L^1(Q)} \to 0$.

 Putting together~\eqref{eq:gradient of pa close to discrete} and~\eqref{eq:discrete gradient close to w and z} we infer that 
 \begin{equation*} 
     \|\de_1 \hat \varphi_n - w_n \|_{L^1(Q)} \to 0 \, , \quad \|\de_2 \hat \varphi_n - z_n\|_{L^1(Q)} \to 0 \, .
 \end{equation*}
We are now in a position to prove~\eqref{eq2:curl to 0}. Since $\curl ( \nabla \hat \varphi_n) = 0$ in~$\mathcal{D}'(Q)$, for every $\xi \in C^\infty_c(Q)$ we obtain that
\begin{equation*}
 \begin{split}
       \langle \curl(w_n,z_n) , \xi \rangle  & =  - \integral{Q}{  w_n \, \de_2 \xi}{\d x} + \integral{Q}{  z_n \, \de_1 \xi}{\d x} \\
     & = \integral{Q}{  (\de_1 \hat \varphi_n - w_n)\, \de_2 \xi}{\d x} - \integral{Q}{  (  \de_2 \hat \varphi_n - z_n)\, \de_1 \xi}{\d x} \to 0 \, .
 \end{split}
\end{equation*}
\end{proof}

The bounds proven in Proposition~\ref{prop:full MM bound} allow us to  obtain a first partial compactness result. 

\begin{proposition} \label{prop:partial compactness}
    Let $(w_n,z_n) \in \Lloc(\RR^2;\RR^2)$ be a sequence satisfying $H_n(w_n,z_n;\Omega) \leq C$. Then there exists $(w,z) \in BV_{\mathrm{loc}}(\Omega;\RR^2)$ such that $w \in \{1, -1\}$, $z \in \{1, -1\}$ a.e.\ in $\Omega$, $\curl(w,z) = 0$ in $\mathcal{D}'(\Omega)$, and (up to a subsequence) $(w_n,z_n) \to (w,z)$ in $\Lloc(\Omega;\RR^2)$. 
\end{proposition}

\begin{remark}
    The result stated in the proposition above is still partial and does not yet give the compactness stated in Theorem~\ref{thm:main}-(i). Indeed, in Theorem~\ref{thm:main}-(i) it is stated, in particular, that $(w,z) \in BV(\Omega;\RR^2)$, while Proposition~\ref{prop:partial compactness} only guarantees that $(w,z) \in BV_{\mathrm{loc}}(\Omega;\RR^2)$. We shall see in the proof of Theorem~\ref{thm:main}-(i) below how the  liminf   inequality   will be exploited to deduce the stronger condition $(w,z) \in BV(\Omega;\RR^2)$.
\end{remark}
\begin{proof}[Proof of Proposition~\ref{prop:partial compactness}] We divide the proof into two steps. We start by proving compactness of $w_n$ and $z_n$ in $L^1(Q;\RR^2)$ for every  open  square $Q \subcc \Omega$ and then we obtain compactness in $\Lloc(\Omega;\RR^2)$ via a diagonal procedure.
    
    \ul{Step 1}: Local compactness. Let us fix two  open squares $Q \subcc Q' \subcc \Omega$.  To prove compactness we rely on the compactness results known for the Modica-Mortola functional.  As in~\cite[Section~3.1]{Bra-Yip}, we compare the discrete Modica-Mortola energy of $w_n$ with the classical Modica-Mortola energy  of a suitable Lipschitz modification of $w_n$. To apply the result contained in~\cite[Section~3.1]{Bra-Yip}, we need to work with bounded functions. For this reason we define the truncated functions $\ol w_n := \max\{ \min \{ w_n, 1 \}, -1\}$.  We observe that~\eqref{eq:discrete full MM bound 1} applied to the square $Q'$ yields 
    \begin{equation} \label{eq:truncation is close}
        \begin{split}
            \|\ol w_n - w_n\|_{L^1(Q')} & \leq \hspace{-1em} \integral{Q' \cap \{|w_n| > 1\}}{  \hspace{-1em}|\ol w_n - w_n|  \big(1+|w_n|\big)}{\d x} =   \hspace{-1em} \integral{Q' \cap \{|w_n| > 1\}}{  \hspace{-1em}|\ol w_n - w_n| |\ol w_n + w_n|}{\d x} \\
            & =  \hspace{-1em} \integral{Q' \cap \{|w_n| > 1\}}{  \hspace{-1em}|1 - (w_n)^2| }{\d x} \leq C \bigg( \integral{Q'}{ W(w_n) }{\d x} \bigg)^{\! 1/2} \leq C \en^{1/2} \to 0 \, .
        \end{split}
    \end{equation}
    Moreover 
    \begin{equation*}
        \integral{Q'}{ W(\ol w_n) }{\d x} =  \hspace{-1em}  \integral{Q' \cap \{|w_n| \leq 1\}}{ \hspace{-1em}  W(w_n) }{\d x} \leq \integral{Q'}{ W(w_n) }{\d x} 
    \end{equation*}
    and $|\nabla^\mathrm{d} \ol w_n| \leq |\nabla^\mathrm{d} w_n|$, from which it follows that 
    \begin{equation} \label{eq:0603191809}
        \frac{1}{ \en}  \integral{Q'}{W(\ol w_n )}{\d x}  + \en \integral{Q'}{ \big| \nabla^{\mathrm{d}} \ol w_n  \big|^2}{\d x} \leq C  \, .
    \end{equation}

    We are now in a position to apply the following observation made in~\cite[Section~3.1]{Bra-Yip} (applied here for $\delta = \ln$): if $\hat w_n$ denotes the piecewise affine interpolation of $\ol w_n$ defined as in~\eqref{eq:affine interpolation 1}--\eqref{eq:affine interpolation 2}, then~\eqref{eq:0603191809} implies
        \begin{equation} \label{eq:0603191829}
        \frac{1}{ \en}  \integral{Q}{W(\hat w_n )}{\d x}  + \en \integral{Q}{ \big| \nabla \hat w_n  \big|^2}{\d x} \leq C
    \end{equation}
    for $n$ large enough (with a possibly larger constant $C$). We can now apply to the sequence $\hat w_n \in W^{1,2}(Q)$ the well-known theory of the classical Modica-Mortola  functional~\cite{ModMor,Mod}  
    and deduce that there exists a subsequence (that we do not relabel) and a function $w \in BV(Q)$, $w \in \{1, -1\}$ a.e.\ in $\Omega$, such that $\hat w_n \to w$ in $L^1(Q)$. 
    As already done for the proof of~\eqref{eq:affine and constant}, it is possible to prove that
    \begin{equation*}
        \|\hat w_n - \ol w_n\|_{L^2(Q)}^2 \leq C \ln^2 \|\nabla^\mathrm{d} \ol w_n\|_{L^2(Q')}^2 \leq C \ln \sqrt{\dn} \, ,
    \end{equation*}
    where in the last inequality we applied~\eqref{eq:0603191809} and the fact that $\en = \frac{\ln}{\sqrt{2 \dn}}$. Together with~\eqref{eq:truncation is close}, this implies that $w_n \to w$ in $L^1(Q)$. 

    We follow the previous argument with~$z_n$ in place of $w_n$ to conclude that $z_n \to z$ in $L^1(Q)$ (up to a not relabeled subsequence) for some $z \in BV(Q)$, $z \in \{1, - 1\}$ a.e.\ in $\Omega$. 

    From~\eqref{eq:curl to 0} it follows that $\curl(w,z) = 0$ in $\mathcal{D}'(Q)$.  In conclusion, $(w_n,z_n) \to (w,z)$ in $L^1(Q;\RR^2)$ with $(w,z) \in \Dom(H;Q)$. 

    \ul{Step 2}: Global compactness. Let $(Q_j)_j$ be a countable family of (overlapping) open squares with sides parallel to the coordinate axes such that $\Omega = \bigcup_{j} Q_j$. With a diagonal argument we extract a subsequence from~$n$ (which we do not relabel) such that for every $j$ we have $(w_n,z_n) \to (w,z)$ in $L^1(Q_j;\RR^2)$ with $(w,z) \in \Dom(H;Q_j)$. In particular, $(w_n,z_n) \to (w,z)$ in $\Lloc(\Omega;\RR^2)$ and $w \in \{1, -1 \}$, $z \in \{1, -1\}$ a.e.\ in $\Omega$. From a standard partition of unity argument we conclude that $(w,z) \in BV_{\mathrm{loc}}(\Omega;\RR^2)$ and $\curl(w,z) = 0$ in~$\mathcal{D}'(\Omega)$.
\end{proof}

\begin{proof}[Proof of Theorem~\ref{thm:main}-(ii)]
    Let $(w_n,z_n) \in \Lloc(\RR^2;\RR^2)$ be a sequence such that $(w_n, z_n) \to (w,z)$ in $\Lloc(\Omega;\RR^2)$ and let us prove that 
    \begin{equation} \label{eq:liminf}
        H(w,z;\Omega) \leq \liminf_{n \to \infty} H_n(w_n,z_n;\Omega) \, .
    \end{equation}
    Up to the extraction of a subsequence (that we do not relabel) we can assume that the $\liminf$ in the right-hand side is actually a limit and that it is finite (the inequality being otherwise trivial). In particular, $H_n(w_n,z_n;\Omega) \leq C$ for some constant $C$. By Proposition~\ref{prop:partial compactness} we infer that $(w,z) \in BV_{\mathrm{loc}}(\Omega;\RR^2)$, $w \in \{1,-1\}$ and $z \in \{1, -1\}$ a.e.\ in $\Omega$, and $\curl(w,z) = 0$ in~$\mathcal{D}'(\Omega)$. 
 
    To prove the result we resort to a slicing argument in the coordinate directions, which reduces the problem to the one-dimensional setting  studied in~\cite{CicSol}. Given a function $u(x_1,x_2)$ defined for $(x_1,x_2) \in \RR^2$ we define $u^{x_2}(x_1) := u(x_1,x_2)$ and given a set $A$ we consider its slices $A^{x_2} := \{x_1 \in \RR \ : \ (x_1,x_2) \in A\}$. If $u$ is a piecewise constant function, then we also adopt the notation $(u^{x_2})^i := u^{x_2}(\ln i)$, identifying functions defined on the lattice $\ln \ZZ$ with piecewise constant functions on the real line. For a given open set $U \subset \RR$, we introduce the set of indices  
    \begin{equation*}
        \I^n_1(U)  := \{ i \in \ZZ \ : \ \ln i + [0,\ln] \, , \ln (i+1) + [0,\ln] \subset U \} \, ,
    \end{equation*}
    and the one-dimensional discrete energy
    \begin{equation*}
        H_n^{\mathrm{1d}}(v;U) := \frac{1}{\sqrt{2} \ln \dn^{3/2} }  \frac{1}{2}  \ \ln  \sum_{i \in \I^n_1(U)} \Big| v^{i+2} - \frac{\alpha_n}{2} v^{i+1}  + v^{i} \Big|^2 
    \end{equation*}
    defined for $v \colon \ln \ZZ \to \SS^1$. 

    Let $u_n \in \PC_{\ln}(\SS^1)$ be such that $(w_n,z_n) = T_n(u_n)$ and let us fix an open set $A \subcc \Omega$. Since $(w_n,z_n) \to (w,z)$ in $L^1(A;\RR^2)$, by Fubini's Theorem there exists a subsequence (not relabeled) such that for $\L^1$-a.e.\ $x_2 \in \RR$ we have
    \begin{equation} \label{eq:convergence of slice}
        w_n^{x_2} \to w^{x_2} \quad \text{in } L^1(A^{x_2}) \, .
    \end{equation}

    For $n$ large enough we have the inequality
    \begin{equation*}
        \begin{split}
            \Hh_n(w_n,z_n;\Omega) & =\frac{1}{\sqrt{2} \ln \dn^{3/2} }  \frac{1}{2}  \ \ln^2 \hspace{-1em}  \sum_{(i,j) \in \I^n(\Omega)} \Big| u_n^{i+2,j} - \frac{\alpha_n}{2} u_n^{i+1,j}  + u_n^{i,j} \Big|^2 \\
            & \geq  \integral{\RR}{ H_n^{\mathrm{1d}}(u_n^{x_2};A^{x_2}) }{\d x_2}  \, .
        \end{split}
    \end{equation*}
    By Fatou's Lemma we obtain that 
    \begin{equation} \label{eq:Fatou}
             \liminf_{n \to \infty}\Hh_n(w_n,z_n;\Omega) \geq \integral{\RR}{\liminf_{n\to \infty} H_n^{\mathrm{1d}}(u_n^{x_2};A^{x_2}) }{\d x_2} \, .
    \end{equation}
    Let us fix an $x_2 \in \RR$ which satisfies~\eqref{eq:convergence of slice} and such that $w^{x_2} \in BV(A^{x_2};\{1,-1\})$ (this property of $w$ is true for $\L^1$-a.e.\ $x_2$).
    Let~$(I^{x_2}_\ell)_{\ell = 1}^\infty$ be a countable family of pairwise disjoint open intervals such that $ \bigcup_\ell I^{x_2}_\ell = A^{x_2}$. Let $L \in \NN$. Then $H_{n}^{\mathrm{1d}}(u_{n}^{x_2};A^{x_2}) \geq \sum_{\ell = 1}^L H_{n}^{\mathrm{1d}}(u_{n}^{x_2}; I^{x_2}_\ell)$. For each $\ell=1,\dots, L$ we apply the one-dimensional result for the  liminf inequality of the discrete functionals $H_{n}^{\mathrm{1d}}(\, \cdot \, ; I^{x_2}_\ell)$ proven in~\cite[Theorem~4.2]{CicSol} to deduce that 
    \begin{equation*}
        \liminf_{n \to \infty} H_{n}^{\mathrm{1d}}(u_{n}^{x_2};A^{x_2}) \geq \sum_{\ell = 1}^L  \frac{8}{3} \# ( J_{w^{x_2}} \cap  I^{x_2}_\ell ) = \sum_{\ell = 1}^L  \frac{4}{3} |\D (w^{x_2})|(I^{x_2}_\ell ) \, .
    \end{equation*} 
    For the inequality above we exploited the convergence $w^{x_2}_{n} \to w^{x_2}$ in $L^1(A^{x_2})$ due to~\eqref{eq:convergence of slice}. We recall that the topology used  for the $\Gamma$-convergence in~\cite{CicSol} is the strong $L^1$ topology with respect to this order parameter. (We remark that the result in~\cite{CicSol} is proven under additional periodic boundary conditions on the spin field variable; nonetheless the boundary conditions do not play any role in the  liminf  inequality in the regime we are interested in, i.e., \cite[Theorem 4.2-(i)]{CicSol}, cf.\ also~\cite[Remark~3.1]{CicSol}. Moreover, the result in~\cite{CicSol} is extended with minor modifications to the case where the domain is a generic interval of finite length.) By letting $L \to \infty$, by~\eqref{eq:Fatou}, and by~\cite[Theorem~3.103]{AmbFusPal} we infer that 
    \begin{equation*}
        \liminf_{n \to \infty}\Hh_n(w_n,z_n;\Omega) \geq \integral{\RR}{\frac{4}{3} |\D (w^{x_2})|(A^{x_2} )}{\d x_2} = \frac{4}{3} |\D_1 w|(A) \, .
    \end{equation*}
    Letting $A \nearrow \Omega$, we conclude that 
    \begin{equation*}
        \liminf_{n \to \infty}\Hh_n(w_n,z_n;\Omega) \geq \frac{4}{3} |\D_1 w|(\Omega) \, .
    \end{equation*}

Following the same steps we prove that 
\begin{equation*}
    \liminf_{n \to \infty}\Hv_n(w_n,z_n;\Omega) \geq \frac{4}{3} |\D_2 z|(\Omega) \, .
\end{equation*}
Using the decomposition $H_n = \Hh_n + \Hv_n$ we conclude the proof of~\eqref{eq:liminf}. 

Note that $(w,z) \in \Dom(H;\Omega)$. Indeed, the fact that $|\D_1 w|(\Omega) + |\D_2 z|(\Omega) < \infty$ and Proposition~\ref{prop:bootstrap} imply $(w,z) \in BV(\Omega;\RR^2)$.
\end{proof}

We are now in a position to conclude the proof of the compactness result.

\begin{proof}[Proof of Theorem~\ref{thm:main}-(i)]
    Let $(w_n,z_n) \in \Lloc(\RR^2;\RR^2)$ be such that $H_n(w_n,z_n;\Omega) \leq C$. By Proposition~\ref{prop:partial compactness} there exists a subsequence (not relabeled) such that $(w_n,z_n) \to (w,z)$ in $\Lloc(\Omega;\RR^2)$, where $(w,z) \in BV_{\mathrm{loc}}(\Omega;\RR^2)$, $w \in \{1,-1\}$ and $z \in \{1,-1\}$ a.e.\ in $\Omega$, and $\curl(w,z) = 0$ in $\mathcal{D}'(\Omega)$. To conclude the proof of the compactness result, we need to prove that $(w,z) \in BV(\Omega;\RR^2)$. This follows from the same argument in the proof of Theorem~\ref{thm:main}-(ii): by the  liminf  inequality we have that 
\begin{equation*}
     \frac{4}{3}\Big( |\D_1 w|(\Omega) + |\D_2 z|(\Omega) \Big) \leq \liminf_{n \to \infty}H_n(w_n,z_n;\Omega)\leq C < \infty \, .
\end{equation*}
Then Proposition~\ref{prop:bootstrap} implies that $(w,z) \in BV(\Omega;\RR^2)$.
\end{proof}

\section{Proof of the  limsup  inequality}

In this section we prove the limsup  inequality, i.e., Theorem~\ref{thm:main}-(iii). Let us fix $\Omega \in \A_0$ and $(w,z) \in \Lloc(\RR^2;\RR^2)$ and let us prove that there exists a sequence $(w_n,z_n) \to (w,z)$ in $L^1(\Omega;\RR^2)$ such that 
\begin{equation} \label{eq:gamma limsup}
    \limsup_{n \to \infty} H_n(w_n , z_n ;\Omega)   \leq H(w,z;\Omega) \, .
\end{equation}
If $(w,z) \notin \Dom(H;\Omega)$ the statement is trivial, thus we assume $(w,z) \in \Dom(H;\Omega)$ in what follows. According to Lemma~\ref{lemma:existence of potential}, $(w,z)$ admits a potential $\varphi  \in BVG(\Omega)$. As it turns out from our construction below, it will be easier to work on the function $\varphi$ for the definition of the recovery sequence.

Relying on the idea that the functionals $H_n$ resemble a discrete version of second-order Modica-Mortola functionals, we resort to a technique proposed in~\cite{Pol} to prove upper bounds for generic singular perturbation problems of the form 
\begin{equation*}
    \frac{1}{\en} \integral{\Omega}{F(\en \nabla^2 \varphi(x), \nabla \varphi(x))}{\d x} \, .
\end{equation*}
Specifically, we shall apply~\cite[Theorems~6.1, 6.2]{Pol} to the sequence of functionals 
\begin{equation} \label{eq:second order MM}
     \integral{\Omega}{\frac{1}{\en} W(\de_1 \varphi) + \frac{1}{\en} W(\de_2 \varphi) + \en |\de_{11} \varphi|^2 + \en |\de_{22}\varphi|^2}{\d x} \, ,
\end{equation}
i.e., to the case 
\begin{equation} \label{eq:def of F}
    F(A,b) = W(b_1) + W(b_2) + a_{11}^2 + a_{22}^2 \quad \text{for } A = (a_{ij})_{i,j=1,2} \in \RR^{2\x 2}, \ b = (b_1,b_2 )\in \RR^2 , 
\end{equation} 
where $W$ is the classical double-well potential given by $W(s) := (1-s^2)^2$. Before proving~\eqref{eq:gamma limsup}, we recall that the technique proposed in~\cite{Pol} uses a sequence of mollifications of $\varphi$ to obtain a candidate for the recovery sequence. This leads to an asymptotic upper bound for the functionals in~\eqref{eq:second order MM} which depends on the choice of the mollifier. Subsequently, the limsup inequality is obtained by optimizing the upper bound over all admissible mollifiers.

To define a mollification of $\varphi$ on $\Omega$ we first extend it to the whole $\RR^2$.  Since $\Omega$ is a $BVG$ domain, by Proposition~\ref{prop:extension of BVG} we can find a compactly supported function $\ol \varphi \in BVG(\RR^2)$ such that $\ol \varphi = \varphi$ a.e.\ in $\Omega$ and $|\D \nabla \varphi|(\de \Omega) = 0$. 

\begin{remark}
    In order to apply~\cite{Pol}, it is not required that the extension $\ol \varphi$ satisfies the condition $\nabla \ol \varphi \in \{1, -1\}^2$. We stress that, in general, it is not possible to extend $\varphi \in BVG(\Omega)$ with $\nabla \varphi \in \Dom(H;\Omega)$ to a function $\ol \varphi \in BVG(\tilde \Omega)$ with $\nabla \ol \varphi \in \Dom(H;\tilde \Omega)$, for some $\tilde \Omega \supcc \Omega$, cf.\ Figure~\ref{fig:nonextendable}.
\end{remark}


\begin{figure}[H]
    \includegraphics{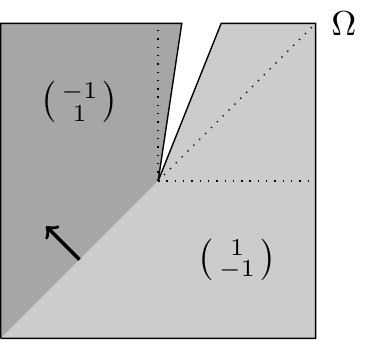}

    \caption{Example of a pair $(w,z)$ attaining the values $(1,-1)$ and $(-1,1)$ defined in the open set $\Omega$ given by the wedged square. It is not possible to extend $(w,z)$ to the full square to an admissible pair $(\tilde w, \tilde z)$.}
    \label{fig:nonextendable}
\end{figure}

We define a sequence $\varphi^\e$ by convolving $\varphi$ with suitable kernels. Following~\cite{Pol}, we introduce the class $\V(\Omega)$ consisting of mollifiers $\eta \in C^3_c(\RR^2 \x \RR^2; \RR)$ satisfying 
\begin{equation} \label{eq:eta has integral 1}
    \integral{\RR^2}{\eta(z,x)}{\d z} = 1 \quad \text{for all } x \in \Omega \, .
\end{equation}
\begin{remark}
    In~\cite{Pol} the author uses a slightly different class of mollifiers, only requiring a~$C^2$ regularity. We remark that the proofs of~\cite[Theorem~6.1, Theorem~6.2]{Pol} also work under this stronger regularity assumption on the convolution kernels. 
\end{remark}
Let us fix a mollifier $\eta \in \V(\Omega)$ and let us define  
    \begin{equation} \label{eq:def of varphieps}
        \varphi^{\e}(x) := \frac{1}{\e^2} \integral{\RR^2}{\eta \big( \tfrac{y-x}{\e} , x\big) \ol \varphi(y)}{\d y}  =   \integral{\RR^2}{\eta ( z , x ) \ol \varphi( x + \e z)}{\d z} \quad \text{for } x \in \RR^2 .
    \end{equation}
Evaluating the sequence of functionals in~\eqref{eq:second order MM} at the functions $\varphi^\en$, we obtain a first asymptotic upper bound. More precisely, by~\cite[Theorem~6.1]{Pol} one has that 
    \begin{equation} \label{eq:Poliakovsky applied}
        \lim_{n \to \infty } \integral{\Omega}{ \sum_{k=1}^2 \frac{1}{\en} W(\de_k \varphi^{\en}) + \en |\de_{kk} \varphi^{\en}|^2 }{\d x}  = Y[\eta](\varphi) \, ,
    \end{equation}
where an explicit formula for $Y[\eta](\varphi)$ is given in~\cite[Formula~(6.4)]{Pol}. The precise expression of $Y[\eta](\varphi)$ is not relevant for our purposes. It is however important to derive the expression obtained when we optimize $Y[\eta](\varphi)$ with respect to $\eta \in \V(\Omega)$.

\begin{proposition} \label{prop:optimization in eta}
    The following equality holds true:
    \begin{equation*} 
        \inf_{\eta \in \V(\Omega)} Y[\eta](\varphi) = H(\de_1 \varphi, \de_2 \varphi;\Omega) = H(w, z;\Omega) \, .
    \end{equation*}
\end{proposition}
\begin{proof}
    We recall that \cite[Theorem~6.2]{Pol} gives 
    \begin{equation*}
        \inf_{\eta \in \V(\Omega)} Y[\eta](\varphi) = \integral{J_{\nabla \varphi}}{ \sigma\big(\nabla \varphi^+(x), \nabla \varphi^-(x), \nu_{\nabla \varphi}(x) \big) }{\d \H^1(x)} \, ,
    \end{equation*}
    where the surface density $\sigma$ is obtained by optimizing the energy for a transition from $\nabla \varphi^-(x)$ to $\nabla \varphi^+(x)$  over one-dimensional profiles and is given by
    \begin{equation*}
        \begin{split}
            \sigma(a,b,\nu) := \inf_{\gamma} \Big\{  \int \limits_{-\infty}^{+\infty} \!  F\big( - \gamma'(t) \, \nu  \otimes \nu , \gamma(t) \,  \nu + b\big) \d t \ : \ \gamma \in C^1(\RR) \, , \text{ there exists } L > 0 & \\[-1em]
            \text{ s.t.\ for } t \geq L  \text{ we have } \gamma(-t) = \chi \text{ and } \gamma(t) = 0 & \Big\}
        \end{split}
    \end{equation*}
    for every $a, b \in \RR^2$ and $\nu \in \SS^1$ such that $(a-b) = \chi  \, \nu$ for some $\chi \in \RR$. This exhaustively defines the energy for the triple $\big(\nabla \varphi^+(x), \nabla \varphi^-(x), \nu_{\nabla \varphi} (x) \big)$ at $\H^1$-a.e.\ $x \in J_{\nabla \varphi}$, cf.\ Subsection~\ref{subsec:BV gradients}. 

    We claim that for $\H^1$-a.e.\ $x \in J_{\nabla \varphi} = J_{(w,z)}$ it holds
    \begin{equation} \label{eq:claim on sigma}
        \sigma\big( \nabla \varphi^+(x), \nabla \varphi^-(x), \nu_{\nabla \varphi}(x) \big) = \frac{4}{3} \Big( | w^+(x)-w^-(x) | |\nu^1_{w}(x)| + |z^+(x) - z^-(x)| |\nu^2_{z}(x)| \Big) \, . 
    \end{equation}
    By the definition of $H$ in~\eqref{eq:def of H}, this will conclude the proof. Notice that $\nu_{w}(x) = \nu_{\nabla \varphi}(x)$ for $\H^1$-a.e.\ $x \in J_{w}$ and $\nu_{z}(x) = \nu_{\nabla \varphi}(x)$ for $\H^1$-a.e.\ $x \in J_{z}$.

    By Subsection~\ref{subsec:rigid jump part}, for $\H^1$-a.e.\ $x \in J_{\nabla \varphi}$ the triple $(\nabla \varphi^+(x), \nabla \varphi^-(x), \nu_{\nabla \varphi}(x))$ can attain only the following values (up to a permutation of $(\nabla \varphi^+(x), \nabla \varphi^-(x))$ and a change of sign of $\nu_{\nabla \varphi}(x)$) 
    \begin{equation*}
        \begin{split}
            & \Big(  \vc{1}{1}, \vc{-1}{1} , \vc{\pm 1}{0}  \Big) \, , \Big(\vc{1}{-1},  \vc{-1}{-1} , \vc{\pm 1}{0}  \Big)  \, , \Big( \vc{1}{1}, \vc{1}{-1} , \vc{0}{\pm 1}  \Big)  \, , \Big( \vc{-1}{1}, \vc{-1}{-1} ,  \vc{0}{\pm 1}  \Big)  \, , \\
            & \Big( \vc{1}{1}, \vc{-1}{-1} ,  \vc{\pm 1/\sqrt{2}}{\pm 1/\sqrt{2}}  \Big) \, , \Big(\vc{-1}{1} , \vc{1}{-1},   \vc{\mp 1/\sqrt{2}}{\pm 1/\sqrt{2}}  \Big) \, .
        \end{split}
    \end{equation*}
    Let us fix $x \in J_{\nabla \varphi}$ such that $\big( \nabla \varphi^+(x), \nabla \varphi^-(x), \nu_{\nabla \varphi}(x) \big) = (a,b,\nu)$, with $(a,b,\nu)$ being one of the triples above. For simplicity, we will exhibit the value of $\sigma(a,b,\nu)$ for  $(a,b,\nu) = \Big( \vc{1}{1} , \vc{-1}{1}, \vc{1}{0}  \Big)$ and for   $ (a,b,\nu) = \Big( \vc{1}{1} , \vc{-1}{-1}, \vc{1/\sqrt{2}}{1/\sqrt{2}}  \Big)$. All the other cases are treated similarly.
    
    {\em Case 1:} $(a,b,\nu) = \Big( \vc{1}{1} , \vc{-1}{1}, \vc{1}{0}  \Big)$, and thus $a-b = \chi \, \nu$ with $\chi = 2$. Let us fix an admissible profile $\gamma \in C^1(\RR)$ such that $\gamma(-t) = 2$  and $\gamma(t) = 0$ for $t \geq L$. Then, by~\eqref{eq:def of F} we have 
    \begin{equation*}
        \int \limits_{-\infty}^{+\infty} \! F\big( - \gamma'(t) \,  \nu  \otimes \nu , \gamma(t) \, \nu + b\big) \d t = \int \limits_{-\infty}^{+\infty} \! W(\gamma(t) - 1) + |\gamma'(t)|^2 \d t \, .
    \end{equation*}
    This implies  that the infimum problem that defines $\sigma(a,b,\nu)$ coincides with the infimum problem for the optimal profile in the one-dimensional Modica-Mortola functional. It is well known that $\sigma(a,b,\nu)= \frac{8}{3}$, see~\cite[Remark~6.1]{Bra}, and therefore~\eqref{eq:claim on sigma} holds true.
    
    {\em Case 2:} $(a,b,\nu) = \Big( \vc{1}{1} , \vc{-1}{-1}, \vc{1/\sqrt{2}}{1/\sqrt{2}}  \Big)$, and thus $a-b = \chi \, \nu$ with $\chi = 2 \sqrt{2}$.
    Let $\gamma \in C^1(\RR)$ be an admissible profile such that $\gamma(-t) = 2\sqrt{2}$  and $\gamma(t) = 0$ for $t \geq L$. Putting $\tilde \gamma(s) := \gamma(\tfrac{s}{\sqrt{2}})/\sqrt{2}$, $s \in \RR$, by~\eqref{eq:def of F} we infer that  
    \begin{equation*}
        \begin{split}
            \int \limits_{-\infty}^{+\infty} \! F\big( - \gamma'(t) \,  \nu  \otimes \nu , \gamma(t) \, \nu + b\big) \d t & = \int \limits_{-\infty}^{+\infty} \! 2 \, W\big(\gamma(t)\tfrac{1}{\sqrt{2}} - 1\big) + \tfrac{1}{2}|\gamma'(t)|^2 \d t \\
            & =  \sqrt{2}  \int \limits_{-\infty}^{+\infty} \! W\big(\tilde \gamma(s) - 1\big) + |\tilde \gamma'(s)|^2 \d s \, .
        \end{split}
    \end{equation*}
    Note that $\tilde \gamma(-s) = 2$  and $\tilde \gamma(s) = 0$ for $s \geq L \sqrt{2}$. Thus, up to the multiplicative factor~$\sqrt{2}$, the infimum problem that defines $\sigma(a,b,\nu)$ coincides with the infimum problem for the optimal profile of the Modica-Mortola functional. In conclusion $\sigma(a,b,\nu) = \sqrt{2} \tfrac{8}{3}$ and again~\eqref{eq:claim on sigma} holds true. 
\end{proof} 

    Thanks to Proposition~\ref{prop:optimization in eta}, to prove~\eqref{eq:gamma limsup} it will be enough to  construct a sequence $(w_n,z_n) = \T_n(u_n) \in \T_n(\PC_{\ln}(\SS^1))$ such that $(w_n,z_n) \to (w,z)$ strongly in $L^1(\Omega;\RR^2)$ and 
\begin{equation} \label{eq:almost gamma limsup}
    \limsup_{n \to \infty} H_n(w_n,z_n;\Omega) \leq Y[\eta](\varphi) \, .
\end{equation}
The final statement~\eqref{eq:gamma limsup} is then obtained by a diagonal argument. 

In order to define $(w_n,z_n)$, we discretize on the lattice $\ln \ZZ^2$ the sequence~$\varphi^{\e_n}$. Specifically, we define $\varphi_n \in \PC_{\ln}(\RR)$ by 
    \begin{equation} \label{eq:recovery phi}
        \varphi_n^{i,j} := \varphi^{\en}(\ln i, \ln j) \, .
    \end{equation}
We start by comparing the second-order Modica-Mortola energy of $\varphi^{\en}$ with its discrete counterpart computed on $\varphi_n$. 

    \begin{proposition} \label{prop:discrete MM}
    For $k=1,2$ we have that 
        \begin{equation} \label{eq:discrete Poliakovsky}
             \integral{\Omega}{ \frac{1}{\en} W(\de_k \varphi^{\en}(x)) + \en |\de_{kk} \varphi^{\en}(x)|^2 }{\d x} = \integral{\Omega}{\frac{1}{\en} W(\de_k^{\mathrm{d}} \varphi_n(x)) + \en |\de_{kk}^{\mathrm{d}} \varphi_n(x)|^2 }{\d x} + o(1)  \, ,
        \end{equation}
        where $o(1) \to 0$ as $n \to \infty$.
    \end{proposition}
    \begin{proof}
        \ul{Step 1}: We claim that there exists a constant $C_0 > 0$ such that 
        \begin{align}
            \big| \de_k \varphi^{\en}(x) \big|& \leq C_0 \, , \label{eq:d of phi bounded} \\ 
            \big |\de^{\mathrm{d}}_k \varphi_n(x) \big | & \leq C_0 \, ,\label{eq:discrete d of phi bounded}
        \end{align}
        for every $x \in \RR^2$, $k=1,2$, and every $n$.
    Indeed, by definition of $\varphi^{\en}$ in~\eqref{eq:def of varphieps} we infer that 
    \begin{equation*}
        \begin{split}
            \big| \de_k \varphi^{\en}(x) \big| 
            & \leq \integral{\RR^2}{ | \de_{x_k} \eta ( z , x ) \ol \varphi( x + \en z)|}{\d z} +  \integral{\RR^2}{| \eta ( z , x ) \de_k \ol \varphi( x + \en z) | }{\d z} \\
            & \leq C \|\nabla \eta\|_{L^\infty} \|\ol \varphi\|_{L^\infty} + C \| \eta \|_{L^\infty} \| \nabla \ol \varphi \|_{L^\infty} =: C_0 \, .
        \end{split}
    \end{equation*}
    Moreover, for $k=1$ we get
    \begin{equation} \label{eq:1912181456}
        |\de^{\mathrm{d}}_1 \varphi_n^{i,j}|  = \Big| \frac{\varphi^{\en}(\ln(i+1),\ln j) - \varphi^{\en}(\ln i,\ln j)}{\ln} \Big| \leq  \int \limits_0^1 \! \big| \de_1 \varphi^{\en}(\ln(i+t),\ln j) \big| \d t \leq C_0 \, .
\end{equation} 
With analogous computations we prove the estimate on $|\de^{\mathrm{d}}_2 \varphi_n^{i,j}|$.

        \ul{Step 2}: We prove the following control on the second and third derivatives of $\varphi^{\en}$: setting $Q^2_{\ln}(i,j) := (\ln i , \ln j) + [0,2 \ln)^2$, we show that 
        \begin{gather}
        \sum_{Q_{\ln}(i,j) \cap \Omega \neq \emptyset}  \ln^2 \sup_{Q^2_{\ln}(i,j)} \! \! |\nabla^2 \varphi^{\en}| \leq C \, , \label{eq:control on second derivatives}\\
        \sup_{\RR^2} |\nabla^3 \varphi^{\en}| \leq \frac{C}{\en^2} \, , \label{eq:control on third derivatives}
        \end{gather} 
        for some constant $C$.

        From the very definition of $\varphi^{\en}$ in~\eqref{eq:def of varphieps} and recalling that $\nabla \ol \varphi \in BV(\RR^2;\RR^2)$ we get 
        \begin{equation} \label{eq:second derivative of phi}
            \begin{split}
                \de_{hk} \varphi^{\en}(x) & = \frac{1}{\en^2} \integral{\RR^2}{\eta(\tfrac{y-x}{\en},x)}{\d \D_{h}\de_{k} \ol \varphi(y) } + \frac{1}{\en^2} \integral{\RR^2}{\de_{x_h x_k} \eta(\tfrac{y-x}{\en},x)\ol \varphi(y)}{\d y} \\
                &\quad + \frac{1}{\en^2} \integral{\RR^2}{\de_{x_h} \eta(\tfrac{y-x}{\en},x) \de_{k} \ol \varphi(y)}{\d y} +  \frac{1}{\en^2} \integral{\RR^2}{\de_{x_k} \eta(\tfrac{y-x}{\en},x) \de_{h} \ol \varphi(y)}{\d y}  
            \end{split}
        \end{equation}
        and 
        \begin{equation} \label{eq:third derivative of phi}
            \begin{split}
                \de_{hk \ell} \varphi^{\en}(x) & = \frac{1}{\en^4} \integral{\RR^2}{\de_{z_h z_\ell }\eta(\tfrac{y-x}{\en},x) \de_{k} \ol \varphi(y)}{\d y }  - \frac{1}{\en^3} \integral{\RR^2}{\de_{z_h x_\ell }\eta(\tfrac{y-x}{\en},x) \de_{k} \ol \varphi(y)}{\d y } \\
                & \quad - \frac{1}{\en^3} \integral{\RR^2}{\de_{z_\ell x_h x_k} \eta(\tfrac{y-x}{\en},x)\ol \varphi(y)}{\d y}  + \frac{1}{\en^2} \integral{\RR^2}{\de_{x_h x_k x_\ell} \eta(\tfrac{y-x}{\en},x)\ol \varphi(y)}{\d y} \\
                &\quad - \frac{1}{\en^3} \integral{\RR^2}{\de_{z_\ell x_h} \eta(\tfrac{y-x}{\en},x) \de_{k} \ol \varphi(y)}{\d y} + \frac{1}{\en^2} \integral{\RR^2}{\de_{x_h x_\ell} \eta(\tfrac{y-x}{\en},x) \de_{k} \ol \varphi(y)}{\d y}   \\
                & \quad - \frac{1}{\en^3} \integral{\RR^2}{ \de_{z_\ell x_k} \eta(\tfrac{y-x}{\en},x) \de_{h} \ol \varphi(y)}{\d y} + \frac{1}{\en^2} \integral{\RR^2}{ \de_{x_k x_\ell } \eta(\tfrac{y-x}{\en},x) \de_{h} \ol \varphi(y)}{\d y} \, ,
            \end{split}
        \end{equation}
        where $\de_{z_h} \eta(z,x)$ and $\de_{x_h} \eta(z,x)$ denote the derivative with respect to the $h$-th variable in the first and second group of variables of $\eta(z,x)$ respectively and $\D_h \de_{k} \ol \varphi$ denotes the $h$-th component of the distributional derivative of $\de_{k} \ol \varphi$. 

        By the assumptions on $\eta$, the function $y \mapsto \eta\big(\tfrac{y - x}{\en},x\big)$ is supported on a ball $B_{R \en}(x)$ for a suitable $R > 0$ (independent of $n$ and $x$). Together with the condition $\ol \varphi \in W^{1,\infty}(\RR^2)$, \eqref{eq:third derivative of phi} yields~\eqref{eq:control on third derivatives}.

        We remark that integrating by parts~\eqref{eq:second derivative of phi} yields the bound $\sup_{\RR^2} |\nabla^2 \varphi^{\en}| \leq \frac{C}{\en}$. However, this estimate is too weak to guarantee~\eqref{eq:control on second derivatives}. The proof of~\eqref{eq:control on second derivatives} requires a finer argument. Using~\eqref{eq:second derivative of phi} and the fact that the function $y \mapsto \eta\big(\tfrac{y - x}{\en},x \big)$ is supported on a ball~$B_{R \en}(x)$, we observe that for every $x \in  Q^2_{\ln}(i,j)$ 
        \begin{equation*}
             |\nabla^2 \varphi^{\en}(x)| \leq  \|\eta\|_{\infty} \frac{1}{\en^2} |\D \nabla \ol \varphi|\big(B_{R \en}(x)\big) + C
        \end{equation*}
        and therefore
        \begin{equation*}
            \sup_{ Q^2_{\ln}(i,j)} \! \! |\nabla^2 \varphi^{\en}| \leq C  \Big[ 1 + \frac{1}{\en^2} |\D \nabla \ol \varphi|\big( Q^2_{\ln}(i,j) + B_{R \en}\big) \Big] \, .
        \end{equation*}
        Summing over all $(i,j) \in \ZZ^2$ such that $ Q_{\ln}(i,j) \cap \Omega \neq \emptyset$, we get the bound
        \begin{equation} \label{eq:1502191748}
            \begin{split}
                \sum_{ Q_{\ln}(i,j) \cap \Omega \neq \emptyset}  \ln^2 \sup_{ Q^2_{\ln}(i,j)} \! \! |\nabla^2 \varphi^{\en}| & \leq  \sum_{ Q_{\ln}(i,j) \cap \Omega \neq \emptyset}  \ln^2  C  \Big[ 1 + \frac{1}{\en^2} |\D \nabla \ol \varphi|\big( Q^2_{\ln}(i,j) + B_{R \en}\big) \Big] \\
                & \leq C + \frac{\ln^2}{\en^2} \sum_{(i,j) \in \ZZ^2} |\D \nabla \ol \varphi|\big( Q^2_{\ln}(i,j) + B_{R \en}\big) \, .
            \end{split}
        \end{equation}
  To estimate the right-hand side, we observe that for every $(i,j) \in \ZZ^2$ 
        \begin{equation*}
            |\D \nabla \ol \varphi|\big( Q^2_{\ln}(i,j) + B_{R \en}\big) \leq \hspace{-3em}\sum_{\substack{(i',j')\in \ZZ^2 \\ Q_{\ln}(i',j') \cap ( Q^2_{\ln}(i,j) + B_{R \en}) \neq \emptyset}} \hspace{-3em} |\D \nabla \ol \varphi|\big(Q_{\ln}(i',j')\big)
        \end{equation*}
        and thus 
        \begin{equation*}
            \frac{\ln^2}{\en^2} \sum_{(i,j) \in \ZZ^2} |\D \nabla \ol \varphi|\big( Q^2_{\ln}(i,j) + B_{R \en}\big) \leq \frac{\ln^2}{\en^2}  \sum_{(i',j') \in \ZZ^2} \sum_{(i,j) \in \mathcal{N}_{n}(i',j')} |\D \nabla \ol \varphi|\big(Q_{\ln}(i',j')\big) \, ,
        \end{equation*} 
        where $\mathcal{N}_{n}(i',j') = \{(i,j) \in \ZZ^2 \ : \  Q_{\ln}(i',j') \cap ( Q^2_{\ln}(i,j) + B_{R \en}) \neq \emptyset \}$. A simple counting argument shows that $\# \mathcal{N}_{n}(i',j') \leq C (\en \ln^{-1})^2$ (where we used that $\en / \ln \to \infty$), thus 
        \begin{equation*}
            \frac{\ln^2}{\en^2} \sum_{(i,j) \in \ZZ^2} |\D \nabla \ol \varphi|\big( Q^2_{\ln}(i,j) + B_{R \en}\big)  \leq C |\D \nabla \ol \varphi|(\RR^2) \, . 
        \end{equation*}
        Together with~\eqref{eq:1502191748}, this concludes the proof of~\eqref{eq:control on second derivatives}.

        \ul{Step 3}: We show that 
        \begin{equation*}
             \frac{1}{\en} \integral{\Omega}{ \big| W(\de_k \varphi^{\en}(x)) - W(\de_k^{\mathrm{d}} \varphi_n(x)) \big| }{\d x}   \to 0 \, ,
        \end{equation*} 
        for $k = 1,2$. For the sake of simplicity, we prove the claim for $k=1$, the case $k=2$ being analogous. 

        We start by observing that for every $x \in Q_{\ln}(i,j)$
        \begin{equation} \label{eq:discretization error in derivative}
            \begin{split}
               \big|\de_1 \varphi^{\en}(x) - \de_1^{\mathrm{d}} \varphi_n(x)  \big| & = \big|\de_1 \varphi^{\en}(x) - \de_1^{\mathrm{d}} \varphi_n^{i,j}  \big| \leq  \int \limits_0^1 \! \big| \de_1 \varphi^{\en}(x) - \de_1 \varphi^{\en}(\ln(i+t), \ln j)  \big|   \d t  \\
                 & \leq \sup_{ Q_{\ln}(i,j)}|\nabla^2 \varphi^{\en}| \sqrt{2} \ln  \, .
            \end{split}
        \end{equation}
        Setting $L := \sup_{|s|\leq C_0} |W'(s)|$, by~\eqref{eq:d of phi bounded}--\eqref{eq:discrete d of phi bounded} we obtain that 
        \begin{equation*}
                \big| W(\de_1 \varphi^{\en}(x)) - W(\de_1^{\mathrm{d}} \varphi_n(x)) \big|  \leq L \big|\de_1 \varphi^{\en}(x) - \de_1^{\mathrm{d}} \varphi_n(x)  \big|  \leq L \sup_{ Q_{\ln}(i,j)}|\nabla^2 \varphi^{\en}| \sqrt{2} \ln  \, .
        \end{equation*}
        Summing over all indices $(i,j)$ such that the square $ Q_{\ln}(i,j)$ intersects $\Omega$, by~\eqref{eq:control on second derivatives} we get 
        \begin{equation*}
            \begin{split}
                \frac{1}{\en} \integral{\Omega}{ \big| W(\de_1 \varphi^{\en}(x)) - W(\de_1^{\mathrm{d}} \varphi_n(x)) \big| }{\d x} & \leq  \sqrt{2}L  \frac{\ln }{\en} \sum_{ Q_{\ln}(i,j) \cap \Omega \neq \emptyset} \ln^2 \sup_{ Q_{\ln}(i,j)}|\nabla^2 \varphi^{\en}|  \\
                & \leq C \frac{\ln }{\en} = C \sqrt{2 \dn} \to 0 \, .
            \end{split}
        \end{equation*}

         \ul{Step 4}: We show that 
         \begin{equation*}
              \en  \integral{\Omega}{ \big| |\de_{kk} \varphi^{\en}(x)|^2 - |\de^{\mathrm{d}}_{kk} \varphi_n(x)|^2 \big| }{\d x}   \to 0 \, ,
         \end{equation*} 
         for $k = 1,2$. Also in this case we prove the claim only for $k=1$. 

         For every $x \in Q_{\ln}(i,j)$ we have that 
         \begin{equation*}
             \de_{11}^{\mathrm{d}}\varphi_n(x) = \int \limits_0^1 \! \int \limits_0^1 \! \de_{11}\varphi^{\en}(\ln(i+s+t),\ln j) \d s \d t
         \end{equation*}
         and thus, noting that $|x - (\ln (i+s+t), \ln j)| \leq \sqrt{5} \ln$,
         \begin{equation*}
            \begin{split}
                \big| |\de_{11} \varphi^{\en}(x)|^2 - |\de^{\mathrm{d}}_{11} \varphi_n(x)|^2 \big| & =    \big| \de_{11} \varphi^{\en}(x) + \de^{\mathrm{d}}_{11} \varphi_n(x) \big| \big| \de_{11} \varphi^{\en}(x) - \de^{\mathrm{d}}_{11} \varphi_n(x) \big| \\
                & \leq 2\!\! \sup_{ Q^2_{\ln}(i,j)} \! \! |\nabla^2 \varphi^{\en}| \ \ \sup_{ Q^2_{\ln}(i,j)} |\nabla^3 \varphi^{\en}|  \sqrt{5} \ln \, .
            \end{split}
         \end{equation*}
         Summing over all indices $(i,j)$ such that the square $ Q_{\ln}(i,j)$ intersects $\Omega$, by~\eqref{eq:control on second derivatives}--\eqref{eq:control on third derivatives} we get 
         \begin{equation*} 
            \begin{split}
                & \en  \integral{\Omega}{ \big| |\de_{kk} \varphi^{\en}(x)|^2 - |\de^{\mathrm{d}}_{kk} \varphi_n(x)|^2 \big| }{\d x} \\
                & \quad \leq  2 \sqrt{5} \en \ln \!\!\! \sum_{ Q_{\ln}(i,j) \cap \Omega \neq \emptyset}  \ln^2 \sup_{ Q^2_{\ln}(i,j)} \! \! |\nabla^2 \varphi^{\en}| \ \ \sup_{ Q^2_{\ln}(i,j)} |\nabla^3 \varphi^{\en}| \\
                & \quad \leq C\frac{\ln}{\en}  = C \sqrt{2 \dn} \to 0 \, .
            \end{split}
        \end{equation*}
        This concludes the proof.
    \end{proof}

    We are now in a position to define the sequence $(w_n,z_n)$. To this end, it is convenient to introduce the spin field $u_n \in \PC_{\ln}(\SS^1)$ defined by
    \begin{equation}\label{recovery-spin}
        u_n^{i,j} := \Big(\cos\Big(\frac{1}{\ln} \arccos(1-\dn) \varphi_n^{i,j} \Big), \sin\Big(\frac{1}{\ln} \arccos(1-\dn) \varphi_n^{i,j} \Big) \Big) \, .
    \end{equation}
    Then we define
    \begin{equation} \label{eq:def of wn zn}
        (w_n,z_n) := \T_n(u_n)  \in \T_n(\PC_{\ln}(\SS^1)) \, ,
    \end{equation}
    i.e., through formula~\eqref{eq:order parameter}. The spin field in \eqref{recovery-spin} is constructed in such a way that the following expression for the angles between neighboring spins holds true.  
    \begin{lemma}
        Let $u_n \in \PC_{\ln}(\SS^1)$ be defined as in \eqref{recovery-spin}. Then, it holds that
        \begin{equation} \label{eq:theta and discrete derivative}
            (\theth_n)^{i,j} = \arccos(1-\dn) \de^{\mathrm{d}}_1 \varphi_n^{i,j} \, , \quad         (\thetv_n)^{i,j} = \arccos(1-\dn) \de^{\mathrm{d}}_2 \varphi_n^{i,j} \, ,
        \end{equation}
        for $n$ large enough. Moreover, 
        \begin{equation} \label{eq:theta goes to 0}
            (\theth_n)^{i,j} \to 0  \quad \text{and} \quad (\thetv_n)^{i,j} \to 0
        \end{equation}
        as $n \to \infty$, uniformly in $(i,j)$.  
    \end{lemma}
    \begin{proof}
        We only prove the statement for $(\theth_n)^{i,j}$, the conclusion for $(\thetv_n)^{i,j}$ being analogous. 
     From~\eqref{eq:discrete d of phi bounded} we deduce that 
    \begin{equation} \label{eq:discrete derivative goes to 0}
        |\arccos(1-\dn) \de^{\mathrm{d}}_1 \varphi_n^{i,j}|  \leq \arccos(1-\dn) C_0 \leq \pi
    \end{equation}
    for $n$ large enough and for all $(i,j)$. In particular, from the very definition of~$\theth_n$ in~\eqref{eq:def of theth} and using standard trigonometric identities it follows that
\begin{equation*}
    \begin{split}
        (\theth_n)^{i,j} & = \mathrm{sign}\Big( \sin\big( \arccos(1-\dn) \de^{\mathrm{d}}_1 \varphi_n^{i,j} \big) \Big) \arccos\Big( \cos\big( \arccos(1-\dn) \de^{\mathrm{d}}_1 \varphi_n^{i,j} \big) \Big) \\
        & = \arccos(1-\dn) \de^{\mathrm{d}}_1 \varphi_n^{i,j} \, .
    \end{split}
\end{equation*}

The convergence in~\eqref{eq:theta goes to 0} follows now from the first inequality in \eqref{eq:discrete derivative goes to 0}.
    \end{proof}

    In the next proposition we prove the convergence of the recovery sequence. 
    \begin{proposition} \label{prop:convergence of w,z}
        Let $(w,z) \in \Dom(H;\Omega)$ and let $(w_n,z_n)$ be defined by \eqref{recovery-spin}--\eqref{eq:def of wn zn}. Then $(w_n,z_n) \to (w,z)$ strongly in $L^1(\Omega;\RR^2)$. 
    \end{proposition}
    \begin{proof}
         We only prove that $w_n \to w$ in $L^1(\Omega)$, the proof of $z_n \to z$ in $L^1(\Omega)$ being analogous.

        We recall that, by~\eqref{eq:def of wn zn}, \eqref{eq:theta and discrete derivative}, and since $(w,z) = \nabla \varphi$, we have that
         \begin{align*}
             w_n^{i,j} & = \sqrt{\frac{2}{\dn}} \sin \Big( \frac{1}{2} \arccos(1-\dn) \de_1^\mathrm{d} \varphi_n^{i,j} \Big) \, , \\
            w(x) & = \de_1 \varphi(x) \, .
         \end{align*}
        
        We start by observing that $\de_1 \varphi^{\en} \to w$ in $L^1(\Omega)$. Indeed, from~\eqref{eq:eta has integral 1} we deduce that 
        \begin{equation*}
            \de_1 \varphi(x) = \integral{\RR^2}{\de_{x_1} \eta(z,x) \ol \varphi(x) + \eta(z,x) \de_1 \ol \varphi(x) }{\d z} \, , \quad \text{for } x \in \Omega \, .
        \end{equation*}
        Together with~\eqref{eq:def of varphieps}, this yields  
        \begin{equation*}
            \begin{split}
               &  \integral{\Omega}{|\de_1 \varphi^{\en}(x) - \de_1 \varphi(x)|}{\d x}\\
                & \leq \integral{\Omega}{ \integral{\RR^2}{|\de_{x_1} \eta(z,x)|\ |\ol \varphi(x+\en z) - \ol \varphi(x)| }{\d z} }{\d x} + \integral{\Omega}{ \integral{\RR^2}{|\eta(z,x)|\ |\de_1 \ol \varphi(x+\en z) - \de_1 \ol \varphi(x)| }{\d z} }{\d x} \\
                & \leq \|\nabla \eta \|_{L^\infty} \integral{B_R}{ \|\ol \varphi(\, \cdot  +\en z) - \ol \varphi\|_{L^1(\Omega)} }{\d z}  + \|\eta \|_{L^\infty} \integral{B_R}{ \|\de_1 \ol \varphi(\, \cdot +\en z) - \de_1 \ol \varphi\|_{L^1(\Omega)} }{\d z} \to 0 
            \end{split}
        \end{equation*} 
        as $n\to \infty$, where $R>0$ is a radius (independent of $n$ and $x$) such that $z \mapsto \eta(z,x)$ is supported in $B_R$ and we used the continuity of translations of $L^1$ functions.

    The bounds~\eqref{eq:control on second derivatives} and~\eqref{eq:discretization error in derivative} already proven in Proposition~\ref{prop:discrete MM} let us deduce that 
        \begin{equation*}
            \|\de_1 \varphi^{\en} - \de_1^\mathrm{d} \varphi_n\|_{L^1(\Omega)} \leq C \ln \to 0 \, , \quad \text{as } n \to \infty \, .
        \end{equation*} 

        Hence, to conclude we need to show that $\|w_n - \de_1^\mathrm{d} \varphi_n \|_{L^1(\Omega)} \to 0$. This is a consequence of~\eqref{eq:discrete d of phi bounded} and of the fact that the sequence of functions $s \mapsto \sqrt{\tfrac{2}{\dn}} \sin \big( \tfrac{1}{2}\arccos(1-\dn) s \big)$ converges locally uniformly to the identity $s \mapsto s$ as $n \to \infty$.
    \end{proof}

    At this stage of the proof, we start to compare the discrete second-order Modica-Mortola energy in~\eqref{eq:discrete Poliakovsky} with the energy $H_n(w_n,z_n;\Omega)$. For this comparison, it is convenient to introduce the $n$-dependent double-well potentials defined by 
    \begin{equation*}
        \tilde W_n(s)  := \Big( 1 - \frac{2}{\dn} \sin^2 \Big( \frac{\arccos(1-\dn)}{2} s \Big) \Big)^2  .
    \end{equation*}
    These potentials are defined in such a way that 
    \begin{equation} \label{eq:W on w and z}
        W(w_n^{i,j}) = \tilde W_n(\de_1^{\mathrm{d}} \varphi_n^{i,j}) 
        \quad \text{and} \quad 
        W(z_n^{i,j}) = \tilde W_n(\de_2^{\mathrm{d}} \varphi_n^{i,j}) \, ,
    \end{equation}
    where $W(s) = (1-s^2)^2$. Note that 
\begin{equation} \label{eq:bound from below of W}
    W(s) \geq \tilde W_n(s) \quad \text{ for every } s \in \RR \, .
\end{equation}
This can be proven, e.g., through the inequality $W\big( \sqrt{\tfrac{2}{\dn}} \sin t \big) \leq W\big( \tfrac{2}{\arccos(1-\dn)} t\big)$, $t \in \RR$.

\begin{proposition}
    Let $\varphi_n$ as in~\eqref{eq:recovery phi} and let $(w_n,z_n)$ be defined by \eqref{recovery-spin} and \eqref{eq:def of wn zn}. Then we have that 
    \begin{equation} \label{eq:bound from below with H}
        \begin{aligned}
            \integral{\Omega}{\frac{1}{\en} W(\de_1^{\mathrm{d}} \varphi_n(x)) + \en |\de_{11}^{\mathrm{d}} \varphi_n(x)|^2 }{\d x}   & \geq (1-r_n) \Hh_n(w_n,z_n;\Omega) \,, \\
            \integral{\Omega}{\frac{1}{\en} W(\de_2^{\mathrm{d}} \varphi_n(x)) + \en |\de_{22}^{\mathrm{d}} \varphi_n(x)|^2 }{\d x}  &  \geq (1-r_n) \Hv_n(w_n,z_n;\Omega) \, ,
        \end{aligned}
    \end{equation}
    where $r_n \to 0$ as $n \to \infty$.
\end{proposition}
\begin{proof}
   We only prove the claim for $\Hh_n(w_n,z_n;\Omega)$, the inequality for $\Hv_n(w_n,z_n;\Omega)$ being analogous.

   We start by finding a relation between $\de^{\mathrm{d}}_1 w_n$ and $\de_{11}^{\mathrm{d}} \varphi_n$. Using the definition of $w_n$, cf.~\eqref{eq:order parameter} and~\eqref{eq:theta and discrete derivative}, we obtain that
   \begin{equation*}
    \begin{split}
        \big| \de^{\mathrm{d}}_1 w_n^{i,j} \big| & = \Big| \frac{w_n^{i+1,j} - w_n^{i,j}}{\ln} \Big| = \frac{\sqrt{2}}{ \sqrt{\dn} \ln } \Big| \sin \Big( \frac{1}{2} (\theth_n)^{i+1,j} \Big) - \sin \Big( \frac{1}{2} (\theth_n)^{i,j} \Big) \Big| \\
        &\leq \frac{\arccos(1-\dn)}{ \sqrt{2 \dn} \ln } | \de^\mathrm{d}_1 \varphi_n^{i+1,j} - \de^\mathrm{d}_1 \varphi_n^{i,j}  | = \frac{\arccos(1-\dn)}{ \sqrt{2 \dn}} |\de^\mathrm{d}_{11} \varphi_n^{i,j}| \, ,
    \end{split}
   \end{equation*}

   We recall that $(\rho_n^\mathrm{hor})^{i,j} = \rho((\theth_n)^{i,j},(\theth_n)^{i+1,j})$, as in Lemma~\ref{lemma:H is a discrete MM}. By Lemma~\ref{lemma:rho is close to 1} and~\eqref{eq:theta goes to 0}, we infer that $(\rho^{\mathrm{hor}}_n)^{i,j} \to 1$ as $n \to \infty$, uniformly in $(i,j)$.  This yields
    \begin{equation} \label{eq:absurd trigonometric factor}
            \big|\de_{11}^{\mathrm{d}} \varphi_n^{i,j} \big|^2   \geq (1-r_n) (\rho^{\mathrm{hor}}_n)^{i,j} \big|\de_{1}^{\mathrm{d}} w_n^{i,j}\big|^2 ,
    \end{equation}
    for some sequence $r_n \to 0$. Note that $(1-r_n) (\rho^{\mathrm{hor}}_n)^{i,j} \geq 0$ and $(1-r_n) (\rho^{\mathrm{ver}}_n)^{i,j} \geq 0$  for~$n$ large enough.  
   
    By~\eqref{eq:bound from below of W}, \eqref{eq:W on w and z}, \eqref{eq:absurd trigonometric factor}, and Lemma~\ref{lemma:H is a discrete MM} we obtain that
    \begin{equation*}
        \begin{split}
            & \integral{\Omega}{\frac{1}{\en} W(\de_1^{\mathrm{d}} \varphi_n(x)) + \en |\de_{11}^{\mathrm{d}} \varphi_n(x)|^2 }{\d x} 
             \geq \integral{\Omega}{\frac{1}{\en} W(w_n(x)) + \en (1-r_n) \rho^\mathrm{hor}_n(x) \big|\de_{1}^{\mathrm{d}} w_n(x)\big|^2 }{\d x} \\
            & \quad \geq \frac{1}{2\en} \ln^2 \hspace{-1em} \sum_{(i,j) \in \I^n(\Omega)} \hspace{-1em}   W(w_n^{i,j})  + W(w_n^{i+1,j}) + \en \ln^2 \hspace{-1em}  \sum_{(i,j) \in \I^n(\Omega)} \hspace{-1em} (1-r_n) (\rho_n^\mathrm{hor})^{i,j} \big| \de^{\mathrm{d}}_{1} w_n^{i,j} \big|^2  \\
            & \quad \geq   (1-r_n) \Hh_n(w_n,z_n;\Omega) \, .             
        \end{split}
    \end{equation*}
    This concludes the proof.
\end{proof}

Thanks to Proposition~\ref{prop:convergence of w,z}, \eqref{eq:Poliakovsky applied}, \eqref{eq:discrete Poliakovsky}, \eqref{eq:bound from below with H}, and since $H_n = \Hh_n + \Hv_n$, we have proved~\eqref{eq:almost gamma limsup}. This concludes the proof of the  limsup  inequality, i.e., Theorem~\ref{thm:main}-(iii). 

\section{Appendix}
 
We recall here the definition of weakly Lipschitz sets and some of their properties.

\begin{definition} \label{def:weakly lipschitz}
	A bounded open set $\Omega \subset \RR^d$ is a {\em weakly Lipschitz set} if every $x \in \de \Omega$ has a closed neighborhood $U_x \subset \RR^d$ such that there exists a bi-Lipschitz map  $\Psi_x \colon U_x \to [-1,1]^d$ satisfying
	\begin{align*}
		\Psi_x(x) & = 0 \, ,\\
		\Psi_x(\Omega \cap U_x) & = [-1,1]^{d-1} \x [-1,0) \, ,\\
		\Psi_x(\de \Omega \cap U_x) & = [-1,1]^{d-1} \x \{ 0 \} \, \\
		\Psi_x(U_x \sm \ol \Omega) & = [-1,1]^{d-1} \x (0, 1] \, .
	\end{align*}  
\end{definition}

\begin{proposition} \label{prop:extension from weakly Lipschitz sets}
	Let $\Omega \subset \RR^d$ be a bounded, open, weakly Lipschitz set. Then, for every $p \in [1,\infty]$, $\Omega$ is an extension domain for $W^{1,p}$, i.e., there exists a linear and continuous operator $\mathcal{E} \colon W^{1,p}(\Omega) \to W^{1,p}(\RR^d)$ satisfying $\mathcal{E}(u)|_\Omega = u$ for every $u \in W^{1,p}(\Omega)$. 
\end{proposition}

\begin{proof}
Using classical arguments involving a partition of unity, the proof is reduced to the case where the support of $u$ lies inside a set of the form $\ol \Omega \cap \mathring U_x$, where $U_x$ is a closed neighborhood of a point $x \in \de \Omega$ with the properties in Definition~\ref{def:weakly lipschitz}. Let $\Psi_x$ be the bi-Lipschitz transformation given therein.

We define the function $v$ on $Q^- := (-1,1)^{d-1} \x (-1,0)$ by $v(x) = u(\Psi_x^{-1}(x))$. By~\cite[Theorem~11.51]{Leo} we have $v \in W^{1,p}(Q^-)$ and $\| v \|_{W^{1,p}(Q^-)} \leq C \| u \|_{W^{1,p}(\Omega)}$. 

We extend $v$ to a function defined almost everywhere on $Q := (-1,1)^d$ by reflection, i.e., we set
\begin{equation*}
	\tilde v (x) := 
	\begin{cases}
		v(x) & \text{if } x \in Q^-  , \\
		v(x_1, \dots , x_{d-1} , -x_d) & \text{if } x \in Q^+ := (-1,1)^{d-1} \x (0,1) \, .
	\end{cases}
\end{equation*}
Using again~\cite[Theorem~11.51]{Leo}, we have $\tilde v \in W^{1,p}(Q)$ with $\| \tilde v \|_{W^{1,p}(Q)} \leq C \| u \|_{W^{1,p}(\Omega)}$, since $\tilde v |_{Q^-}$ and $\tilde v |_{Q^+}$ have the same trace on $(-1,1)^{d-1} \x \{0\}$. Note moreover that $\tilde v$ is compactly supported in $Q$, since the support of $u$ is contained in~$\ol \Omega \cap \mathring U_x$.

Finally, let us set
\begin{equation*}
	\mathcal{E}(u) =
	\begin{cases}
		\tilde v \circ \Psi_x & \text{in } \mathring U_x \, , \\
		0 & \text{outside of } \mathring U_x \, .
	\end{cases}
\end{equation*}
Since $\tilde v$ has compact support in $Q$ and invoking once again~\cite[Theorem~11.51]{Leo}, it follows that $\mathcal{E}(u) \in W^{1,p}(\RR^d)$ and $\| \mathcal{E}(u) \|_{W^{1,p}(\RR^d)} \leq C \| \tilde v \|_{W^{1,p}(Q)} \leq C \| u \|_{W^{1,p}(\Omega)}$. We also clearly have $\mathcal{E}(u) |_\Omega = u$.
\end{proof}

\begin{remark} \label{rmk:rellich-kondrachov in weakly lipschitz}
	From the previous proposition it follows that the Poincar\'e-Wirtinger Inequality holds true in $W^{1,p}(\Omega)$ if $\Omega$ is a bounded, open, connected, weakly Lipschitz set. Indeed, this inequality holds true for connected extension domains, see, e.g., \cite[Theorem~12.23]{Leo} for the details. 
\end{remark}

\begin{proposition} \label{prop:bounded gradient and Lipschitz boundary}
    Let $\Omega \subset \RR^d$ be a bounded, open, weakly Lipschitz set, let $u \in L^\infty_{\mathrm{loc}}(\Omega)$ and assume that its distributional gradient $\nabla u$ belongs to $L^\infty(\Omega;\RR^d)$. Then $u \in L^\infty(\Omega)$. 
\end{proposition}
\begin{proof}

Thanks to a standard covering argument on the compact set $\overline \Omega$, it is enough to show that for every $x \in \de \Omega$ the function $u$ is bounded in $\Omega \cap U_x$, $U_x$ being some neighborhood of~$x$. 
	Given $x \in \de \Omega$, let $U_x$ and $\Psi_x$ be as in Definition~\ref{def:weakly lipschitz}. Let~$A$ be an open, convex set such that $A \subcc (-1,1)^{d-1} \x (-1,0)$, so that $\Psi_x^{-1}(A) \subcc \Omega \cap \mathring U_x$. Since $u \in W^{1,\infty}(\Psi_x^{-1}(A))$, by~\cite[Theorem~11.51]{Leo} we have $u \circ \Psi_x^{-1} \in W^{1,\infty} (A)$, and thus \cite[Proposition~2.13]{AmbFusPal} implies that $u \circ \Psi_x^{-1}$ is Lipschitz in $A$ with Lipschitz constant $\|\nabla (u \circ \Psi_x^{-1}) \|_{L^\infty}$ (independent of~$A$). By convexity of $(-1,1)^{d-1} \x (-1,0)$ we deduce that~$u \circ \Psi_x^{-1}$ is Lipschitz on the whole $(-1,1)^{d-1} \x (-1,0)$ and this implies, in particular, that~$u$ is bounded on~$\Omega \cap \mathring U_x$.
\end{proof}
	
\noindent {\bf Acknowledgments.}  The work of M.\ Cicalese was supported by the DFG Collaborative Research Center TRR 109, ``Discretization in Geometry and Dynamics''. G.\ Orlando has been supported by the Alexander von Humboldt Foundation.

\bigskip

\begin{thebibliography}{10}

    \bibitem{AliBraCic}
    {\sc R.~Alicandro, A.~Braides, and M.~Cicalese}, {\em Phase and anti-phase
      boundaries in binary discrete systems: a variational viewpoint}, Netw.
      Heterog. Media, 1 (2006), pp.~85--107.
    
    \bibitem{AliCic}
    {\sc R.~Alicandro and M.~Cicalese}, {\em Variational analysis of the
      asymptotics of the {$XY$} model}, Arch. Ration. Mech. Anal., 192 (2009),
      pp.~501--536.
    
    \bibitem{AliCicGlo}
    {\sc R.~Alicandro, M.~Cicalese, and A.~Gloria}, {\em Variational description of
      bulk energies for bounded and unbounded spin systems}, Nonlinearity, 21
      (2008), pp.~1881--1910.
    
    \bibitem{Pon}
    {\sc R.~Alicandro, M.~Cicalese, and M.~Ponsiglione}, {\em Variational
      equivalence between {G}inzburg-{L}andau, {$XY$} spin systems and screw
      dislocations energies}, Indiana Univ. Math. J., 60 (2011), pp.~171--208.
    
    \bibitem{AliCicPon}
    \leavevmode\vrule height 2pt depth -1.6pt width 23pt, {\em Variational
      equivalence between {G}inzburg-{L}andau, {$XY$} spin systems and screw
      dislocations energies}, Indiana Univ. Math. J., 60 (2011), pp.~171--208.
    
    \bibitem{AliDLGarPon}
    {\sc R.~Alicandro, L.~De~Luca, A.~Garroni, and M.~Ponsiglione}, {\em
      Metastability and dynamics of discrete topological singularities in two
      dimensions: a {$\Gamma$}-convergence approach}, Arch. Ration. Mech. Anal.,
      214 (2014), pp.~269--330.
    
    \bibitem{AliLazPal}
    {\sc R.~Alicandro, G.~Lazzaroni, and M.~Palombaro}, {\em On the effect of
      interactions beyond nearest neighbours on non-convex lattice systems}, Calc.
      Var. Partial Differential Equations, 56:42 (2017).
    
    \bibitem{AmbFusPal}
    {\sc L.~Ambrosio, N.~Fusco, and D.~Pallara}, {\em Functions of bounded
      variation and free discontinuity problems}, Oxford Mathematical Monographs,
      The Clarendon Press, Oxford University Press, New York, 2000.
    
    \bibitem{BadCicDLPon}
    {\sc R.~Badal, M.~Cicalese, L.~De~Luca, and M.~Ponsiglione}, {\em
      {$\Gamma$}-convergence analysis of a generalized {$XY$} model: fractional
      vortices and string defects}, Comm. Math. Phys., 358 (2018), pp.~705--739.
    
    \bibitem{BalJam}
    {\sc J.~M. Ball and R.~D. James}, {\em Fine phase mixtures as minimizers of
      energy}, Arch. Rational Mech. Anal., 100 (1987), pp.~13--52.
    
    \bibitem{Bra}
    {\sc A.~Braides}, {\em {$\Gamma$}-convergence for beginners}, vol.~22 of Oxford
      Lecture Series in Mathematics and its Applications, Oxford University Press,
      Oxford, 2002.
    
    \bibitem{BraCic}
    {\sc A.~Braides and M.~Cicalese}, {\em Interfaces, modulated phases and
      textures in lattice systems}, Arch. Ration. Mech. Anal., 223 (2017),
      pp.~977--1017.
    
    \bibitem{BraConGar}
    {\sc A.~Braides, S.~Conti, and A.~Garroni}, {\em Density of polyhedral
      partitions}, Calc. Var. Partial Differential Equations, 56:28 (2017).
    
    \bibitem{BraGarPal}
    {\sc A.~Braides, A.~Garroni, and M.~Palombaro}, {\em Interfacial energies of
      systems of chiral molecules}, Multiscale Model. Simul., 14 (2016),
      pp.~1037--1062.
    
    \bibitem{BraKre}
    {\sc A.~Braides and L.~Kreutz}, {\em Design of lattice surface energies}, Calc.
      Var. Partial Differential Equations, 57:97 (2018).
    
    \bibitem{BraTru}
    {\sc A.~Braides and L.~Truskinovsky}, {\em Asymptotic expansions by
      {$\Gamma$}-convergence}, Contin. Mech. Thermodyn., 20 (2008), pp.~21--62.
    
    \bibitem{Bra-Yip}
    {\sc A.~Braides and N.~K. Yip}, {\em A quantitative description of mesh
      dependence for the discretization of singularly perturbed nonconvex
      problems}, SIAM J. Numer. Anal., 50 (2012), pp.~1883--1898.
    
    \bibitem{CicOrlRuf}
    {\sc M.~Cicalese, G.~Orlando, and M.~Ruf}, {\em From the {$N$}-clock model to
      the {$XY$} model: emergence of concentration effects in the variational
      analysis}, Preprint,  (2019).
    
    \bibitem{CicSol}
    {\sc M.~Cicalese and F.~Solombrino}, {\em Frustrated ferromagnetic spin chains:
      a variational approach to chirality transitions}, J. Nonlinear Sci., 25
      (2015), pp.~291--313.
    
    \bibitem{ConFarMag}
    {\sc S.~Conti, D.~Faraco, and F.~Maggi}, {\em A new approach to counterexamples
      to {$L^1$} estimates: {K}orn's inequality, geometric rigidity, and regularity
      for gradients of separately convex functions}, Arch. Ration. Mech. Anal., 175
      (2005), pp.~287--300.
    
    \bibitem{Con-Fon-Leo}
    {\sc S.~Conti, I.~Fonseca, and G.~Leoni}, {\em A {$\Gamma$}-convergence result
      for the two-gradient theory of phase transitions}, Comm. Pure Appl. Math., 55
      (2002), pp.~857--936.
    
    \bibitem{ConSch}
    {\sc S.~Conti and B.~Schweizer}, {\em Rigidity and gamma convergence for
      solid-solid phase transitions with {SO}(2) invariance}, Comm. Pure Appl.
      Math., 59 (2006), pp.~830--868.
    
    \bibitem{Diep}
    {\sc H.~Diep et~al.}, {\em Frustrated spin systems}, World Scientific, 2013.
    
    \bibitem{DmiKri}
    {\sc D.~Dmitriev and V.~Krivnov}, {\em Universal low-temperature properties of
      frustrated classical spin chain near the ferromagnet-helimagnet transition
      point}, Eur. Phys. J. B, 82 (2011), pp.~123--131.
    
    \bibitem{Fed}
    {\sc H.~Federer}, {\em Geometric measure theory}, Die Grundlehren der
      mathematischen Wissenschaften, Band 153, Springer-Verlag New York Inc., New
      York, 1969.
    
    \bibitem{Fon-Leo-Par}
    {\sc I.~Fonseca, G.~Leoni, and R.~Paroni}, {\em On {H}essian matrices in the
      space {$BH$}}, Commun. Contemp. Math., 7 (2005), pp.~401--420.
    
    \bibitem{KitLucRul}
    {\sc G.~Kitavtsev, S.~Luckhaus, and A.~R\"{u}land}, {\em Surface energies
      emerging in a microscopic, two-dimensional two-well problem}, Proc. Roy. Soc.
      Edinburgh Sect. A, 147 (2017), pp.~1041--1089.
    
    \bibitem{Leo}
    {\sc G.~Leoni}, {\em A first course in {S}obolev spaces}, vol.~105 of Graduate
      Studies in Mathematics, American Mathematical Society, Providence, RI, 2009.
    
    \bibitem{Lic}
    {\sc M.~W. Licht}, {\em Smoothed projections over weakly {L}ipschitz domains},
      Math. Comp., 88 (2019), pp.~179--210.
    
    \bibitem{Lui-Vai}
    {\sc J.~Luukkainen and J.~V\"{a}is\"{a}l\"{a}}, {\em Elements of {L}ipschitz
      topology}, Ann. Acad. Sci. Fenn. Ser. A I Math., 3 (1977), pp.~85--122.
    
    \bibitem{Mod}
    {\sc L.~Modica}, {\em The gradient theory of phase transitions and the minimal
      interface criterion}, Arch. Rational Mech. Anal., 98 (1987), pp.~123--142.
    
    \bibitem{ModMor}
    {\sc L.~Modica and S.~Mortola}, {\em Il limite nella {$\Gamma $}-convergenza di
      una famiglia di funzionali ellittici}, Boll. Un. Mat. Ital. A (5), 14 (1977),
      pp.~526--529.
    
    \bibitem{Mos}
    {\sc R.~Moser}, {\em Structure and rigidity of functions in {$\rm
      BV^2_{loc}(\mathbb R^2)$} with gradients taking only three values}, Proc.
      Lond. Math. Soc. (3), 116 (2018), pp.~813--846.
    
    \bibitem{Orn}
    {\sc D.~Ornstein}, {\em A non-inequality for differential operators in the
      {$L_{1}$} norm}, Arch. Rational Mech. Anal., 11 (1962), pp.~40--49.
    
    \bibitem{Pol07}
    {\sc A.~Poliakovsky}, {\em Upper bounds for singular perturbation problems
      involving gradient fields}, J. Eur. Math. Soc. (JEMS), 9 (2007), pp.~1--43.
    
    \bibitem{Pol}
    \leavevmode\vrule height 2pt depth -1.6pt width 23pt, {\em A general technique
      to prove upper bounds for singular perturbation problems}, J. Anal. Math.,
      104 (2008), pp.~247--290.
    
    \bibitem{rastelli1979non}
    {\sc E.~Rastelli, A.~Tassi, and L.~Reatto}, {\em Non-simple magnetic order for
      simple hamiltonians}, Physica B+ C, 97 (1979), pp.~1--24.
    
    \bibitem{schoenherr2018topological}
    {\sc P.~Schoenherr, J.~M{\"u}ller, L.~K{\"o}hler, A.~Rosch, N.~Kanazawa,
      Y.~Tokura, M.~Garst, and D.~Meier}, {\em Topological domain walls in
      helimagnets}, Nature Physics, 14 (2018), pp.~465--468.
    
    \bibitem{SciVal}
    {\sc G.~Scilla and V.~Vallocchia}, {\em Chirality transitions in frustrated
      ferromagnetic spin chains: a link with the gradient theory of phase
      transitions}, J. Elasticity, 132 (2018), pp.~271--293.
    
    \bibitem{uchida2006real}
    {\sc M.~Uchida, Y.~Onose, Y.~Matsui, and Y.~Tokura}, {\em Real-space
      observation of helical spin order}, Science, 311 (2006), pp.~359--361.
    
    \end{thebibliography}

\end{document}